\documentclass{amsart}
\usepackage[margin=3.45cm]{geometry}

\usepackage{amssymb}
\usepackage{graphicx} 
\usepackage{mathrsfs}
\usepackage{enumerate}
\usepackage{xspace}
\usepackage{color}
\usepackage{enumerate}
\usepackage{enumitem}
\usepackage{amsthm}
\usepackage{dsfont}
\usepackage{bbm}
\usepackage[colorlinks=true, linkcolor = blue, citecolor = blue]{hyperref}
\usepackage[numbers]{natbib}
\usepackage[backgroundcolor=white,bordercolor=red]{todonotes}
\DeclareMathAlphabet{\mathpzc}{OT1}{pzc}{m}{it}
\usepackage[nameinlink]{cleveref}
\usepackage{longtable}
\numberwithin{equation}{section}
\begin{document}

\theoremstyle{plain}

\newtheorem{theorem}{Theorem}[section]
\newtheorem{lemma}[theorem]{Lemma}
\newtheorem{example}[theorem]{Example}
\newtheorem{proposition}[theorem]{Proposition}
\newtheorem{corollary}[theorem]{Corollary}
\newtheorem{definition}[theorem]{Definition}
\newtheorem{Ass}[theorem]{Assumption}
\newtheorem{condition}[theorem]{Condition}
\theoremstyle{definition}
\newtheorem{remark}[theorem]{Remark}
\newtheorem{SA}[theorem]{Standing Assumption}
\newtheorem*{discussion}{Discussion}
\newtheorem{remarks}[theorem]{Remark}
\newtheorem*{notation}{Remark on Notation}
\newtheorem*{SoI}{Sketch of Idea}

\newcommand{\of}{[\hspace{-0.06cm}[}
\newcommand{\gs}{]\hspace{-0.06cm}]}

\newcommand\llambda{{\mathchoice
		{\lambda\mkern-4.5mu{\raisebox{.4ex}{\scriptsize$\setminus$}}}
		{\lambda\mkern-4.83mu{\raisebox{.4ex}{\scriptsize$\setminus$}}}
		{\lambda\mkern-4.5mu{\raisebox{.2ex}{\footnotesize$\scriptscriptstyle\setminus$}}}
		{\lambda\mkern-5.0mu{\raisebox{.2ex}{\tiny$\scriptscriptstyle\setminus$}}}}}

\newcommand{\1}{\mathds{1}}

\newcommand{\F}{\mathbf{F}}
\newcommand{\G}{\mathbf{G}}

\newcommand{\B}{\mathbf{B}}

\newcommand{\M}{\mathcal{M}}

\newcommand{\la}{\langle}
\newcommand{\ra}{\rangle}

\newcommand{\lle}{\langle\hspace{-0.085cm}\langle}
\newcommand{\rre}{\rangle\hspace{-0.085cm}\rangle}
\newcommand{\blle}{\Big\langle\hspace{-0.155cm}\Big\langle}
\newcommand{\brre}{\Big\rangle\hspace{-0.155cm}\Big\rangle}


\newcommand{\tr}{\operatorname{tr}}
\newcommand{\N}{{\mathcal{N}}}
\newcommand{\cadlag}{c\`adl\`ag }
\newcommand{\on}{\operatorname}
\newcommand{\oP}{\overline{P}}
\newcommand{\oO}{\mathcal{O}}
\newcommand{\D}{D(\mathbb{R}_+; \mathbb{R})}
\newcommand{\bx}{\mathsf{x}}
\newcommand{\z}{\mathfrak{z}}
\newcommand{\bb}{\hat{b}}
\newcommand{\bs}{\hat{\sigma}}
\newcommand{\bv}{\hat{v}}
\renewcommand{\v}{\mathfrak{m}}
\newcommand{\ob}{\widehat{b}}
\newcommand{\os}{\widehat{\sigma}}
\newcommand{\us}{\bar{\sigma}}
\renewcommand{\j}{\varkappa}
\newcommand{\x}{\mathscr{X}}
\newcommand{\MM}{\mathscr{M}}
\newcommand{\blv}{\mathop{\vphantom{\int}}\nolimits_\V \hspace{-0.1cm}}
\newcommand{\bly}{\mathop{\vphantom{\int}}\nolimits_{\Y^*}\hspace{-0.1cm}}
\newcommand{\Z}{\mathscr{S}}
\newcommand{\tC}{\mathscr{C}^2_c}
\newcommand{\tT}{\mathscr{T}}
\newcommand{\test}{\zeta}

\renewcommand{\H}{\mathbb{H}}
\newcommand{\V}{\mathbb{V}}
\newcommand{\X}{\mathbb{X}}
\newcommand{\U}{\mathbb{U}}
\newcommand{\Y}{\mathbb{Y}}

\renewcommand{\epsilon}{\varepsilon}

\newcommand{\fPs}{\fP_{\textup{sem}}}
\newcommand{\fPas}{\fP^{\textup{ac}}_{\textup{sem}}}
\newcommand{\rrarrow}{\twoheadrightarrow}
\newcommand{\cA}{\mathcal{A}}
\newcommand{\ocA}{\mathcal{U}}
\newcommand{\cR}{\mathcal{R}}
\newcommand{\cK}{\mathcal{K}}
\newcommand{\cQ}{\mathcal{Q}}
\newcommand{\cF}{\mathcal{F}}
\newcommand{\cE}{\mathcal{E}}
\newcommand{\cC}{\mathcal{C}}
\newcommand{\cD}{\mathcal{D}}
\newcommand{\bC}{\mathbb{C}}
\newcommand{\cH}{\mathcal{H}}
\newcommand{\bth}{\overset{\leftarrow}\theta}
\renewcommand{\th}{\theta}
\newcommand{\cG}{\mathcal{G}}
\newcommand{\bbB}{\mathbb{B}}
\newcommand{\oX}{X} 

\newcommand{\bR}{\mathbb{R}}
\newcommand{\nnabla}{\nabla}
\newcommand{\f}{\xi} 
\newcommand{\g}{\mathfrak{g}}
\newcommand{\oconv}{\overline{\on{co}}\hspace{0.075cm}}
\renewcommand{\a}{\mathfrak{a}}
\renewcommand{\b}{\mathfrak{b}}
\renewcommand{\d}{m_\Omega} 
\renewcommand{\r}{m_\m}
\newcommand{\bS}{\mathbb{S}^d_+}
\newcommand{\p}{\mathsf{p}} 
\newcommand{\dr}{r} 
\newcommand{\m}{\mathbb{M}}
\newcommand{\Q}{Q}
\newcommand{\n}{\overline{\nu}} 
\newcommand{\usc}{\textit{USC}}
\newcommand{\lsc}{\textit{LSC}}
\newcommand{\q}{\mathfrak{q}}
\newcommand{\W}{\mathcal{P}}
\newcommand{\fP}{\mathcal{P}}
\newcommand{\w}{\mathsf{w}}
\newcommand{\oM}{\mathsf{M}}
\newcommand{\oZ}{\mathsf{Z}}
\newcommand{\oK}{\mathsf{K}}
\renewcommand{\o}{{q}}
\newcommand{\wV}{V_\textup{weak}}
\newcommand{\I}{\varrho}
\newcommand{\lV}{\langle}
\newcommand{\rV}{\rangle_V}
\renewcommand{\B}{\mathscr{B}}
\newcommand{\qq}{\mathsf{q}}
\newcommand{\al}{\varrho}
\newcommand{\gG}{\mathsf{G}}
\newcommand{\J}{\mathscr{J}}
\newcommand{\NN}{\mathscr{G}}
\renewcommand{\Upsilon}{\overline{\Theta}}

\renewcommand{\emptyset}{\varnothing}

\allowdisplaybreaks

\makeatletter
\@namedef{subjclassname@2020}{%
	\textup{2020} Mathematics Subject Classification}
\makeatother

 \title[Limit theory for controlled McKean--Vlasov SPDEs]{A limit theory for controlled \\McKean--Vlasov SPDE{\small s}} 
\author[D. Criens]{David Criens}
\address{Albert-Ludwigs University of Freiburg, Ernst-Zermelo-Str. 1, 79104 Freiburg, Germany}
\email{david.criens@stochastik.uni-freiburg.de}



\thanks{The author is very grateful to the referee for many helpful remarks and suggestions.}
\date{\today}

\maketitle

\begin{abstract}
We develop a limit theory for controlled mean field stochastic partial differential equations in a variational framework. More precisely, we prove existence results for mean field limits and particle approximations, and we establish a set-valued propagation of chaos result which shows that sets of empirical distributions converge to sets of mean field limits in the Hausdorff metric topology. 
Further, we discuss limit theorems related to stochastic optimal control theory. 
To illustrate our findings, we apply them to a controlled interacting particle system of stochastic porous media equations. 
\end{abstract}

\section{Introduction}
	The area of controlled McKean--Vlasov dynamics, also known as mean field control, has rapidly developed in the past years, see, e.g., the monograph \cite{car_della} and the references therein.
	There is also increasing interest in infinite dimensional systems such as controlled McKean--Vlasov stochastic PDEs (mean field SPDEs).
	For controlled mean field SPDEs within the semigroup approach of Da Prato and Zabczyk \cite{DaPratoEd1}, well-posedness of the state equation, the dynamic programming principle and a Bellman equation were proved in the recent paper \cite{CGKPR23}, where also comments on related literature can be found.

	Mean field dynamics are typically motivated by particle approximations (related to propagation of chaos), see, for instance, Sznitman's seminal monograph \cite{SnzPoC}. It is an important task to make this motivation rigorous.
	For finite dimensional controlled systems, a general limit theory was developed in the paper \cite{LakSIAM17} and extended in \cite{DPT22} to a setup with common noise.  
	The purpose of the present paper is to establish a limit theory for controlled interacting infinite dimensional SPDEs in the variational framework initiated by Pardoux \cite{par75} and Krylov--Rozovskii \cite{krylov_rozovskii}, see also the more recent monographs \cite{liu_rockner, par}. 
	
	To explain our contribution in more detail, consider a particle system \(Y = (Y^1, \dots, Y^n)\) of the form 
	\begin{align*}
		d Y^k_t &= \int b (f, t, Y^k_t, \x_n (Y_t)) \, \v^k(t, df) \, dt + \us (\v^k (t, \cdot\,), t, Y^k_t, \x_n (Y_t)) \, d W^k_t, 
	\end{align*}
with initial values \(Y^k_0 = x\), where 
	\begin{align*}
		\x_n (Y_t) &= \frac{1}{n} \sum_{i = 1}^n \delta_{Y^i_t}
	\end{align*}
denotes the empirical distribution of the particles, and \(\us\) is such that
\[
\us \us^* (\v^k (t, \cdot\,), t, Y^k_t, \x_n (Y_t)) = \int \sigma \sigma^* (f, t, Y^k_t, \x_n (Y_t)) \, \v^k (t, df)
\]
for a volatility coefficient \(\sigma\).
Here, \(\v^1, \dots, \v^n\) are probability kernels that model the control variables, and \(W^1, \dots, W^n\) are independent cylindrical Brownian motions. This corresponds to a relaxed control framework in the spirit of \cite{nicole1987compactification,ElKa15}. Let \(\cR^n(x)\) be the set of joint empirical distributions of particles and controls (latter are captured via \(\v^k (t, df) \, dt\) in a suitable space of Radon measures). 
The associated set of mean field limits is denoted by \(\cR^0 (x)\). It consists of  probability measures supported on the set of laws of \((Y, \v (t, df) \, dt)\), where \(Y\) solves a controlled McKean--Vlasov equation of the form
\[
d Y_t = \int b (f, t, Y_t, P^Y_t) \, \v (t, df) \, dt + \us( \v (t, \cdot\,), t, Y_t, P^Y_t) \, d W_t, \quad Y_0 = x, 
\]
with \(P^Y_t = \mathcal{L} \textit{aw}\, (Y_t)\). 

Conceptually, our main results are divided into two groups. The first one is probabilistic and deals with the convergence of controlled particle systems, while the second one sheds light on mean field limits from the perspective of stochastic optimal control.

On the probabilistic side, we show that \(\cR^n (x)\) and \(\cR^0 (x)\) are nonempty and compact in a suitable Wasserstein space, and that the maps \(x \mapsto \cR^n (x)\) converge to \(x \mapsto \cR^0 (x)\) uniformly on compacts in the Hausdorff metric topology. This result can be seen as set-valued propagation of chaos. Indeed, when \(\cR^n (x)\) and \(\cR^0 (x)\) are singletons, we recover a classical formulation of the propagation of chaos property. To the best of our knowledge, the concept and formulation of set-valued propagation of chaos is new. 
The observation \(\cR^0 (x) \not = \emptyset\) includes an existence result for (uncontrolled) McKean--Vlasov SPDEs in a variational framework, comparable to a recent result from \cite{HLL23}. 
Our proof for \(\cR^0 (x) \not = \emptyset\) is based on a particle approximation. In the absence of controls, this can be compared to a recent propagation of chaos result from \cite{HLL23} for the stochastic 2D Navier--Stokes equation.

As second main contribution, we investigate approximation properties of optimal control problems. Namely, if \(\psi\) is a continuous input function of suitable growth, we prove that the value functions associated to the particle approximation, i.e., 
\[
x \mapsto \sup_{Q \in \cR^n (x) } E^Q \big[ \psi \big]
\]
converge uniformly on compacts to the value function of the mean field limit, i.e., 
\[
x \mapsto \sup_{Q \in \cR^0 (x)} E^Q \big[ \psi \big].
\]
 We also derive ramifications of this statement for upper and lower semicontinuous input functions~\(\psi\) of suitable growth. These results allow us to deduce limit theorems in the spirit of the seminal work \cite{LakSIAM17}. Namely, we show that all accumulation points of sequences of \(n\)-state nearly optimal controls maximize the mean field value function, and that any optimal mean field control can be approximated by a sequence of \(n\)-state nearly optimal controls. 

	Let us now discuss our assumptions on the coefficients \(b\) and \(\sigma\). The main condition has two layers, each formulated with the help of a separate Gelfand triple. On the first layer, we impose continuity, coercivity and growth assumptions that can be compared to mild existence conditions from \cite{GRZ}. These assumptions entail that \(\cR^n (x)\) and \(\cR^0 (x)\) are nonempty and compact, as well as the particle approximation. As explained in \cite{GRZ}, such conditions can be verified for stochastic porous media and Navier--Stokes equations. 
	
	The second Gelfand triple is used to formulate growth and weak monotonicity conditions, which we employ to establish the set-valued propagation of chaos and the limit theory for stochastic control problems. These assumptions can be verified for stochastic porous media equations, but they fail for stochastic Navier--Stokes equations that lack weak monotonicity properties (see \cite{liu_rockner}). It is an interesting open problem whether the weak monotonicity conditions in our work can be replaced by local monotonicity conditions (\cite{liu_r:JFA}). Such a strengthening appears to be challenging due to the non-local structure of McKean--Vlasov equations and we leave it for future research.
	
	To illustrate our theory, we verify our assumptions for a slow diffusion framework where the controlled particle system is of the form
	\[
	d Y^k_t = \Big[ \Delta ( |Y^k_t|^{q - 2} Y^k_t) + \frac{1}{n} \sum_{i = 1}^n\, (Y^k_t - Y^i_t)+ \int c (f)\, \v^k (t, df) \Big] \, dt + \sigma \, d W^k_t,
	\]
	with \(q \geq 2\) and \(k = 1, \dots, n\). It is worth mentioning that our results apply to the classical finite dimensional framework. 

We now comment on related literature and proofs. 
Our work is heavily inspired by the paper~\cite{LakSIAM17}, which develops a limit theory for mean field control in a finite dimensional framework.
Similar to \cite{LakSIAM17}, parts of our proofs rely on compactification and martingale problem methods that were established in~\cite{nicole1987compactification}. In contrast to \cite{nicole1987compactification,LakSIAM17}, we deal with an infinite dimensional setting that is technically different. 
Let us also mention that our work covers some new finite dimensional cases.
In the following, we highlight some important points that lead to technical difficulties.
Due to multiple state spaces (recall that we work with two Gelfand triples), we have to keep track of several topologies that influence convergence and continuity properties. 
Under our assumptions on the coefficients, we have no a priori moment estimates, which we therefore have to build into our setting. Furthermore, we avoid conditions for strong existence and uniqueness of equations with random coefficients. To overcome this obstacle, we prove new weak existence results that keep track of the driving noise, adapting a method from \cite{jacod1981weak}, and we apply a change of topology, where we work with the second Gelfand triple. Finally, compared to the finite dimensional case, we also require different relative compactness methods, which are influenced by~\cite{GRZ,LakSPA15}.
	
	The paper is structured as follows. Our framework and the main results are explained in Section~\ref{sec: main1}. 
	The slow diffusion example is discussed in Section~\ref{sec: ex} and the proof of our main theorem is given in Section~\ref{sec: pf}. To improve the readability of this paper, we added Appendix~\ref{sec: appendix}, which collects auxiliary results from the literature that are used in our proofs.
	
	\smallskip 
\noindent
{\bf Remarks on Notation.}
	 In this paper, \(C\) denotes a generic positive constant that might change from line to line. In case the constant depends on important quantities, this is mentioned specifically. 
For reader's convenience, let us also provide an overview on notations used in the main body of the paper. 

\smallskip 
\noindent
{\bf State spaces:}
\smallskip

	\begin{tabular}{ll}
	\(\Y \subset \H \subset \X \subset \V\)  & \(\H, \X, \V\) separable Hilbert spaces, \(\Y\) separable reflexive Banach space.
	\\
	\(\V\) & State equations have paths in \(C ([0, T]; \V)\).
		\\
	\(\Y\) & State equations have paths in \(L^\alpha ([0, T]; \Y)\).
	\\
	\(\H\) & State equations have moments associated with \(\| \cdot \|_\H\).
	\\
		\(\X\) & Auxiliary space used in Condition~\ref{cond: main2}.
	\\
		\(\mathbb{U}\)   &Separable Hilbert space; state space for the driving noise. 
	\\
	\(F\)  &Compact Polish space; action space for the control variable. 
	\end{tabular}
	
		\smallskip
		\newpage
		
	\noindent
	{\bf Coefficients, constants and auxiliary functions:}

	\smallskip 
	
	\noindent
	\begin{tabular}{ll}
		\(T > 0\) & Finite time horizon.
		\\
		$b,\sigma$ & Drift and diffusion coefficients of the state equations; see \eqref{eq: def coefficients}.
		\\
		\(\lambda, \alpha, \gamma, \beta, \eta, \al \in \bR_+\) & Constants satisfying the constraints \eqref{eq: cond constants}.
		\\
		\(\lambda\) & Growth constant from Condition~\ref{cond: main1}.
		\\
		\(\alpha\) & Level of integrability in the path space \(\Omega\); see \eqref{eq: Omega def}. 
		\\
		\(\gamma\) & Power constant in Condition~\ref{cond: main1}.
		\\
		\(\beta\) & Power constant in Conditions~\ref{cond: main1} and \ref{cond: main2}.
		\\
		\(\eta\) & Level of integrability of state equations; see \eqref{eq: def G}. 
		\\
		\(\al\) & Power of Wasserstein space in main result; see Theorem~\ref{theo: main1}.
		\\
		$\mathcal{N}, \mathcal{N}_p$ & Auxiliary growth function; see \eqref{eq: Np}. 
		\\
		\(\varkappa\) & Moment control function; see \eqref{eq: kappa}.
		\\
		\(\mathcal{L}_{g, v}\) & Generator mean field equations; see \eqref{eq: def gen MF}.
		\\
		\(\mathcal{L}^i_{g, v^1, \dots, v^n}\) & Generator \(i\)-th state equation; see \eqref{eq: def gen SE}.
	\end{tabular}
	
	\smallskip
	
	\noindent
	{\bf Path/Wasserstein spaces:}
	
	\smallskip 
	
	\noindent
	\begin{tabular}{ll}
	$(\Omega, \cF)$  &Canonical path space for the state process with its Borel \(\sigma\)-field; see \eqref{eq: Omega def}. 
	\\
	$(\mathbb{M}, \mathcal{M})$ & Space of Borel measures on \([0, T] \times F\) with canonical \(\sigma\)-field.
	\\
	$(\Theta, \mathcal{O})$ & \(= (\Omega \times \mathbb{M}, \cF \otimes \mathcal{M})\); joint path space for state variable and control process.
	\\
	$\mathcal{P}(E)$  & Borel probability measures on a metric space $E$. 
	\\
	\(\mathcal{P}^r (E)\)& Probability measures on $E$ with finite $r$-th moment. 
	\\
 \(\fP^{2 \eta}_{\al} (\H)\) & \(\fP^{2 \eta} (\H)\) endowed with the topology induced by \(\w^{\V}_{\al}\).
	\end{tabular}
	
		\smallskip
		
	\noindent
	{\bf Metrics:}
	
	\smallskip 
	
	\noindent
	\begin{tabular}{ll}
			$\mathsf{w}_E^r$  & $r$-Wasserstein metric on $\mathcal{P}^r(E)$; see \eqref{eq: WM}.
		\\
		$\mathsf{h}$  & Hausdorff metric on compact subsets of \(\mathcal{P}^\varrho (\mathcal{P}^\varrho (\Theta))\); see \eqref{eq: def HDM}.
		\\
		\(\| \mu \|_{r, E}\) & \(r\)-th moment of measure \(\mu \in \fP (E)\); see \eqref{eq: def moment function}.
	\end{tabular} 
	
		\smallskip
		
	\noindent
	{\bf Canonical maps:}
	
	\smallskip 
	
	\noindent
	\begin{tabular}{ll}
	\(X\) & Canonical process on \(\Omega\); \(X_t (\omega) = \omega (t)\).
	\\
	\(M, \mathfrak{m}\) & Canonical map and kernel on \(\mathbb{M}\); \(M (dt, df) = \mathfrak{m} (t, M, df) \, dt\).
	\\
	\(\Z_n\) & Empirical measure on the product space \(\Theta^n\); see \eqref{eq: def emp theta}.
	\\
	\(\x_n\) & Empirical measure on the product space \(\Omega^n\); see \eqref{eq: def emp omega^n}. 
	\end{tabular}

	\smallskip
	
	\noindent
	{\bf Control rules:}
	
	\smallskip 
	
	\noindent
	\begin{tabular}{ll}
	$\cC^n(x)$  & Laws of controlled $n$-particle system; see Definition~\ref{def: C^n}.
	\\
	$\cC^0(x)$  & Laws of controlled mean field system; see Definition~\ref{def: C MK}.
	\\
	$\cR^n(x)$ & Joint empirical distributions of particles and controls; see \eqref{eq: def R^n}.
	\\
	$\cR^0(x)$  & \(= \mathcal{P} (\cC^0 (x)) \cap \J(x)\).
		\\
	\(\mathscr{G}\) & Moment control for mean field equations; see \eqref{eq: def G}.
	\\ 
	\(\J (x)\) & Moment control for state equations; see \eqref{eq: def J}.
	\end{tabular}

\section{Limit Theory for Controlled SPDEs in a Variational Framework} \label{sec: main1}

	\subsection{Motivation}
	The objective of this paper is to establish the convergence of interacting controlled infinite dimensional systems to their mean-field limit. To motivate the structure of our mathematical results, we first provide motivation and intuition for the finite-dimensional setting.

	Consider a system with \(n\) players whose state processes interact through their empirical distributions given by the system 
	\[
	d X_t^k = b (X^k_t, u_t^k, \x^n_t ) \, dt + \sigma (X^k_t, u_t^k, \x^n_t ) \, dW^k_t, \quad X^k_0 = x, \quad \x^n_t = \frac{1}{n} \sum_{k = 1}^n \delta_{X_t^k},
	\] 
	where \(W^1, \dots, W^n\) are independent Brownian motions and \(u^1, \dots, u^n\) are controls selected by the central planner. The planner seeks to maximize an expected payoff functional, typically of the form 
	\[
	\frac{1}{n} \sum_{k = 1}^n E \Big[ \int_0^T g (X^k_s, u^k_s, \x^n_s) \, ds + f (X^k_T, \x^n_T) \Big]
	\] 
	with a running term \(g\) and a terminal term \(f\). The coefficients \(b\) and \(\sigma\), as well as the payoffs \(g\) and \(f\), are the same for all agents. While such systems provide natural models for large-scale interactions, they tend to become computationally intractable when the number \(n\) of agents is large. Consequently, we seek a simplified limiting structure.
	
	In the \textit{uncontrolled} case, the theory of propagation of chaos (\cite{SnzPoC}) establishes that the weakly interacting equations
	\[
	d Y^k_t = b (Y^k_t, \mathscr{Y}^n_t) \, dt + \sigma (Y^k_t, \mathscr{Y}^n_t ) \, d W^k_t, \quad Y^k_0 = y, \quad \mathscr{Y}^n_t = \frac{1}{n} \sum_{k = 1}^n \delta_{Y^k_t}, 
	\] 
	become asymptotically i.i.d., converging to the McKean--Vlasov limit
	\[
	d Y_t = b (Y_t, P^Y_t) \, dt + \sigma (Y_t, P^Y_t ) \, d W_t, \quad Y_0 = y, \quad P^Y_t = \text{Law of \(Y_t\)}.
	\] 
	Under suitable continuity assumptions on the coefficients and payoff functions, the value function of this system converges to the value function of the limit system. Intuitively, this holds because a single agent exerts negligible influence on the empirical average, and the symmetry ensures exchangeability.

The {\it controlled} case is substantially more delicate. Informally, one immediate reason is that limiting and optimization operations do not commute in general. In addition, one must track not only the state equations but also the control processes. We emphasize that, for general controls, the resulting controlled system need not be exchangeable, and small perturbations in the empirical distribution may lead to large changes in the control. These features make the analysis of optimal control problems more subtle than classical propagation of chaos settings.
	
	To address the convergence of the optimal control problems, we adopt an abstract set-valued formulation. We define the set of admissible laws for the \(n\)-particle system as
	\begin{align*}
	K^n = \Big\{ P \circ \Big( \frac{1}{n} \sum_{k = 1}^n \delta_{X^k} \Big)^{-1} \colon  \, &d X_t^k = b(X^k_t, u_t^k, \x^n_t) \, dt + \sigma (X^k_t, u_t^k, \x^n_t) \, dW^k_t, \\& X^k_0 = x, \ u^1, \dots, u^n \text{ control processes} \Big\}
	\end{align*}
	and its mean field limits as
	\begin{align*}
	K^0 = \Big\{ &\text{probability measures concentrated on the set of laws of} 
	\\&\ d X_t =b(X_t, u_t, P^X_t) \, dt + \sigma (X_t,u_t, P^X_t) \, d W_t, \ X_0 = x, \ u \text{ control process} \Big\}.
	\end{align*} 
	From this perspective, our goal is to verify the convergence of the value functions:
	\begin{align} \label{eq: motivating problem}
	\sup_{P \in K^n} E^P \big[ \text{payoff functional} \big] \to \sup_{P \in K^0} E^P \big[ \text{payoff functional} \big].
	\end{align} 
This formulation motivates studying the convergence of control problems using the Hausdorff metric topology. It is a well-known result in set-valued analysis (see, e.g., \cite[Lemma~2.15]{C_SIAM}) that the value map
	\[
	\mathcal{K} (E) \ni K \mapsto \sup_K \varphi
	\]
	is continuous on the space of compact subsets \(\mathcal{K} (E)\) endowed with the Hausdorff metric, provided the payoff \(\varphi\) is continuous. Thus, \eqref{eq: motivating problem} follows if \(K^0, K^1, \dots\) are compact and \(K^n \to K^0\) in the Hausdorff metric.
	
	This notion of convergence generalizes classical propagation of chaos to the controlled setting. Indeed, in the uncontrolled case, the sets \(K^n\) and \(K^0\) are singletons, and Hausdorff convergence reduces to the convergence of the unique elements. Moreover, Hausdorff convergence captures stability properties beyond the convergence of the value functions. For instance, the continuity of maps such as
	\begin{align*}
	&K \mapsto \sup_{P, Q \in K} \big| E^P [ \varphi ] - E^Q [ \varphi ] \big|, \quad
	K \mapsto \sup_{P \in K} \Big( E^P [ \varphi^2 ] - E^P [ \varphi ]^2 \Big) = \sup_{P \in K} \on{Var}^P [ \varphi ], 
	\end{align*}
	implies that measures of uncertainty and model sensitivity are preserved in the limit.
	
	In the remainder of this section, we outline how the Hausdorff convergence \(K^n \to K^0\) is established, which motivates the structure of our main result (Theorem~\ref{theo: main1}).
Under suitable compactness assumptions, this convergence is characterized by the following two properties:
	\begin{enumerate}
		\item[(a)] Every sequence \((\mu^n)_{n = 1}^\infty\) with \(\mu^n \in K^n\) is relatively compact and all of its accumulation points are in \(K^0\).
		\item[(b)] For every \(\mu^0 \in K^0\) there exists a sequence \((\mu^n)_{n = 1}^\infty\) such that \(\mu^n \in K^n\) and \(\mu^n \to \mu^0\). 
	\end{enumerate}
	Property (a) guarantees that the set of admissible laws is closed under limits, ensuring that accumulation points of \(n\)-state controls remain feasible in the mean-field limit. Property (b) ensures that every mean-field admissible law can be approximated by a sequence of \(n\)-state controls. As shown in \cite{LakSIAM17}, these properties imply that limits of \(n\)-state nearly optimal controls are optimal for the mean-field system, and conversely, every mean-field optimal control can be approximated by \(n\)-state nearly optimal controls. 
	
Our program investigates (a) and (b) in the context of interacting controlled SPDEs. Compared to the finite-dimensional case, the infinite-dimensional setting introduces significant analytical challenges. The state equations are posed in a variational framework involving distinct function spaces, requiring convergence to be tracked across multiple topologies. Furthermore, standard moment estimates and strong-solution arguments are unavailable under our assumptions. Consequently, we develop specialized tightness methods, martingale problem techniques, and weak existence results to establish the required compactness and convergence properties.

\subsection{The Setting and Notation} \label{sec: notation} 
Let \(\H, \X\) and \(\V\) be separable Hilbert spaces and let \(\Y\) be a separable reflexive Banach space, whose topological duals are denoted by \(\H^*, \X^*, \V^*\) and \(\Y^*\), respectively.
We assume that 
\[
\Y \subset \H \subset \X \subset \V
\]
continuously and densely. 
For any Banach space \(\mathbb{E}\), we denote its norm by \(\|\cdot\|_{\mathbb{E}}\) and the dualization between \(\mathbb{E}\) and \(\mathbb{E}^*\) by \(_\mathbb{E}\langle \, \cdot, \cdot\, \rangle_{\mathbb{E}^*}\), i.e., 
\[
_\mathbb{E} \langle e, e^* \rangle_{\mathbb{E}^*} := e^* (e), \quad e^* \in \mathbb{E}^*, e \in \mathbb{E}.
\]
In general, we endow all Banach spaces with their norm topologies. 

Below, we will work with two Gelfand triples
\[
\V^* \subset \H^* \cong\, \H \subset \V, \quad \Y \subset \X \cong\, \X^* \subset \Y^*.
\]
In particular, for \(x \in \H, y \in \V^*\) and \(z \in \X, v \in \Y\), we have 
\[
{_\V} \langle x, y \rangle_{\V^*} = \langle x, y \rangle_\H, \qquad 
{_\Y} \langle v, z \rangle_{\Y^*} = \langle v, z \rangle_\X, 
\]
where \(\langle \cdot, \cdot\rangle_\H\) and \(\langle \cdot, \cdot \rangle_\X\) are the scalar products of \(\H\) and \(\X\), respectively. For an example of how the spaces can be chosen in specific situations, the reader is referred to Section~\ref{sec: ex} below.

Throughout this paper, we extend scalar functions from a smaller to a larger Banach space by setting them \(\infty\) outside its original domain. For instance, we set \(\|x\|_\Y := \infty\) for \(x \in \V \setminus \Y\). In this way, \(\|\cdot\|_\Y\) becomes a lower semicontinuous map from \(\V\) into \([0, \infty]\), see \cite[Excercise~4.2.3]{liu_rockner}. 

Before we proceed, let us give an overview on all parametric constants that we use in this paper:
\begin{align*}
	T, \lambda, \alpha, \gamma, \beta, \eta, \al \in \bR_+,
\end{align*}
such that
\begin{align} \label{eq: cond constants}
	T, \lambda > 0, \quad \alpha, \gamma > 1, \quad \beta \geq 2 \vee \gamma, \quad \eta \geq \frac{\beta}{2} \vee \frac{\alpha}{2}, \quad \eta > 2, \quad \alpha > \al \geq 1.
\end{align}
Define the space
\begin{align} \label{eq: Omega def}
\Omega := \Big\{ \omega \in C ([0, T]; \V) \colon \int_0^T \|\omega(s)\|_\Y^\alpha \, ds < \infty \Big\}
\end{align} 
and endow it with the topology induced by the metric 
\[
\d (\omega, \omega') := \sup_{s \in [0, T]} \|\omega(s) - \omega' (s)\|_{\V} + \Big( \int_0^T \| \omega (s) - \omega' (s)\|^\alpha_\Y \, ds \Big)^{1/ \alpha}, \ \ \omega, \omega' \in \Omega,
\]
which turns it into a Polish space. Further, set \(\cF := \mathcal{B}(\Omega)\) and 
let  \(X = (X_t)_{t \in [0, T]}\)  be the coordinate map on \(\Omega\), i.e., \(X_t (\omega) = \omega (t)\) for \(\omega \in \Omega\) and \(t \in [0, T]\). The corresponding filtration \(\mathbf{F} = (\cF_t)_{t \in [0, T]}\) is defined by \(\cF_t := \sigma (X_s, s \in [0, t])\).

Take another separable Hilbert space~\(\U\), which we use as state space for the randomness that drives our systems. The space of Hilbert--Schmidt operators from \(\U\) into \(\H\) is denoted by \(L_2 (\U; \H)\) and the Hilbert--Schmidt norm is denoted by \(\|\cdot\|_{L_2 (\U; \H)}\). 

For any metric space \((E, m_E)\), let \(\fP(E)\) be the set of Borel probability measures on \(E\) and endow this space with the weak topology, i.e., the topology of convergence in distribution. It is well-known that \(\fP (E)\) is a Polish space once \(E\) has the same property, see Lemma~\ref{appendix: Suslin prob meas} in Appendix~\ref{sec: appendix}.
For \(r \geq 1\) and an arbitrary reference point \(e_0 \in E\), we set (with abuse of notation)
\begin{align} \label{eq: def moment function}
 \|\mu\|_{r, E} := \Big( \int m_E (e, e_0)^r\, \mu (d e) \Big)^{1/r}, \quad \mu \in \fP(E),
\end{align} 
and
\[
\fP^r (E) := \big\{ \mu \in \fP (E) \colon \|\mu\|_{r, E} < \infty \big\}.
\]
Further, let \(\w^E_r\) be the \(r\)-Wasserstein metric on \(\fP^r (E)\), i.e., 
\begin{align} \label{eq: WM}
\w^E_r (\mu, \nu) := \Big( \inf_{\pi \in \Pi (\mu, \nu)} \int m _E(x, y)^r \, \pi (dx, dy) \Big)^{1/r}, \quad \mu, \nu \in \fP^r (E),
\end{align} 
where \(\Pi (\mu, \nu)\) is the set of couplings of \(\mu\) and \(\nu\), i.e., the set of all Borel probability measures \(\pi\) on \(E^2\) such that \(\pi (dx \times E) = \mu\) and \(\pi (E\times dy) = \nu\).
As shown in \cite{CD08}, \(\w^E_r\) is truly a metric once \((E, m_E)\) is separable.
 It is also well-known (\cite[Theorem~6.18]{villani09}) that \((\fP^r (E), \w^E_r)\) is a complete separable metric space once \((E, m_E)\) has the same property. We endow \(\fP^r (E)\) with the topology induced by \(\w^E_r\). In case \(E\) is a (separable) Banach space, we take \(e_0 = 0\) by convention.

Next, fix a function \(\N \colon \Y \to [0, \infty]\) with the following properties: it is lower semicontinuous, \(\N (x) = 0\) implies \(x = 0\), 
\[
\N (cy) \leq c^\alpha \N (y), \quad \forall \, c \geq 0, \, y \in \Y, 
\]
and 
\[
\big\{ y \in \Y \colon \N (y) \leq 1 \big\} \text{ is relatively compact in \(\Y\)}.
\]
Moreover, we set
\begin{align} \label{eq: Np}
\N_p (y) := \|y\|^{2 (p - 1)}_\H \N (y), \quad p \geq 1, \, y \in \Y.
\end{align} 

To model the control variable, we fix a compact metrizable space \(F\), which is typically called action space. Let \(\m\) be the set of all Radon measures on \(([0, T] \times F, \mathcal{B} ([0, T] \times F))\) whose projection to \([0, T]\) coincides with the Lebesgue measure. 
We endow \(\m\) with the weak topology, which turns it into a compact metrizable space (\cite[Theorem 2.2]{EKNJ88}). The Borel \(\sigma\)-field on \(\m\) is denoted by \(\mathcal{M}\) and the identity map on \(\m\) is denoted by \(M\).
Further, we define the filtration
\[
\mathcal{M}_t := \sigma \big(M (C) \colon C \in \mathcal{B}([0, t] \times F)\big), \quad t \in [0, T].
\]
In the following, we consider the product space \(\Theta:= \Omega \times \m\). Let \(\r\) be a metric on \(\m\) that induces its topology. Then, we endow \(\Theta\) with the metric \(\d + \r\), which induces the product topology. 
We also set \(\mathcal{O} := \cF \otimes \mathcal{M}\) and \(\mathbf{O} := (\mathcal{O}_t)_{t \in [0, T]}\) with \(\mathcal{O}_t := \cF_t \otimes \mathcal{M}_t\). Adapting our previous notation, we denote the coordinate map on \(\Theta\) by \((X, M)\).

Let \(\j \colon \H \to \bR_+\) be a continuous function that is bounded on bounded subsets of \(\H\) and such that
\begin{equation} \label{eq: kappa}
	\begin{split}
\j (x) \geq  (1 &+ 2 \|x\|^{2\eta}_\H  ) e^{6  \lambda \eta T (1 + 20 \eta)} + \frac{1}{2} \Big[ \|x\|^{2}_\H + 9 \lambda T  (1 + 2 \|x\|^{2}_\H  ) e^{126  \lambda  T} \Big]
\\&+ \frac{1}{\beta + 2} \Big[ \|x\|^{\beta + 2}_\H + 3 \lambda \Big( \frac{\beta}{2} + 1\Big) T (\beta + 3) (1 + 2 \|x\|^{\beta + 2}_\H  ) e^{3  \lambda (\beta + 2) T (10 \beta + 21)}  \Big].
\end{split}
\end{equation} 
The lower bound for \(\j\) is ad hoc. In our proofs, it will appear in a Gronwall argument. 
For \(n \in \mathbb{N}\) and \(x \in \H\), define
\begin{align} \label{eq: def G}
	\NN& := \Big\{ \mu \in \fP (\Theta) \colon E^\mu \Big[ \sup_{s \in [0, T]} \|X_s\|^{2\eta}_\H + \int_0^T \big[ \N (X_s) + \N_{\beta /2 + 1} (X_s) \big] \, ds \Big] < \infty \Big\},
\end{align}
and
\begin{equation} \label{eq: def J}
\begin{split}
\J (x) := \Big\{ Q \in \fP (\fP (\Theta)) \colon \int E^\mu \Big[ & \sup_{s \in [0, T]} \|X_s\|^{2\eta}_\H \\&+ \int_0^T \big[ \N (X_s) + \N_{\beta /2 + 1} (X_s) \big] \, ds \Big] \, Q (d \mu) \leq \j (x) \Big\}.
\end{split}
\end{equation} 
	With applications in view, we would like to circumvent certain types of linear growth conditions (with respect to certain ``large'' norms) on the coefficients of the model investigated below. Such conditions typically lead to moment estimates, which are very useful in proofs of compactness properties. To overcome this, we include the necessary moment properties already in the definition of the model, which explains the role of \(\NN\) and~\(\J\). Of course, this procedure needs some care with regard to the existence of the model, which is the reason for the very explicit lower bound for \(\j\).
	
By Lemma~\ref{appendix: bogachev} in Appendix~\ref{sec: appendix}, the sets \(\NN\) and \(\J(x)\) are Borel subsets of \(\fP(\Theta)\) and \(\fP(\fP(\Theta))\), respectively.

We write \(\fP^{2 \eta}_{\al} (\H)\) for the space \(\fP^{2 \eta} (\H)\) endowed with the topology induced by \(\w^{\V}_{\al}\), i.e., the subspace topology coming from \(\fP^\al (\V)\).

Finally, let \(\mathsf{h}\) be the Hausdorff metric on the space of compact subsets of \(\mathcal{P}^\varrho (\mathcal{P}^\varrho (\Theta))\), which, for two compact sets \(\mathscr{K}, \mathscr{K}^* \subset \mathcal{P}^\varrho (\mathcal{P}^\varrho (\Theta))\),  is defined by 
\begin{align} \label{eq: def HDM}
	\mathsf{h} (\mathscr{K}, \mathscr{K}^*) := \max \Big\{ \max_{P \in \mathscr{K}} \w_{\varrho}^{\mathcal{P}^\varrho (\Theta)} (P, \mathscr{K}^*), \, \max_{P \in \mathscr{K}^*} \w_{\varrho}^{\mathcal{P}^\varrho (\Theta)} (P, \mathscr{K}) \Big\}, 
\end{align}
where
\[
\w_{\varrho}^{\mathcal{P}^\varrho (\Theta)} (P, \mathscr{K}^*) := \inf_{Q \in \mathscr{K}^*} \w_\varrho^{\mathcal{P}^\varrho (\Theta)} (P, Q), \quad \w_{\varrho}^{\mathcal{P}^\varrho (\Theta)} (P, \mathscr{K}) := \inf_{Q \in \mathscr{K}} \w_\varrho^{\mathcal{P}^\varrho (\Theta)} (P, Q)
\]
are the distance functions associated to \(\mathscr{K}^*\) and \(\mathscr{K}\). We refer to the Sections~3.16 and 3.17 in~\cite{charalambos2013infinite} for more information on the Hausdorff metric and the associated Hausdorff metric topology. 

\subsection{The Coefficients}
Next, we introduce the coefficients for the equations under consideration. 
We take two Borel functions  
\begin{equation} \label{eq: def coefficients} \begin{split}
	&b \colon F \times [0, T] \times \Y \times \fP^{2 \eta}_\al(\H) \to \V, \\
	&\sigma \colon F \times [0, T] \times \Y \times \fP^{2 \eta}_\al (\H) \to L_2 (\U; \H),
\end{split}
\end{equation}
such that 
\[
\sup_{f \in F} \| \sigma (f, t, y, \mu)\|_{L_2 (\U; \H)} < \infty
\]
for all \((t, y, \mu) \in [0, T] \times \Y \times \fP^{2 \eta}_{\al}\).
Further, let 
\[
\us \colon \fP (F) \times [0, T] \times \Y \times \fP^{2 \eta}_\al (\H) \to L_2 (\U; \H)
\]
be a Borel map such that 
\begin{align} \label{eq: os}
\us \us^* (\nu, t, y, \mu) = \int \sigma \sigma^* (f, t, y, \mu) \, \nu (df) 
\end{align}
for all \((\nu, t, y, \mu) \in \fP (F) \times [0, T] \times \Y \times \fP^{2 \eta}_\al (\H)\). 
We emphasize that \(\us\) always exists. For example, one can take the nonnegative operator root (\cite[Proposition~2.3.4]{liu_rockner} or \cite[Theorem~VI.9]{ReSi}) of the right hand side in~\eqref{eq: os}. As the root map \(A \mapsto A^{1/2}\) is continuous in the norm and strong operator topologies (\cite[Problem~14, p. 217]{ReSi}), it does not only preserve measurability but also continuity.

Our main result has two layers. 
The first layer requires the following condition.

\begin{condition} \label{cond: main1}
	\begin{enumerate}
		\item[\textup{(i)}]  \(\V^* \subset \Y\) compactly.
		\item[\textup{(ii)}] For every \(v \in \V^*, t \in [0, T]\) and \(\nu \in \fP (F)\), the maps 
		\begin{align*}
			F \times \Y \times \fP^{2 \eta}_\al(\H) \ni (f, y, \mu) &\mapsto {_\V}\langle b (f, t, y, \mu), v \rangle_{\V^*} \in \bR, \\
			F \times \Y \times \fP^{2 \eta}_\al(\H) \ni (f, y, \mu)  &\mapsto \sigma^* (f, t, y, \mu) v \in \U, \\
				\Y \times \fP^{2 \eta}_\al(\H) \ni (y, \mu)  &\mapsto \us^* (\nu, t, y, \mu) v \in \U
		\end{align*}
	are continuous. Further, for every \(v \in \V^*\) and all compact sets \(\mathscr{K} \subset \Y\) and \(\mathscr{H} \subset \fP^{2 \eta}_\al (\H)\), the maps \({_\V} \langle b, v \rangle_{\V^*}\) and \(\us^* v\) are bounded on \(F \times [0, T] \times \mathscr{K} \times \mathscr{H}\) and \(\fP (F) \times [0, T] \times \mathscr{K} \times \mathscr{H}\), respectively.
		\item[\textup{(iii)}] For the constant \(\lambda> 0\) as in \eqref{eq: cond constants}, it holds that
		\begin{align}
			{_\V}\langle b (f, s, w, \mu), w\rangle_{\V^*} & \leq \lambda \big(1 +\|w\|^2_\H + \|\mu\|^2_{2, \H}\big)  - \N (w), \label{eq: cond coe}
			\\
		\|\sigma (f, s, v, \mu)\|^{2 \gamma}_{L_2 (\U; \H)} + \| b (f, s, v, \mu)\|^{\gamma}_{\V} &\leq \lambda \big( (1 + \N (v) ) (1 + \|v\|^\beta_\H )+ \|\mu\|^{\beta}_{\beta, \H} \big), \label{eq: cond growth drift}
			\\
			\| \us (\nu, s, v, \mu) \|^2_{L_2 (\U; \H)} &\leq  \lambda \big(1 + \|v\|^2_\H + \|\mu\|^2_{2, \H} \big), \label{eq: cond growth diffusion}
		\end{align}
		for all \(f \in F, \nu \in \fP (F), s \in [0, T], w \in \V^*, v \in \Y\) and \(\mu \in \fP^{2 \eta}_\al (\H)\).
	\end{enumerate}
\end{condition}

The second layer needs also the next condition.

\begin{condition} \label{cond: main2}
	\begin{enumerate}
		\item[\textup{(i)}] \(b (F \times [0, T] \times \Y \times \fP^{2 \eta}_\al(\H)) \subset \Y^*\).
		\item[\textup{(ii)}] There exists a constant \(C > 0\) such that 
		\begin{align}
			\| b (f, s, y, \mu) \|^{\alpha / (\alpha - 1)}_{\Y^*} &\leq C \big( (1 + \mathcal{N} (y) ) (1 + \|y\|^\beta_\H) + \|\mu\|^\beta_{\beta, \H}\big), \label{eq: cond 2 (ii.1)}
			\\
			{_{\Y}}\langle  y - v, b (f, s, y, \mu) - b (f, s, v, \mu^*) \rangle_{\Y^*} &\leq C \big( \|y - v\|^2_{\X} + \w^{\X}_2 (\mu, \mu^*) \big), \label{eq: cond 2 (ii.2)}
			\\
			\| \us (\nu, s, y, \mu) - \us ( \nu, s, v, \mu^*)\|^2_{L_2 (\U; \X)} &\leq C \big( \|y - v\|^2_{\X} + \w^{\X}_2 (\mu, \mu^*) \big) \label{eq: cond 2 (ii.3)}
		\end{align} 
	for all \(f \in F, \nu \in \fP (F), s \in [0, T]\) and \(y, v \in \Y, \mu, \mu^* \in \fP^{2 \eta}_\al(\H)\). 
	\end{enumerate}
\end{condition}

We are in the position to introduce the particle systems of interest and its proposed mean field limits.

\subsection{Controlled Particles Systems and Mean Field Limits}
For \(n \in \mathbb{N}\), define 
\begin{align} \label{eq: def emp theta}
&\Z_n \colon \Theta^n \to \fP (\Theta), \quad \Z_n (\theta^1, \dots, \theta^n) := \frac{1}{n} \sum_{k = 1}^n \delta_{\theta^k},
\\ \label{eq: def emp omega^n}
&\x_n \colon \Omega^n \to \fP (\Omega), \quad \x_n (\omega^1, \dots, \omega^n) := \frac{1}{n} \sum_{k = 1}^n \delta_{\omega^k}.
\end{align}
With little abuse of notation, we will also write 
\[
\x_n (X_s) = \frac{1}{n} \sum_{k = 1}^n \delta_{X^k_s}, \quad X_s = (X^1_s, \dots, X^n_s).
\] 
It is known (\cite[Lemma 3.2]{LakSPA15}) that there exists a predictable probability kernel \(\v\) from \([0, T] \times \Theta\) into \(F\) such that 
\[
M (dt, df) = \v (t, M, df) \, dt.
\]
We will also write \(\v (t, M)\) for the measure \(\v (t, M, df)\).
The following are the particles systems of interest.
\begin{definition} \label{def: C^n}
	For \(n \in \mathbb{N}\) and \(x \in \H\), let \(\cC^n (x)\)  be the set of all \(Q \in \fP (\Theta^n)\) with the following properties:
	\begin{enumerate}
	\item[\textup{(i)}] \(Q \circ \Z_n^{-1} \in \J (x)\); 
	\item[\textup{(ii)}] possibly on a standard extension of \((\Theta^n, \mathcal{O}^n, \mathbf{O}^n, Q)\), there exist independent cylindrical standard Brownian motions \(W^1, \dots, W^n\) over \(\U\) such that a.s., for \(k = 1, \dots, n\), all \(t \in [0, T]\) and \(v \in \V^*\),
	\begin{align*}
		{_\V}\langle X^k_t, v \rangle_{\V^*} = {_\V}\langle x, v\rangle_{\V^*} &+ \int_0^t \blv \Big\langle \int b (f, s, \oX^k_s, \x_n (X_s)) \, \v (s, M^k, df), v \Big \rangle_{\V^*} \, ds
		\\&+ \Big\langle \int_0^t \us( \v (s, M^k), s, \oX^k_s, \x_n (X_s)) \, d W^k_s, v \Big \rangle_\H.
	\end{align*}
	Of course, it is implicit that the integrals are well-defined.
\end{enumerate}
	Further, we define 
	\begin{align} \label{eq: def R^n}
	\cR^n (x) := \big\{ Q \circ \Z_n^{-1} \colon Q \in \cC^n (x) \big\} \subset \fP (\fP (\Theta)).
	\end{align} 
\end{definition}

The next definition introduces the proposed mean field limit of the system \((\cR^n (x))_{n = 1}^\infty\).
\begin{definition} \label{def: C MK}
	For \(x \in \H\), we define \(\cC^0 (x)\) to be the set of all measures \(Q \in \fP(\Theta)\) 
	with the following properties:
		\begin{enumerate}	
	\item[\textup{(i)}] \(Q \in \NN\); 
	\item[\textup{(ii)}] possibly on a standard extension of \((\Theta, \mathcal{O}, \mathbf{O}, Q)\), there exists a cylindrical standard Brownian motion \(W\) over \(\U\) such that a.s., for all \(t \in [0, T]\) and \(v \in \V^*\),
	\begin{align*}
		{_\V}\langle X_t, v \rangle_{\V^*} = {_\V}\langle x, v\rangle_{\V^*}&+ \int_0^t \blv \Big\langle  \int b (f, s, \oX_s, Q^X_s) \, \v (s, M, df), v \Big \rangle_{\V^*} \, ds
		\\&+ \Big\langle \int_0^t \us( \v (s, M),s, \oX_s, Q^X_s) \, d W_s, v \Big \rangle_\H.
	\end{align*}
	Of course, it is implicit that the integrals are well-defined.
\end{enumerate}
	Further, we set 
	\[
	\cR^0 (x) := \big\{ P \in \fP (\fP(\Theta)) \colon P (\cC^0 (x)) = 1\big\} \cap \J (x).
	\]
\end{definition}

	Under Condition~\ref{cond: main1}~(iii), the set \(\cC^0 (x)\) is a Borel subset of \(\fP (\Theta)\) by Lemma~\ref{lem: mg chara} below, and in this case \(\cR^0 (x) = \fP (\cC^0 (x)) \cap \J(x)\).

The Definitions~\ref{def: C^n} and \ref{def: C MK} are in the spirit of the variational framework for SPDEs as studied by Pardoux \cite{par75} and Krylov--Rozovskii \cite{krylov_rozovskii}, see also \cite{liu_rockner,par}.
In Section~\ref{sec: mp} below, we establish characterizations via controlled martingale problems as studied in \cite{EKNJ88,ElKa15}.

\subsection{Main Result}
The following theorem is the main result of this paper. 
\begin{theorem} \label{theo: main1}
	Assume that Condition \ref{cond: main1} holds and take a sequence \((x^n)_{n = 0}^\infty \subset \H\) such that \(x^n \to x^0\) in \(\|\cdot\|_\H\). Then, the following hold:
	\begin{enumerate}
		\item[\textup{(i)}] The sets \(\cR^0 (x^0)\) and \(\cR^n (x^n)\), for every \(n \in \mathbb{N}\), are nonempty and compact in \(\fP^\al (\fP^\al (\Theta))\). 
		\item[\textup{(ii)}] Every sequence \((Q^n)_{n = 1}^\infty\) with \(Q^n \in \cR^n (x^n)\) is relatively compact in \(\fP^\al (\fP^\al (\Theta))\) and each of its accumulation points is in the set \(\cR^0 (x^0)\).
		\item[\textup{(iii)}] For every upper semicontinuous function \(\psi \colon \fP^\al (\Theta) \to \bR\),
such that 
\begin{align} \label{eq: bound property}
	\exists \, C > 0 \colon \ \ |\psi (\nu)| \leq C  \big[ 1 + \|\nu\|^\al_{\al,  \Theta}\big] \ \ \forall \, \nu \in \fP^\al (\Theta),
\end{align}
		it holds that
		\[
		\limsup_{n \to \infty} \sup_{Q \in \cR^n (x^n)} E^Q \big[ \psi \big] \leq \sup_{Q \in \cR^0 (x^0)} E^Q \big[ \psi \big].
		\]
	\end{enumerate}
	Assume in addition that Condition \ref{cond: main2} holds.	
	\begin{enumerate}
		\item[\textup{(iv)}] For every \(Q^0 \in \cR^0 (x^0)\) there are \(Q^{n} \in \cR^{n} (x^{n})\) such that \(Q^{n} \to Q^0\) in \(\fP^\al(\fP^\al (\Theta))\). 
		\item[\textup{(v)}] For every lower semicontinuous function \(\psi \colon \fP^\al (\Theta) \to \bR\) with the property \eqref{eq: bound property}, it holds that 
		\[
		 \sup_{Q \in \cR^0 (x^0)} E^Q \big[ \psi \big] \leq \liminf_{n \to \infty} \sup_{Q \in \cR^n (x^n)} E^Q \big[ \psi \big].
		\]
				\item[\textup{(vi)}] For every compact set \(\mathscr{K} \subset \H\) and every continuous function \(\psi \colon \fP^\al (\Theta) \to \bR\) with the property \eqref{eq: bound property}, it holds that
				\begin{align} \label{eq: compact convergence}
				\sup_{x \in \mathscr{K} } \Big| \sup_{Q\in \cR^n (x)} E^Q \big[ \psi \big] - \sup_{Q\in \cR^0(x)} E^Q \big[ \psi \big] \Big| \to 0, \quad n \to \infty,
				\end{align}
			and the map 
			\[
			\H \ni x \mapsto \sup_{Q\in \cR^0(x)} E^Q \big[ \psi \big]
			\]
			is continuous.
		\item[\textup{(vii)}] 
		For every compact set \(\mathscr{K} \subset \H\), it holds that
		\[
		\sup_{x \in \mathscr{K}} \mathsf{h} ( \cR^n (x), \cR^0 (x)) \to 0, \quad n \to \infty,
		\]
		where \(\mathsf{h}\) is the corresponding Hausdorff metric.
		In particular, the map \(x \mapsto \cR^0 (x)\) is continuous in the Hausdorff metric topology. 
	\end{enumerate}
\end{theorem}

\begin{remark}
	\begin{enumerate}
		\item[\textup{(i)}]
		Part (i) of Theorem~\ref{theo: main1} (actually this follows from part (ii)) includes an existence result for certain mean field controlled stochastic PDEs. In particular, taking \(F\) as singleton, this includes an existence result for usual (in the sense of uncontrolled) McKean--Vlasov stochastic PDEs, comparable to an existence result from \cite{HLL23}. 
		
		\item[\textup{(ii)}]
		Part (ii) of Theorem~\ref{theo: main1} provides a particle approximation for controlled equations. For uncontrolled situations, this is comparable to a propagation of chaos result from \cite{HLL23} for the stochastic 2D Navier--Stokes equation, which is covered by our result, see part (iv) of this remark. 
		\item[\textup{(iii)}]
		The items (iii), (v) and (vi) from Theorem~\ref{theo: main1} are related to stochastic optimal control theory, as they provide some understanding of mean field control problems and their approximations. We will discuss more details in Corollary~\ref{coro: control} below.
		
		The flavor of part (vii) from Theorem~\ref{theo: main1} is probabilistic. It can be viewed as {\em set-valued propagation of chaos}. To the best of our knowledge, this is a new concept in the literature on mean field control.
				\item[\textup{(iv)}] Let us also comment on our main assumptions.
		Condition~\ref{cond: main1} is rather mild and heavily inspired by (C1) -- (C3) from \cite{GRZ}. To give some benchmarks, it is satisfied for stochastic porous media equations (\cite[Section~5]{GRZ}) and Navier--Stokes equations on bounded domains (\cite[Section~6]{GRZ}). 
		
		Condition~\ref{cond: main2} is a type of growth and weak monotonicity condition that appears frequently in the literature (\cite{krylov_rozovskii,liu_rockner,par}). For example, it can be verified for stochastic porous media equations (\cite[Example~4.1.11]{liu_rockner}). However, it fails for stochastic Navier--Stokes equations which lack weak monotonicity (\cite[Section~5.1.3]{liu_rockner}). 
		
		In Section~\ref{sec: ex} below, we show that both Conditions~\ref{cond: main1} and \ref{cond: main2} hold for certain controlled slow diffusion equations. 
		
		In finite dimensional cases, the Conditions~\ref{cond: main1} and \ref{cond: main2} correspond to classical growth and monotonicity conditions. In fact, in many standard finite dimensional settings, monotonicity conditions are weaker than their Lipschitz counterparts, meaning that Theorem~\ref{theo: main1} covers some new cases even for finite dimensional frameworks (cf. \cite[Assumption B]{LakSIAM17}).
		
			It is an interesting open problem to weaken our monotonicity conditions to local versions as introduced in \cite{liu_r:JFA} for classical variational SPDE frameworks. This appears to be challenging, since localization techniques do not always interact well with McKean--Vlasov equations, whose dependence on the distribution is inherently nonlocal.     In particular, the localization-based uniqueness argument of \cite{liu_r:JFA} does not transfer straightforwardly to the McKean--Vlasov setting under local monotonicity conditions.
		 While in finite dimensional settings such issues may sometimes be overcome by Lyapunov function techniques (see \cite{HSS_21, HHL_24}), to the best of our knowledge no general tools are currently available in our infinite dimensional setting.
	\end{enumerate}
\end{remark}

The next corollary is related to \cite[Theorems 2.11, 2.12]{LakSIAM17}, which are limit theorems in the context of stochastic optimal control. 
The first part shows that accumulation points of \(n\)-state nearly optimal controls are optimal for the McKean--Vlasov system, while the second part explains that every optimal McKean--Vlasov control can be obtained as limit of \(n\)-state nearly optimal controls.

\begin{corollary} \label{coro: control}
	Suppose that the Conditions~\ref{cond: main1} and \ref{cond: main2} hold, take a continuous function \[\psi \colon \fP^\al (\Theta)\to \bR\] with the property \eqref{eq: bound property} and an initial value \(x \in \H\).
	\begin{enumerate}
		\item[\textup{(i)}] Let \((\varepsilon_n)_{n = 1}^\infty \subset \bR_+\) be a sequence such that \(\varepsilon_n \to 0\). For \(n \in \mathbb{N}\), suppose that \(Q^n \in \cR^n (x)\) is such that 
					\[
		\sup_{Q \in \cR^n (x)} E^{Q} \big[ \psi \big] - \varepsilon^n \leq E^{Q^n} \big[ \psi \big].
		\]
		In other words, \(Q^n\) is a so-called {\em \(n\)-state \(\varepsilon_n\)-optimal control}.
		Then, the sequence \((Q^n)_{n = 1}^\infty\) is relatively compact in \(\fP^\al (\fP^\al(\Theta))\) and every accumulation point \(Q^0\) is in \(\cR^0 (x)\) and optimal in the sense that 
		\begin{align} \label{eq: def opti}
		E^{Q^0} \big[ \psi \big] = \sup_{Q \in \cR^0 (x)} E^Q \big[ \psi \big].
		\end{align}
		\item[\textup{(ii)}] 
		Take a measure \(Q^0 \in \cR^0 (x)\) that is optimal (i.e., it satisfies \eqref{eq: def opti}). Then, there are sequences \((\varepsilon_n)_{n = 1}^\infty \subset \bR_+\) and \((Q^n)_{n = 1}^\infty \subset \fP^\al (\fP^\al(\Theta))\) such that \(\varepsilon_n \to 0\), each \(Q^n\) is an \(n\)-state \(\varepsilon_n\)-optimal control and \(Q^n \to Q^0\) in \(\fP^\al (\fP^\al(\Theta))\).
	\end{enumerate}
\end{corollary}

\begin{proof}
	(i). By Theorem \ref{theo: main1} (vi), we have 
					\[
\sup_{Q \in \cR^0 (x)} E^Q \big[ \psi \big] \leftarrow \sup_{Q \in \cR^n (x)} E^{Q} \big[ \psi \big] - \varepsilon^n \leq E^{Q^n} \big[ \psi \big] \leq \sup_{Q \in \cR^n (x)} E^Q \big[ \psi \big] \to \sup_{Q \in \cR^0 (x)} E^Q \big[ \psi \big],
\]
which shows that 
	\[
	\lim_{n \to \infty} E^{Q^n} \big[ \psi \big] = \sup_{Q \in \cR^0 (x)} E^Q \big[ \psi \big].
	\]
	By part (ii) of Theorem \ref{theo: main1}, \((Q^n)_{n = 1}^\infty\) is relatively compact in \(\fP^\al (\fP^\al (\Theta))\) and every accumulation point \(Q^0\) is in \(\cR^0 (x)\).
	Thus, we get that
	\[
	E^{Q^0} \big[ \psi \big] = \lim_{n \to \infty} E^{Q^n} \big[ \psi \big] = \sup_{Q \in \cR^0 (x)} E^Q \big[ \psi \big].
	\]
	This is the claim.
	
	(ii). By Theorem \ref{theo: main1} (iv), there exists a sequence \((Q^n)_{n = 1}^\infty\) such that \(Q^n \in \cR^n (x)\) and \(Q^n \to Q^0\) in \(\fP^\al (\fP^\al (\Theta))\). Using that \(Q^0\) is optimal and Theorem~\ref{theo: main1}~(vi), we get that
	\[
	\lim_{n \to \infty} \sup_{Q \in \cR^n (x)} E^{Q} \big[ \psi \big] = E^{Q^0} \big[ \psi \big] = \lim_{n \to \infty} E^{Q^n} \big[ \psi \big].
	\]
	Consequently, 
	\[
	0 \leq \varepsilon^n := \sup_{Q \in \cR^n (x)} E^{Q} \big[ \psi \big] - E^{Q^n} \big[ \psi \big]\to 0,
	\]
	which shows that \(Q^n\) is an \(n\)-state \(\varepsilon^n\)-optimal control. The claim is proved.
\end{proof}

\section{Example: Mean Field Control for Porous Media Models} \label{sec: ex}
In this section we verify the Conditions~\ref{cond: main1} and \ref{cond: main2} for interacting controlled stochastic porous media systems of the form
\[
d Y^k_t = \Big[ \Delta ( |Y^k_t|^{q - 2} Y^k_t) + \frac{1}{N} \sum_{i = 1}^N \, (Y^k_t - Y^i_t)+ \int c\, (f)\, \v (t, M^k, df) \Big] \, dt + \sigma \, d W^k_t,
\]
where \(q \geq 2\) and \(k = 1, \dots, N\). For \(q = 2\) this equation corresponds to the stochastic heat equation and the cases \(q > 2\) are called slow diffusion models. We refer to the monograph~\cite{BPR} for more information about (stochastic) porous media equations. 

Let us introduce our precise setting for this system.
Let \(d \in \mathbb{N}\) be a fixed dimension and let \(\mathcal{O} \subset \bR^d\) be a bounded domain with smooth boundary. For \(k \in \mathbb{N}\) and \(p > 1\), let \(W^{k, p}_0 (\mathcal{O})\) be the usual Sobolev space on \(\mathcal{O}\) with Dirichlet boundary conditions. The dual space of \(W^{k, p}_0 (\mathcal{O})\) can be identified with \(W^{-k, p'} (\mathcal{O})\), where \(p' = p / (p - 1)\) (\cite[Theorem~3.12]{adams} or \cite[Section~III.1]{krylov_rozovskii}). In this way, the dualization between \(W^{k, p}_0 (\mathcal{O})\) and \(W^{-k, p'} (\mathcal{O})\) is defined by means of the usual scalar product in \(L^2 (\mathcal{O})\). 
In the following, we take some \(q \geq 2\) and 
\begin{alignat*}{3}
	\Y &:= L^q (\mathcal{O}), \qquad \qquad &&\H:= L^2 (\mathcal{O}), \qquad \qquad &&\X := W^{-1, 2} (\mathcal{O}), 
	\\
	\V &:= W^{- d - 2, 2} (\mathcal{O}), && \V^* := W^{d + 2, 2}_0 (\mathcal{O}). &&
\end{alignat*}
Further, we set 
\begin{align*}
\N (y) := \begin{cases} \int_{\mathcal{O}} \big| \nabla ( |y(u)|^{q/2 - 1} y (u) ) \big|^2 \, du, & \text{ if } |y|^{q/2 - 1} y \in W^{1, 2}_0 (\mathcal{O}), \\
		+ \infty, &\text{ otherwise}.
		\end{cases}
\end{align*}
Thanks to \cite[Lemma~5.1]{GRZ}, for \(\alpha := q\), the function \(\N\) is as in the previous section, i.e., it has all assumed properties.
The action space \(F\) is still a compact metrizable space. 

Next, we introduce the coefficients \(b\) and \(\sigma\). 
We set 
\[
b (f, s, x, \mu) \equiv b (f, x, \mu) (\,\cdot\,) := \Delta (|x (\,\cdot\,)|^{q - 2} x (\,\cdot\,)) + \int (x (\,\cdot\,) - z (\,\cdot\,)) \, \mu(dz) + c (f) (\,\cdot\,), 
\]
where \(F \times \bR^d \ni (f, u) \mapsto c (f) (u)\) is a bounded continuous function. Further, for some separable Hilbert space \(\U\), let \(\sigma \in L_2 (\U; \H)\) be a constant Hilbert--Schmidt coefficient.
\begin{lemma}
	In the above setting, Condition~\ref{cond: main1} holds for \(\gamma = q / (q - 1)\), large enough \(\lambda\) and all \(\beta, \eta\) that satisfy~\eqref{eq: cond constants}.
\end{lemma}
\begin{proof}
	We have \(\V^* \subset \Y\) compactly by \cite[Theorem~6.3]{adams}, i.e., (i) holds. Moreover, except for the measure dependent component \(
	- \int z \mu (dz),
	\)
	(ii) and (iii) from Condition~\ref{cond: main1} follow from \cite[Lemma~5.2]{GRZ}. In the following, we comment on the measure component. 
	
	Take \(v \in \V^*\) and a sequence \((\mu^n)_{n = 0}^\infty \subset \fP^{2 \eta}_\al (\H)\) such that \(\mu^n \to \mu^0\)  in \(\fP^{2 \eta}_\al (\H)\). 
	By the continuity of the map \(z \mapsto {_\V} \langle z, v \rangle_{\V^*}\) from \(\V\) into \(\bR\), and the fact that \(| {_\V} \langle z, v \rangle_{\V^*} | \leq \|z\|_\V \| v \|_{\V^*}\), we deduce from \(\mu^n \to \mu^0\) in \(\fP^{2 \eta}_\al (\H)\) that
	\begin{align*}
	\blv \Big \langle \int  z \,  \mu^n (dz) - \int  z  \, \mu^0 (dz), v \Big \rangle_{\V^*} = \int  {_\V}\langle z, v \rangle_{\V^*} \, ( \mu^n (dz) -\mu^0 (dz)) \to 0,
	\end{align*}
which shows that Condition~\ref{cond: main1}~(ii) holds. 

Next, we discuss Condition~\ref{cond: main1}~(iii). Take \(v \in \V^*, w \in \Y\) and \(\mu \in \fP^{2 \eta}_\al (\H)\). First, notice that 
\begin{align*}
	  - \blv \Big \langle \int  z \, \mu (dz), w \Big \rangle_{\V^*} &= - \int \langle z, w \rangle_\H \, \mu (dz) \leq \|w\|_\H \| \mu \|_{1, \H} \leq \tfrac{1}{2} \big( \|w\|^2_\H + \|\mu\|^2_{2, \H} \big).
\end{align*}
We conclude that \eqref{eq: cond coe} holds (where the term \(- \N\) comes from the non-measure dependent part of the coefficient \(b\), cf. \cite[Lemma~5.2]{GRZ}).
Finally, notice that 
\[
\Big\| \int z \, \mu (dz) \Big\|^\gamma_\V \leq \Big( \int \|z\|_\V \, \mu (dz) \Big)^{\gamma} \leq C \|\mu\|_{1, \H}^\gamma \leq C\|\mu\|_{2, \H}^\gamma,
\]
i.e., we conclude \eqref{eq: cond growth drift}. This completes the proof. 
\end{proof}

Next, we also discuss Condition~\ref{cond: main2}. In the following, we consider the Gelfand triple
\[
\Y = L^q (\mathcal{O}) \subset \X = W^{-1, 2}(\mathcal{O}) \, \cong \, \X^* = W^{1, 2}_0 (\mathcal{O}) \subset \Y^* = (L^q (\mathcal{O}))^*, 
\]
where we understand \(\,\cong\,\) via the Riesz map \(\mathcal{R} := (- \Delta)^{-1}\) as discussed in \cite[Lemma~4.1.12]{liu_rockner}. We also recall the well-known fact (\cite[Lemma~4.1.13]{liu_rockner}) that \(\Delta\) extends to a linear isometry from \(L^{q / (q - 1)} (\mathcal{O})\) into \(\Y^*\). 

\begin{lemma}
In the above setting, Condition~\ref{cond: main2} holds.
\end{lemma}

\begin{proof}
	Notice that \(x \in \Y = L^q (\mathcal{O})\) implies that \(|x|^{q - 2} x \in L^{q / (q - 1)} (\mathcal{O})\).
	Hence, as \(\Delta\) extends from \(L^{q / (q - 1)} (\mathcal{O})\) into \(\Y^*\), we get that \(\Delta (|x|^{q - 2} x) \in \Y^*\) for all \(x \in \Y\). This implies that Condition~\ref{cond: main2}~(i) holds. 

	Next, we discuss Condition~\ref{cond: main2}~(ii). Take \(x \in \Y\) and recall that \(\alpha = q\). 
	Thanks to \cite[Lemma~C.1]{GRZ} (or see \eqref{eq: C1} below, where the used inequality is restated), we obtain that
	\begin{align*}
		\| \Delta (|x|^{q - 2} x) \|^{q / (q - 1)}_{\Y^*} &= \int_\mathcal{O} |x (u)|^{q} du = \|x\|_\Y^q \leq C \big( \N (x) + \|x\|^q_\V \big) \leq C \big( \N (x) + \|x\|^q_\H \big). 
	\end{align*}
Further, using that \(\H \subset \X \subset \Y^*\), we get that 
\begin{align*}
	\Big\| x - \int z \, \mu(dz) + c (f) \Big\|_{\Y^*}^{q / (q - 1)} 
	&\leq C\, \Big\| x - \int z \, \mu(dz) + c (f) \Big\|_{\H}^{q / (q - 1)}
	\\&\leq C\, \big(\|x\|^{q / (q - 1)}_\H + \|\mu\|_{2, \H}^{q/ (q - 1)} + 1 \big).
\end{align*}
Therefore, \eqref{eq: cond 2 (ii.1)} holds. 

It is easy to check that \((|t|^{q - 2} t - |s|^{q - 2} s)(t - s) \geq 0\) for all \(s, t \in \bR\).
Hence, for \(y, v \in \Y\), we obtain that
\begin{align*}
	{_{\Y}} \langle y - v&,  \Delta (|y|^{q - 2} y) - \Delta (|v|^{q- 1} v) \rangle_{\Y^*} 
	\\&= - \int_{\mathcal{O}} \big( |y (u)|^{q - 2} y (u) - |v (u)|^{q - 2} v (u)\big) \big( y(u) - v (u) \big) du 
	\leq 0.
\end{align*}
Take \(\mu, \nu \in \fP^{2 \eta}_\al (\H)\) and recall that the Wasserstein metric can be realized by some optimal coupling (\cite[part I, p. 353]{car_della}), i.e., there exists a coupling \(\pi\) of \(\mu\) and \(\nu\) such that 
\[
\w^{\X}_1 (\mu, \nu) = \int \|z - w\|_\X \, \pi(dz, dw).
\]
Consequently, we observe that 
\begin{align*}
	\Big\| \int z \, \mu (dz) - \int z \, \nu (dz) \Big\|_\X & = \Big\| \iint z \,\pi (dz, dw) - \iint w \, \pi (dz, dw) \Big\|_\X
	\\&\leq \iint \|z - w \|_\X \, \pi (dz, dw)
	= \w^{\X}_1 (\mu, \nu). 
\end{align*}
Using this observation, we get
\begin{equation*} 
	\begin{split}
	\mathop{\vphantom{\int}}\nolimits_{\Y}\hspace{-0.1cm} \Big \langle y - v, y - v + \int z \, \mu (dz) - \int z \, \nu(dz) \Big\rangle_{\Y^*}
	&= \Big \langle y - v,y - v + \int z \, \mu (dz) - \int z \, \nu(dz) \Big \rangle_\X
	\\&\leq \| y - v \|_\X\,  \Big\| y - v + \int z \, \mu (dz) - \int z \, \nu(dz)\Big\|_\X 
	\\&\leq \|y - v \|_\X^2 + \|y - v\|_\X^2 + \w^\X_1 (\mu, \nu)^2 \phantom \int 
	\\&\leq 2 \|y - v\|^2_\X + \w^\X_2 (\mu, \nu)^2. \phantom \int 
\end{split} 
\end{equation*}
All together, we conclude that \eqref{eq: cond 2 (ii.2)} holds. 

Finally, as \(\sigma\) is constant, \eqref{eq: cond 2 (ii.3)} holds trivially and the proof is complete.
\end{proof}

	\section{Proof of Theorem \ref{theo: main1}} \label{sec: pf}

	In this section we prove our main Theorem \ref{theo: main1}. Let us shortly comment on its structure. In Section~\ref{sec: mp}, we start with a martingale problem characterizations for the sets \(\cC^0\) and~\(\cC^n\), in Section~\ref{sec: compactness} we investigate compactness properties, and 
	in Section~\ref{sec: existence} we provide a general existence result for stochastic PDEs with certain random coefficients. 
	In the remaining sections, we proceed with the proof of Theorem~\ref{theo: main1} in a chronological order.

\subsection{Martingale Problem Characterizations of \(\cC^0\) and \(\cC^n\)} \label{sec: mp}
Let \(\B := \{e_i \colon i \in \mathbb{N}\} \subset \V^*\) be a countable dense set.
Furthermore, suppose that \(\tC\) is a countable subset of \(C^2_c(\bR; \bR)\) that is dense for the norm \(\|f\|_\infty + \|f'\|_\infty + \|f''\|_\infty\). 
For a Polish space \(Z\) and a \(\sigma\)-field \(\mathcal{Z} \subset \mathcal{B} (Z)\), we call \(G \subset C_b (Z; \bR)\) a separating class for \(\mathcal{Z}\) if \(\int g \, d\mu = \int g \, d \nu\) implies \(\mu = \nu\) on \(\mathcal{Z}\) for all \(\mu, \nu \in \mathcal{P} (Z)\).
Finally, for \(s \in [0, T]\), let \(\tT_s \subset C_b (\Theta; \bR)\) be a countable separating class for \(\mathcal{O}_s\), whose existence follows as in the proof of \cite[Lemma A.1]{LakSIAM17}. 

We also need the following notation to define relaxed control rules. For \(g \in C^2_b (\bR; \bR), v \in \V^*\) and \((f, t, y, \mu) \in F \times [0, T] \times \Y \times \fP^{2 \eta}_\al (\H)\), we set 
\begin{align} \label{eq: def gen MF}
	\mathcal{L}_{g, v} (f, t, y, \mu) := g' ({_\V} \langle y, v \rangle_{\V^*}) {_\V}\langle b (f, t, y, \mu), v \rangle_{\V^*} 
	+ \tfrac{1}{2} g'' ({_\V}\langle y, v \rangle_{\V^*}) \| \sigma^* (f, t, y, \mu) v \|^2_\U.
\end{align}
Finally, for \(g \in C^2_b (\bR; \bR), f \in F,  v^1, \dots, v^n \in \V^*,  y^1, \dots, y^n \in \Y, \mu \in \fP^{2 \eta}_{\al} (\H)\) and \(i = 1, \dots, n\), we set 
\begin{equation} \label{eq: def gen SE} \begin{split}
	\mathcal{L}^i_{g, v^1, \dots, v^n} &(f, t, y^1, \dots, y^n, \mu) 
	\\&\hspace{-0.75cm} := 
	g' \Big( \sum_{k = 1}^n {_\V}\langle y^k, v^k \rangle_{\V^*} \Big) {_\V}\langle b (f, t, y^i, \mu), v^i \rangle_{\V^*} + \frac{1}{2}\, g'' \Big( \sum_{k = 1}^n {_\V}\langle y^k, v^k \rangle_{\V^*} \Big) \| \sigma^* (f, t, y^i, \mu) v^i \|^2_\U.
\end{split}
\end{equation} 
We are in the position to provide a martingale characterization for the sets \(\cC^0\) and \(\cC^n\).

\begin{lemma} \label{lem: mg chara}
	Suppose that Condition \ref{cond: main1} \textup{(iii)} holds. Let \(x \in \H\) and \(Q\in \fP(\Theta)\). The following are equivalent:
	\begin{enumerate}
		\item[\textup{(i)}]
		\(Q \in \cC^0 (x)\).
		\item[\textup{(ii)}] The following properties hold:
		\begin{enumerate}
	\item[\textup{(a)}]
	\(Q (X_0 = x) = 1\);
	\item[\textup{(b)}] \(Q \in \NN\); 
	\item[\textup{(c)}] 
	for all \(v \in \V^*\) and \(g \in C^2_c (\bR; \bR)\), the process 
	\[
	\mathsf{M}^{g, v} := g ({_\V}\langle X, v\rangle_{\V^*})  - \int_0^\cdot \int \mathcal{L}_{g, v} (f, s, \oX_s, Q^X_s) \, M (ds, df)
	\]
	is a (square integrable) \(Q\)-\(\mathbf{O}\)-martingale.
\end{enumerate}
		\item[\textup{(iii)}] The following properties hold:
		\begin{enumerate}
			\item[\textup{(a)}]
			\(Q (X_0 = x) = 1\);
			\item[\textup{(b)}] \(Q \in \NN\); 
			\item[\textup{(c)}] 
			for all \(v \in \mathscr{B}, g \in \tC, s, t \in \mathbb{Q}_+ \cap [0, T], s < t\) and all \(\psi \in \tT_s\), 
			\[
			E^Q \big[ (\mathsf{M}^{g, v}_t - \mathsf{M}^{g, v}_s) \psi \big] = 0.
			\]
		\end{enumerate}
	\end{enumerate}
\end{lemma}
\begin{proof}
	We will prove the following implications:
	\begin{align*}
		\textup{(i)} \Rightarrow \textup{(ii)}, \quad \textup{(ii)} \Rightarrow \textup{(i)}, \quad \textup{(iii)} \Rightarrow \textup{(ii)}.
	\end{align*}
	As (ii) \(\Rightarrow\) (iii) is trivial, these complete the proof. 
	
	{\em \(\textup{(i)} \Rightarrow \textup{(ii)}\):}  Notice that (a) and (b) hold.
	We now work on a standard extension of \((\Theta, \mathcal{O}, \mathbf{O}, Q)\) and we use the notation from Definition~\ref{def: C MK}, i.e., a.s., for all \(t \in [0, T]\) and \(v \in \V^*\), 
	\begin{align*}
		{_\V}\langle X_t, v \rangle_{\V^*} = {_\V}\langle x, v\rangle_{\V^*}&+ \int_0^t \blv \Big\langle \int b (f, s, \oX_s, Q^X_s) \, \v (s, M, df), v \Big\rangle_{\V^*} ds
\\&+ \Big\langle \int_0^t \us (\v (s, M), s, \oX_s, Q^X_s) \, d W_s, v \Big \rangle_\H.
	\end{align*}
	In the following, fix \(v \in \V^*\).
	The classical It\^o formula for real-valued semimartingales yields that a.s., for all \(t \in [0, T]\), 
	\begin{align*}
		\mathsf{M}^{g, v}_t 
		&= g ({_\V} \langle X_t, v \rangle_{\V^*}) - \int_0^t \Big( g' ({_\V} \langle X_s, v \rangle_{\V^*})  \blv \Big\langle \int b (f, s, \oX_s, Q^X_s) \, \v (s, M, df), v \Big\rangle_{\V^*} 
		\\&\hspace{5cm} + \frac{g'' ({_\V}\langle X_s, v \rangle_{\V^*}) }{2} \| \us^* (\v (s, M), s, X_s, Q^X_s)v\|_\U^2 \Big) \, ds
		\\
		&= g ({_\V} \langle x, v \rangle_{\V^*}) + \int_0^t g' ( {_\V}\langle X_s, v\rangle_{\V^*} ) \langle \us^* (\v (s, M), s, \oX_s, Q^X_s) v, d W_s \rangle_\U.
	\end{align*}
	Denoting the quadratic variation process by \([\,\cdot\, , \, \cdot\,]\), we obtain that a.s.
	\begin{align*}
		\big[ \mathsf{M}^{g, v}, \mathsf{M}^{g, v} \big]_T = \int_0^T \big( g' ({_\V} \langle X_s, v\rangle_{\V^*})\big)^2 \|\us^* (\v (s, M), s, \oX_s, Q^X_s) v\|^2_\U \, ds.
	\end{align*}
	Using the growth bound \eqref{eq: cond growth diffusion} and that \(Q \in \NN\), we obtain that 
	\[
	E^Q \Big[ \big[ \mathsf{M}^{g, v}, \mathsf{M}^{g, v}\big]_T \Big] < \infty. 
	\]
We conclude that \(\mathsf{M}^{g, v}\) 
	is a (square integrable) \(Q\)-\(\mathbf{O}\)-martingale. This proves (ii). 
	
	{\em \(\textup{(ii)} \Rightarrow \textup{(i)}\):} 
	Take \(h \in C^2_b (\bR; \bR)\) and let \(g^n \in C^2_c (\bR; \bR)\) be such that \(h = g\) on \([-2n, 2 n]\). With
	\[
	T_n := \inf \{t \in [0, T] \colon | {_\V}\langle X_t, v \rangle_{\V^*} | \geq n\},
	\]
	the hypothesis of (ii) yields that the process \(\mathsf{M}^{h, v}_{\cdot \wedge T_n} = \mathsf{M}^{g^n, v}_{\cdot \wedge T_n}\) is a \(Q\)-\(\mathbf{O}\)-martingale. Consequently, because \(T_n \nearrow \infty\) as \(n \to \infty\), the process \(\mathsf{M}^{h, v}\) is a local \(Q\)-\(\mathbf{O}\)-martingale.
	Now, we deduce from Lemma~\ref{appendix: JS} in Appendix~\ref{sec: appendix} that, for every \(v \in \V^*\), the process 
	\[
	{_\V}\langle X, v \rangle_{\V^*} - {_\V} \langle x, v \rangle_{\V^*} - \int_0^\cdot \blv \Big \langle \int b (f, s, \oX_s, Q^X_s) \, \v (s, M, df), v \Big\rangle_{\V^*} \, ds 
	\]
	is a continuous local \(Q\)-\(\mathbf{O}\)-martingale with quadratic variation process 
	\[
	\int_0^\cdot \| \us^* (\v (s, M), s, \oX_s, Q^X_s) v \|^2_\U \, ds.
	\]
	The growth condition \eqref{eq: cond growth diffusion} and \(Q \in \NN\) imply that \(Q\)-a.s. 
	\[
	\int_0^T \| \us (\v(s, M), s, \oX_s, Q^X_s) \|^2_{L_2 (\U; \H)} \, ds < \infty.
	\]
	As the linear space \(\V^*\) is dense in \(\H^*\), Lemma~\ref{appendix: O Rep} in Appendix~\ref{sec: appendix} yields that, possibly on a standard extension of \((\Theta, \mathcal{O}, \mathbf{O}, Q)\), there exists a cylindrical Brownian motion \(W\) over \(\U\) such that, for all \(v \in \V^*\),
	\begin{align*}
		{_\V}\langle X, v \rangle_{\V^*} - {_\V} \langle x, v \rangle_{\V^*} - \int_0^\cdot  \blv \Big \langle  \int b (f&, s, \oX_s, Q^X_s) \, \v (s, M, df), v \Big\rangle_{\V^*} \, ds 
		\\&= \Big \langle \int_0^\cdot \us (\v (s, M), s, \oX_s, Q^X_s) \, d W_s, v \Big \rangle_\H.
	\end{align*}
	This proves that part (ii) from Definition~\ref{def: C MK} holds. Consequently, \(Q \in \cC^0 (x)\).
	
	{\em \(\textup{(iii)} \Rightarrow \textup{(ii)}\):} This implication follows readily by a density argument.
\end{proof}

A similar result can also be proved for the set \(\cC^n (x)\). 
	\begin{lemma} \label{lem: mg chara N}
	Suppose that Condition \ref{cond: main1} \textup{(iii)} holds. Let \(n \in \mathbb{N}, x \in \H\) and \(Q\in\fP(\Theta^n)\). The following are equivalent:
	\begin{enumerate}
		\item[\textup{(i)}]
		\(Q \in \cC^n (x)\).
		\item[\textup{(ii)}] The following hold:
			\begin{enumerate}
		\item[\textup{(a)}]
		\(Q (X^k_0 = x, k = 1, \dots, n) = 1\);
		\item[\textup{(b)}] \(Q \circ \Z_n^{-1} \in \J (x)\); 
		\item[\textup{(c)}]
		for all \(v^1, \dots, v^n \in \V^*\) and \(g \in C^2_b(\bR; \bR)\), the process 
		\begin{align*}
			g \Big( \sum_{k = 1}^n\, {_\V}\langle X^k, v^k\rangle_{\V^*}\Big) 
			- \sum_{k = 1}^n \int_0^\cdot \int  \mathcal{L}^k_{g, v^1, \dots, v^n} (f, s, \oX_s, \x_n (X_s)) \, M^k (ds, df)
		\end{align*}
		is a (square integrable) \(Q\)-\(\mathbf{O}\)-martingale.
	\end{enumerate}
	\end{enumerate}
\end{lemma}

\begin{proof} The proof is similar to those of (i) \(\Leftrightarrow\) (ii) from Lemma \ref{lem: mg chara} and we omit the details for brevity. \end{proof}

The following is a direct consequence of Lemma \ref{lem: mg chara N}.

\begin{corollary} \label{coro: conv}
	Suppose that Condition \ref{cond: main1} \textup{(iii)} holds. For every \(x \in \H\) and \(n \in \mathbb{N}\), the set \(\cC^n (x)\) is convex. 
\end{corollary}

\subsection{Compactness properties}\label{sec: compactness}
In this section we investigate (relative) compactness of the sets \(\cR^n\) and \(\cR^0\). 

	\begin{lemma} \label{lem: Rk rel comp}
		Suppose that Condition \ref{cond: main1} \textup{(i)} and \textup{(iii)} hold.
		For every nonempty bounded set \(B \subset \H\), the set
		\[
		\cR (B) := \bigcup_{n \in \mathbb{N}} \bigcup_{x \in B} \cR^n (x)
		\]
		is relatively compact in \(\fP^\al (\fP^\al (\Theta))\). 
	\end{lemma}
\begin{proof}
	{\em Step 1: Relative compactness in \(\fP(\fP(\Theta))\).} We define the map \(I \colon \fP(\fP(\Theta)) \to \fP(\Omega)\) by 
	\begin{align} \label{eq: I}
	I (Q) (G) := \int \mu (G \times \m) \, Q(d\mu), \quad G \in \mathcal{B}(\Omega).
	\end{align}
	In the following we prove relative compactness of the family \[\mathcal{I} := \{I (Q) \colon Q \in \cR(B)\} \subset \fP (\Omega).\]
	By virtue of Lemma~\ref{appendix: tight GRZ} in Appendix~\ref{sec: appendix}, Condition~\ref{cond: main1}~(i) and the definition of \(\J (x)\), it suffices to prove that there exists a constant \(\ell > 0\) such that 
	\begin{align} \label{eq: to show for tightness}
	\sup_{Q \in \mathcal{R} (B)} E^{I (Q)} \Big[ \sup_{s \not = t} \frac{\|X_s - X_t\|_\V}{|s - t|^{\ell}} \Big] < \infty.
	\end{align}
	Take \(Q = P \circ \Z^{-1}_n \in \cR^n (x)\). Then, 
	\[
	E^{I (Q)} \Big[ \sup_{s \not = t} \frac{\|X_s - X_t\|_\V}{|s- t|^{\ell}} \Big]  = \frac{1}{n} \sum_{k = 1}^n E^P \Big[ \sup_{s \not = t} \frac{\|X^k_s - X^k_t\|_\V}{|s - t|^{\ell}} \Big].
	\]
	By definition of \(\cC^n (x)\), have \(P\)-a.s., for all \(v \in \V^*\), 
\begin{align*}
		{_\V}\langle X^k_t, v \rangle_{\V^*} =  {_\V}\langle x, v \rangle_{\V^*} &+ \int_0^t \blv \Big\langle \int b (f, s, \oX^k_s, \x_n (X_s)) \v (s, M^k, df), v \Big \rangle_{\V^*} \, ds
		\\&+ \Big \langle \int_0^t \us (\v (s, M^k), s, \oX^k_s, \x_n (X_s))\, d W^k_s, v \Big\rangle_\H, 
\end{align*}
	for independent cylindrical Brownian motions \(W^1, \dots, W^n\).
	Thus, \(P\)-a.s., for \(s < t\),
	\begin{align*}
	\|X_t^k- X^k_s\|_\V \leq \int_s^t \int \| b &(f, r, \oX^k_r, \x_n (X_r)) \|_\V \, \v (r, M^k, df) \, dr 
	\\&+ \Big\| \int_s^t \us (\v (r, M^k), r, \oX^k_r, \x_n (X_r)) \, d W^k_r \Big\|_\H.
	\end{align*}
		Take \(\lambda^* \geq \lambda\) such that \(\lambda^* T \geq 1\). 
		Using H\"older's inequality for the first inequality, Jensen's inequality for the second, \eqref{eq: cond growth drift} for the third and finally that \(z^{1/\gamma} \leq z\) when \(z \geq 1\) for the fourth (the assumption \(\lambda^* T \geq 1\) implies that the value of the integral is larger or equal than one), we obtain that 
	\begin{align*}
		 \int_s^t \int \| b (f,& r, \oX^k_r, \x_n (X_r)) \|_{\V} \, \v (r, M^k, df) \, dr 
		 	\\&\leq |s - t|^{1 - 1 / \gamma} \, \Big( \int_s^t \Big( \int \| b (f, r, \oX^k_r, \x_n (X_r))\|_{\V} \, \v (r, M^k, df) \Big)^\gamma \, dr \Big)^{1/\gamma}
		 		 	\\&\leq |s - t|^{1 - 1 / \gamma} \, \Big( \int_s^t \int \| b (f, r, \oX^k_r, \x_n (X_r))\|_{\V}^\gamma \, \v (r, M^k, df) \, dr \Big)^{1/\gamma}
	 \\&\leq |s - t|^{1 - 1 /\gamma} \, \Big( \int_0^T \lambda^*  \Big( (1 + \N (X^k_r)) ( 1 + \|X^k_r\|^\beta_\H ) + \frac{1}{n} \sum_{i = 1}^n \|X^i_r\|^\beta_\H \Big) \, dr \Big)^{1/\gamma}
		 \\&\leq |s - t|^{1 - 1 /\gamma}  \int_0^T \lambda^* \Big( (1 + \N (X^k_r)) ( 1 + \|X^k_r\|^\beta_\H ) + \frac{1}{n} \sum_{i = 1}^n \|X^i_r\|^\beta_\H \Big) \, dr.
	\end{align*}
Thus, recalling \(\beta \leq 2 \eta\) and \(\N_{\beta /2 + 1} (X_r^k) = \|X_r^k\|^\beta_\H\, \N (X_r^k)\), we deduce from \(P \circ \Z_n^{-1} \in \J (x)\) that 
\begin{align*}
	\frac{1}{n} \sum_{k = 1}^n E^P & \Big[ \sup_{s \not = t} \int_s^t \int \|b (f, X^k_r, \x_n (X_r))\|_\V \, \v (r, M^k, df) \, dr \Big \slash |s - t|^{1 - 1 /\gamma} \Big] \leq C, 
\end{align*}
where the constant only depends on the boundedness of \(B\) and the function \(\j\) that appears in the definition of \(\J (x)\). 
By Burkholder's inequality and \eqref{eq: cond growth diffusion}, we obtain that 
\begin{align*}
	E^P \Big[ \Big\| \int_s^t \us (\v (r&, M^k), r, \oX^k_r, \x_n (X_r)) \, d W^k_r \Big\|^{2 \eta}_\H \, \Big] 
	\\&\leq CE^P \Big[ \Big( \int_s^t \| \us (\v (r, M^k), r, \oX^k_r, \x_n (X_r)) \|^2_{L_2 (\U; \H)} \, dr \Big)^{\eta} \, \Big] 
	\\&\leq CE^P \Big[ \Big( \int_s^t \lambda\, \Big( 1 + \|X^k_r\|^2_\H + \frac{1}{n} \sum_{i = 1}^n \|X^i_r\|^2_\H \Big) \, dr \Big)^{\eta}\, \Big]
	\\&\leq  C|s - t|^\eta \, 3^\eta \lambda^\eta \, \Big( 1 + E^P \Big[ \sup_{r \in [0, T]} \|X^k_r\|^{2\eta}_\H + \frac{1}{n} \sum_{i = 1}^n \sup_{r \in [0, T]} \|X^i_r\|^{2 \eta}_\H \Big]\Big).
\end{align*}
Consequently, using that \(P \circ \Z_n^{-1} \in \J(x)\), there exists a constant \(C > 0\), with the same dependencies as above, such that 
\begin{align*}
	\frac{1}{n} \sum_{k = 1}^n E^P \Big[ \Big\| \int_s^t \us (\v (r, M^k), r, \oX^k_r, \x_n (X_r)) \, d W^k_r \Big\|^{2 \eta}_\H\, \Big] &\leq C |s - t|^{\eta}. 
\end{align*}
Using Jensen's inequality and the Besov--H{\"o}lder embedding (see Lemma~\ref{appendix: besov} in Appendix~\ref{sec: appendix}; apply it with \(q = 2 \eta\) and \(\alpha = (\eta - 1)/2 \eta\), where \(\alpha > 1/q\) is equivalent to \(\eta > 2\)) yields that 
\begin{align*}
	\frac{1}{n} \sum_{k = 1}^n & E^P \Big[  \sup_{s \not = t} \Big\| \int_s^t \us (\v (r, M^k), r, \oX^k_r, \x_n (X_r)) \, d W^k_r \Big\|_\H \Big \slash |s - t|^{(\eta - 2) / 2 \eta} \Big] 
	\\&\leq \Big( \frac{1}{n} \sum_{k = 1}^n \, E^P \Big[  \sup_{s \not = t} \Big\| \int_s^t \us (\v (r, M^k), r, \oX^k_r, \x_n (X_r)) \, d W^k_r \Big\|_\H^{2 \eta} \Big \slash |s - t|^{\eta - 2} \Big] \Big)^{1 / 2 \eta}
	\\&\leq \Big( \frac{1}{n} \sum_{k = 1}^n E^P \Big[ \int_0^T \int_0^T \Big\| \int_s^t \us (\v (r, M^k), r, \oX^k_r, \x_n (X_r)) d W^k_r \Big\|^{2\eta}_\H \Big \slash |s - t|^{\eta} \, dt \, ds \Big] \Big)^{1/ 2 \eta}
	\\&\leq C T^{1 / \eta}. \phantom {\int_0^T}
\end{align*}
We conclude that \eqref{eq: to show for tightness} holds with \(\ell = (1 - 1 / \gamma) \wedge (\eta - 2)/ 2\eta\). 

\smallskip
\noindent
{\em Step 2: Relative compactness in \(\fP^\al (\fP^\al (\Theta))\).} 
The following is implied by \cite[Lemma~C.1]{GRZ}: there exists a constant \(C > 0\) such that 
	\begin{align} \label{eq: C1}
	\|x\|_\Y^\alpha \leq C\, ( \N (x) + \|x\|^\alpha_\V ), \quad \forall \, x \in \Y  \text{ with } \N(x) < \infty. 
	\end{align} 
Hence, using the definition of \(\J\), we obtain that
\begin{align*}
\sup_{Q \in \cR (B)} E^{I (Q)} \big[ \d (X, 0)^\alpha \big] &\leq C \, \sup_{Q \in \cR (B)} E^{I (Q)} \Big[ \sup_{s \in [0, T]} \| X_s\|^\alpha_\V + \int_0^T \| X_s \|^\alpha_\Y \, ds \Big] 
\\&\leq C \, \sup_{Q \in \cR (B)} E^{I (Q)} \Big[ \sup_{s \in [0, T]} \| X_s\|^\alpha_\V + \int_0^T \big[ \N (X_s) + \|X_s\|^\alpha_\V \big] \, ds \Big] 
\\&\leq C \, \sup_{Q \in \cR (B)} E^{I (Q)} \Big[ \sup_{s \in [0, T]} \| X_s\|^\alpha_\V + \int_0^T \N (X_s) \, ds \Big] 
< \infty. 
\end{align*}
	By Lemma~\ref{appendix: tight wasserstein} in Appendix~\ref{sec: appendix}, the compactness of \(\m\) and Step~1, we conclude that \(\cR (B)\) is relatively compact in \(\fP^\al (\fP^\al (\Theta))\). The proof is complete.
\end{proof}

\begin{lemma} \label{lem: M closed}
	Let \((x^n)_{n = 0}^\infty \subset \H\) be such that \(x^n \to x^0\) in \(\|\,\cdot\,\|_\H\). Furthermore, let \((Q^n)_{n = 0}^\infty \subset \fP (\fP (\Theta))\) be such that \(Q^n \in \J (x^n)\) for all \(n \in \mathbb{N}\) and \(Q^n \to Q^0\) in \(\fP (\fP (\Theta))\). Then, \(Q^0 \in \J (x^0)\). 
\end{lemma}
\begin{proof} 
 First, we claim that \((t, \omega) \mapsto \omega (t)\) is continuous from \([0, T] \times \Omega\) into \(\V\). To see this, take a sequence \((t^n, \omega^n)_{n = 0}^\infty \subset [0, T] \times \Omega\) such that \((t^n, \omega^n) \to (t^0, \omega^0)\). By our choice for the topology of~\(\Omega\), \(\{\omega^n \colon n \in \mathbb{N}\}\) is a relatively compact subset of \(C ([0, T]; \V)\), endowed with the uniform topology. Therefore, we deduce from the Arzel\`a--Ascoli theorem (see Lemma~\ref{appendix: arzela} in Appendix~\ref{sec: appendix}) that 
		\[
		\lim_{h \to 0} \sup_{n \in \mathbb{N}} \sup \Big\{ \| \omega^n (s) - \omega^n (t) \|_\V \colon s, t \in [0, T], \, |s - t| \leq h \Big\} = 0.
		\]
	Consequently, as \(|t^n - t^0| \to 0\), we obtain that 
	\[
	\| \omega^n (t^n) - \omega^0 (t^0) \|_\V \leq \| \omega^n (t^n) - \omega^n (t^0)\|_\V + \| \omega^n (t^0) - \omega^0 (t^0) \|_\V \to 0
	\]
	as \(n \to \infty\). This proves the continuity of \((t, \omega) \mapsto \omega (t)\).
As a consequence, \((t, \omega) \mapsto \|\omega (t)\|_\H\) is lower semicontinuous from \([0, T] \times \Omega\) into \([0, \infty]\) (as \(x \mapsto \|x\|_\H\) is lower semicontinuous from \(\V\) into \([0, \infty]\); recall our convention that \(\|x\|_\H = \infty\) for \(x \in \V\setminus \H\)). By Berge's maximum theorem (see Lemma~\ref{appendix: berge} in Appendix~\ref{sec: appendix}), the same is true for the map \((t, \omega) \mapsto \sup_{s \in [0, t]} \| \omega(s)\|_\H\).
By Lemma~\ref{appendix: lower semi} in Appendix~\ref{sec: appendix}, 
\[
\omega \mapsto \int_0^T \big[ \N (\omega (s)) + \N_{\beta /2 + 1} (\omega (s)) \big] \, ds 
\]
is lower semicontinuous from \(\Omega\) into \([0, \infty]\). 
Summing up, as the class of \([0, \infty]\)-valued lower semicontinuous functions is stable under summation, the map 
	\[
	\omega \mapsto \sup_{s \in [0, t]} \| \omega (s)\|_\H + \int_0^T \big[ \N (\omega(s)) + \N_{\beta / 2 + 1} (\omega (s)) \big] \, ds 
	\] 
	is lower semicontinuous from \(\Omega\) into \([0, \infty]\). By Lemma~\ref{appendix: lower semi} in Appendix~\ref{sec: appendix}, we conclude that
	\[
	\mu \mapsto E^\mu \Big[ \sup_{s \in [0, T]} \|X_s\|^{2\eta}_\H + \int_0^T \big[ \N (X_s) + \N_{\beta /2 + 1} (X_s) \big] \, ds \Big]
	\]
	is lower semicontinuous from \(\mathcal{P} (\Theta)\) into \([0, \infty]\).

Using once again Lemma~\ref{appendix: lower semi} in Appendix~\ref{sec: appendix}, the assumption that \(Q^n \in \J (x^n)\), and the continuity of \(x \mapsto \j (x)\), it follows that
\begin{align*}
	\int E^{\mu} \Big[ \sup_{s \in [0, T]} & \|X_s\|^{2\eta}_\H + \int_0^T \big[ \N (X_s) + \N_{\beta /2 + 1} (X_s) \big] ds \Big] \, Q^0 (d \mu)
	\\&\leq \liminf_{n \to \infty} 	\int E^{\mu} \Big[ \sup_{s \in [0, T]} \|X_s\|^{2 \eta}_\H + \int_0^T \big [ \N (X_s) + \N_{\beta /2 + 1} (X_s) \big] ds \Big] \, Q^n (d \mu)
	\\&\leq \liminf_{n \to \infty} \j (x^{n})	
	= \j ( x^0).  \phantom{\int}
\end{align*}
We conclude that \(Q^0 \in \J (x^0)\). 
\end{proof}

\begin{lemma} \label{lem: C comp}
	Assume that Condition~\ref{cond: main1} holds.
	For every compact set \(\mathscr{K} \subset \H\),
	the set
 	\begin{align*}
 		\cR^0 (\mathscr{K}) := \bigcup_{x \in \mathscr{K}} \cR^0 (x)
 	\end{align*}
 is compact in \(\fP^\al (\fP^\al (\Theta))\). 
\end{lemma}
\begin{proof}
	
	{\em Step 1.} We first show that \(\cR^0 (\mathscr{K})\) is relatively compact in \(\fP^\al (\fP^\al (\Theta))\). Here, we argue along the lines of the proof for Lemma~\ref{lem: Rk rel comp}. Recall the definition of the map \(I\) from \eqref{eq: I}. First, we prove relative compactness of \(\mathcal{J} := \{ I (Q) \colon Q \in \cR^0(\mathscr{K})\}\) in \(\mathcal{P}(\Omega)\). Take \(Q \in \cR^0(x)\) with \(x \in \mathscr{K}\).
	By the definition of \(\cR^0(x)\), \(Q\)-a.a. \(\mu \in \fP (\Theta)\) are elements of \(\cC^0(x)\). 
	Now, we can argue as in the proof of Lemma~\ref{lem: Rk rel comp} to conclude that
	\begin{equation*} \begin{split}
		E^{I (Q)} \Big[ \sup_{s \not = t} & \frac{\|X_t - X_s\|_\V}{|t - s|^{(1 - 1/ \gamma) \wedge (\eta - 2) / 2\eta}} \Big] 
		\\&= \int E^{\mu} \Big[ \sup_{s \not = t} \frac{\|X_t - X_s\|_\V}{|t - s|^{(1 - 1/ \gamma) \wedge (\eta  - 2) / 2 \eta}} \Big]  \, Q (d \mu)
		\\&\leq C \Big( 1 + \int E^\mu \Big[ \sup_{s \in [0, T]} \|X_s\|^{2 \eta}_\H + \int_0^T \N (X_s) ds +  \int_0^T \N_{\beta/2 + 1}(X_s) ds \Big] \, Q (d \mu) \Big).
	\end{split}
\end{equation*}
By the definition of \(\J\), the properties of \(\j\) and the boundedness of \(\mathscr{K}\), there exists a constant independent of \(Q\) such that 
\[
E^{I (Q)} \Big[ \sup_{s \not = t} \frac{\|X_t - X_s\|_\V}{|t - s|^{(1 - 1/ \gamma) \wedge (\eta - 2) / 2 \eta}} \Big]  \leq C.
\] 
Using the same facts again, it follows from Condition~\ref{cond: main1}~(i) and Lemma~\ref{appendix: tight GRZ} that \(\mathcal{J}\) is relatively compact in~\(\fP (\Omega)\). 
Finally, relative compactness in \(\fP^\al (\fP^\al (\Theta))\) follows from Lemma~\ref{appendix: tight wasserstein} in Appendix~\ref{sec: appendix}, taking \eqref{eq: C1} and, once again, the definition of \(\J\) into consideration.
	
	\smallskip
	\noindent
	{\em Step 2.} In this step, we prove that \(\cR^0 (\mathscr{K})\) is closed, where we use the martingale problem characterization of \(\cC^0\) as given by Lemma~\ref{lem: mg chara}.
	Take a sequence \((Q^n)_{n = 1}^\infty\subset \cR^0(\mathscr{K})\) such that \(Q^n \to Q^0\) in \(\fP^\al (\fP^\al (\Theta))\). By definition of \(\cR^0 (\mathscr{K})\), there exists a sequence \((x^n)_{n = 1}^\infty \subset \mathscr{K}\) such that \(\mu \circ X^{-1}_0 = \delta_{x^n}\) for \(Q^n\)-a.a. \(\mu \in \fP(\Theta)\). By the compactness of \(\mathscr{K}\), possibly passing to a subsequence, we may assume that \(x^n \to x^0 \in \mathscr{K}\) in \(\|\cdot\|_\H\). 
	
	Take \(\varepsilon > 0\) and set \(G = G (\varepsilon) := \{ Q \in \W (\Theta) \colon \mathsf{p}_{\V} (Q \circ X^{-1}_0, \delta_{x^0}) \leq \varepsilon \}\), where \(\mathsf{p}_{\V}\) denotes some metric on \(\fP (\V)\) that induces the weak topology.
	As the map \(Q \mapsto \p_{\V} (Q \circ X^{-1}_0, \delta_{x_0})\) is continuous from \(\fP(\Theta)\) into \(\mathbb{R}_+\), the set \(G\) is closed in \(\fP(\Theta)\). Hence, by the Portmanteau theorem, 
	\begin{align*}
		Q^0 (G) &\geq \limsup_{n \to \infty} Q^n (G) 
		= \limsup_{n \to \infty} \1 \{ \mathsf{p}_{\V} (\delta_{x^n}, \delta_{x^0} ) \leq \varepsilon \} = 1.
	\end{align*}
	As \(\varepsilon > 0\) was arbitrary, it follows that 
	\begin{align*}
		Q^0 ( \{ \mu \in \W(\Theta) \colon \mu\circ X^{-1}_0 = \delta_{x^0} \} ) = 1, 
	\end{align*}
	that is almost all realizations of \(Q^0\) satisfy (a.i) from Lemma~\ref{lem: mg chara} with initial value \(x^0\).
	
	We deduce from Lemma~\ref{lem: M closed} that \(Q^0 \in \J (x^0)\). In particular, this implies that \(\Q^0 (\NN) = 1\), i.e., (a.ii) from Lemma~\ref{lem: mg chara} holds for \(Q^0\)-a.a. realizations.

	It remains to show that \(Q^0\)-a.a. \(\mu \in \fP (\Theta)\) satisfy (a.iii) from Lemma~\ref{lem: mg chara}. Take \(g \in C^2_c(\bR; \bR)\) and \(v \in \V^*\).
	For \(k > 0\) and \(r \in [0, T]\), we define \(\oM_r^k \colon \Theta \times \fP (\Theta) \to \bR\) by 
	\begin{align*}
		\oM_r^k (\omega, m, \mu) := g ({_\V}\langle \omega(r), v \rangle_{\V^*}) - \int_0^r\int \, \Big[ (- k) \vee \mathcal{L}_{g, v} (f, s, \omega(s), \mu_s) \wedge k \Big] \,m (ds, df), 
	\end{align*}
	where \(\mu_s := \mu \circ X^{-1}_s\).
	\begin{lemma}\label{lem: oM mit k cont}
		Suppose that Condition~\ref{cond: main1}~\textup{(ii)} holds. 
		For every \(k > 0\) and \(r \in [0, T]\), the map \(\oM^k_r\) is continuous from \(\Theta \times \NN\) into \(\bR\), where we endow \(\NN\) with the subspace topology coming from the Wasserstein space \(\fP^\al (\Theta)\).
	\end{lemma}
	\begin{proof}
		Take a sequence \((\omega^n, m^n, \mu^n)_{n = 0}^\infty \subset \Theta \times \NN\) with
		\(
		(\omega^n, m^n, \mu^n) \to (\omega^0, m^0, \mu^0)
		\) and notice that 
		\begin{align*}
			\big| \oM^k_r (\omega^n&, m^n, \mu^n) - \oM^k_r (\omega^0, m^0, \mu^0) \big| 
			\\&\leq \big| g ({_\V}\langle \omega^n(r), v\rangle_{\V^*}) - g ({_\V}\langle \omega^0 (r), v \rangle_{\V^*}) \big|
			\\&\qquad+  \int_0^r \sup_{f \in F} \big| (-k) \vee \mathcal{L}_{g, v} (f, s, \omega^n (s), \mu^n_s) \wedge k  
			- (-k) \vee \mathcal{L}_{g, v} (f, s, \omega^0 (s), \mu^0_s) \wedge k\big| \, ds 
			\\&\qquad + \Big| \int_0^r \int (- k) \vee \mathcal{L}_{g, v} (f, s, \omega^0 (s), \mu^0_s) \wedge k\, \big( m^n (ds, df) - m^0 (ds, df) \big)\Big|
			\\&=: I_n + II_n + III_n.
		\end{align*} 
		First, as \(\omega^n \to \omega^0\) in \(\Omega\), we have \(\omega^n (r) \to \omega^0 (r)\) in \(\V\), which immediately implies that \(I_n \to 0\).
		
		Next, \(\omega^n \to \omega^0\) in \(\Omega\) also implies that \(\|\omega^n - \omega^0\|_\Y \to 0\) in Lebesgue measure. 
		By the continuity of \([0, T] \times \Omega \ni (s, \omega) \mapsto \omega (s) \in \V\), it follows that
		\(\mu^n_s \to \mu^0_s\) in \(\fP^{2 \eta}_\al (\H)\) for each \(s \in [0, T]\). Using Condition~\ref{cond: main1}~(ii) and Berge's maximum theorem (see Lemma~\ref{appendix: berge} in Appendix~\ref{sec: appendix}), we obtain that the integrand in \(II_n\) converges to zero in Lebesgue measure. As it is bounded (by \(2k\)), the dominated convergence theorem yields that \(II_n \to 0\). 
		
		Finally, as the integrand in \(III_n\) is continuous in the \(F\)-variable for each fixed time point, \(III_n \to 0\) follows from Lemma~\ref{appendix: stable} in Appendix~\ref{sec: appendix}. The proof is complete.
	\end{proof}
	
	Take \(0 \leq s < t \leq T\) and \(\test \in \tT_s\).
	For \(\mu \in \fP (\Theta)\), define 
	\[
	\oZ^k (\mu) := \iint \big[ \oM^k_t (\omega, m, \mu) - \oM^k_s (\omega, m, \mu) \big]  \test (\omega, m) \, \mu (d \omega, dm),
	\]
	and 
	\[
	\oZ (\mu) := \liminf_{k \to \infty} \oZ^k (\mu).
	\]
	By Lemma~\ref{appendix: bogachev} in Appendix~\ref{sec: appendix}, \(\oZ^k\) and \(\oZ\) are Borel maps. 
	Thanks to \eqref{eq: cond growth drift}, we obtain that
	\begin{equation} \label{eq: bound}
		\begin{split}
			\big| \mathcal{L}_{g, v} (f, r, y, \mu_r) \big|^{\gamma} \leq C\, \Big( (1 + \N (y)) (1 + \|y\|^\beta_\H) + E^\mu \big [ \|X_r\|^\beta_\H \big] \Big)
		\end{split}
	\end{equation}
for all \(f \in F, r \in [0, T], y \in \Y\) and \(\mu \in \mathcal{P} (\Theta)\). 
	As \(Q^n \in \J (x^n)\) and \(Q^0 \in \J(x^0)\), this shows that, for \(Q^n\)-a.a. and \(Q^0\)-a.a. \(\mu \in \fP (\Theta)\),
	\[
	E^\mu \Big[ \int_0^T \int \big| \mathcal{L}_{g, v} (f, r, X_r, \mu_r) \big|^{\gamma} \, M (dr, df)\Big] < \infty. 
	\]	
	Hence, by the dominated convergence theorem (applied for the product space \([0, T] \times F \times \Theta\) with the finite measure \(M (dr, df) d \mu\)), we obtain that, for \(Q^n\)-a.a. and \(Q^0\)-a.a. \(\mu \in \fP (\Theta)\),
	\begin{align*} 
	\lim_{k \to \infty} E^\mu \Big[ \int_s^t \int \Big[ (- k) \vee \mathcal{L}_{g, v} &(f, r, X_r, \mu_r) \wedge k \Big] \, M (dr, df) \, \test \, \Big] 
	\\&= E^\mu \Big[ \int_s^t \int \mathcal{L}_{g, v} (f, r, X_r, \mu_r) \, M (dr, df) \, \test \, \Big], 
	\end{align*} 
	and consequently, also 
	\begin{align} \label{eq: oZ after bound}
	\oZ (\mu) = E^\mu \Big[ \Big( g ( {_\V}\langle X_t , v \rangle_{\V^*} ) - g ({_\V}\langle X_s, v \rangle_{\V^*} ) - \int_s^t \int \mathcal{L}_{g, v} (f, r, \oX_r, \mu_r) \, M (dr, df) \Big)\, \test\ \Big]. 
	\end{align}
 In the following, we prove that 
	\[
	E^{Q^0} \big[ | \oZ | \big] = 0.
	\]
	As \(\B, \tC\) and \(\tT_s\) are countable, this implies that \(Q^0\)-a.a. \(\mu \in \fP (\Theta)\) satisfy (a.iii) from Lemma~\ref{lem: mg chara}, completing the proof. 
	
	Notice that \(Q^n\)-a.s. \(\oZ = 0\), as \(Q^n (\cC^0 (x^n)) = 1\). Thus,
		by the triangle inequality, we observe that 
	\begin{equation} \label{eq: main triangle}
		\begin{split}
			E^{Q^0} \big [ | \oZ | \big] &\leq E^{Q^n} \big[ | \oZ^k- \oZ | \big]  + \big| E^{Q^n} \big [ | \oZ^k | \big ] - E^{Q^0} \big[ | \oZ^k | \big] \big|  + \big	| E^{Q^0} \big [ | \oZ^k | \big ] - E^{Q^0} \big [ | \oZ | \big ]\big |
			\\&=: I_{n, k} + II_{n, k} + III_{k}.
		\end{split}
	\end{equation}

Lemma~\ref{lem: oM mit k cont} and Lemma~\ref{appendix: bogachev} in Appendix~\ref{sec: appendix} show that \(\mathscr{G} \ni \mu \mapsto \oZ^k (\mu)\) is bounded and continuous and consequently, \(II_{n, k} \to 0\) as \(n \to \infty\) for every \(k > 0\). 
Using \eqref{eq: bound}, we obtain 
\begin{align*}
	I_{n, k} &\leq \int E^\mu \Big[ \int_0^T \int \big| \mathcal{L}_{g, v} (f, r, \oX_r, \mu_r) \big| \1_{\{|\mathcal{L}_{g, v} (f, r, \oX_r, \mu_r)| \, >\, k\}} \, M(dr, df) \Big] \, Q^n (d\mu)
	\\&\leq \frac{C}{k^{\gamma - 1}}, 
\end{align*}	
where \(C > 0\) depends on \(\j\) from \(\J\) but is independent of \(n\). 
Similarly, we obtain that 
\[
	III_{k} \leq \frac{C}{k^{\gamma - 1}}.
\]
Hence, \(I_{n, k} + III_{k}\to 0\) as \(k \to \infty\) uniformly in \(n\). 
Thus, choosing first a large \(k\) and taking then \(n \to \infty\) shows that \(I_{n, k} + II_{n, k} + III_{k}\) can be made arbitrarily small, which entails that \(Q^0\)-a.s. \(\oZ = 0\). In summary, \(Q^0\)-a.a. realizations satisfy (iii) from Lemma~\ref{lem: mg chara} and consequently, \(Q^0 \in \cR^0 (x^0) \subset \cR^0 (\mathscr{K})\). The proof is complete.
	\end{proof}

In the following, we endow \(\Theta^n\) with the metric formed by adding the metrics of \(\Theta\), i.e., 
\[
((\theta^1_1, \dots, \theta^1_n), (\theta^2_1, \dots, \theta^2_n)) \mapsto \sum_{k = 1}^n m_\Theta (\theta^1_k, \theta^2_k), \quad m_\Theta = \d + \r.
\]
We record another observation. 
\begin{lemma} \label{lem: cCN compact}
	Suppose that Condition \ref{cond: main1} holds.
	For every \(x \in \H\) and \(n \in \mathbb{N}\), the sets \(\cC^n (x)\) and \(\cR^n (x)\) are compact in \(\fP^\al (\Theta^n)\) and \(\fP^\al (\fP^\al (\Theta))\), respectively.
\end{lemma}
\begin{proof}
Similar to Lemma~\ref{lem: Rk rel comp}, one proves that the set \(\cC^n (x)\) is relatively compact in \(\fP^\al (\Theta^n)\). Further, a martingale problem argument (related to Lemma~\ref{lem: mg chara N}, see also the proof of Lemma~\ref{lem: C comp}) shows that \(\cC^n (x)\) is closed in \(\fP^\al (\Theta^n)\). We omit the details for brevity. In summary, \(\cC^n (x)\) is compact. These claims transfer directly to \(\cR^n (x)\) by the continuity of \(P \mapsto P \circ \Z_n^{-1}\) from \(\fP^\al (\Theta^n)\) into \(\fP^\al (\fP^\al(\Theta))\), which follows from Lemma~\ref{appendix: wasserstein cont} in Appendix~\ref{sec: appendix}.
\end{proof}

\subsection{An existence result}  \label{sec: existence}
Fix an \(N \in \mathbb{N}\) and a filtered probability space \((\Sigma, \mathcal{A}, (\mathcal{A}_t)_{t \in [0, T]}, P)\) that supports independent cylindrical Brownian motions \(W^1, \dots, W^N\) over \(\U\) and predictable kernel \(\q^1, \dots, \q^N\) from \([0, T] \times \Sigma\) into \(F\).
We define the product setup
\[
\Psi := \Sigma \times \Omega^N, \quad \mathcal{G} := \mathcal{A} \otimes \mathcal{F}^N, \quad \mathcal{G}_t := \bigcap_{s > t}\,(\mathcal{A}_s \otimes \mathcal{F}^N_s). 
\]
With little abuse of notation, let \(X = (X^1, \dots, X^N)\) be the projection to the second coordinate, and extend \(W^1, \dots, W^N\) in the obvious manner to the product setup.
The following can be seen as an extension of \cite[Theorem~4.6]{GRZ} to a setting with random coefficients. The result is very much in the spirit of the seminal paper \cite{jacod1981weak} whose ideas we also adapt.

\begin{proposition} \label{prop: existence}
Assume that Condition~\ref{cond: main1} holds and fix \(x \in \H\). Then, there exists a probability measure \(Q\) on \((\Psi, \mathcal{G}, (\mathcal{G}_{t})_{t \in [0, T]})\) such that \(W^1, \dots, W^N\) are independent cylindrical Brownian motions under \(Q\) and \(Q\)-a.s., for all \(k = 1, \dots, N, t \in [0, T]\) and \(v \in \V^*\), 
\begin{equation} \label{eq: existence identity} \begin{split}
    {_\V}\langle X^k_t, v \rangle_{\V^*} =  {_\V}\langle x, v \rangle_{\V^*} &+ \int_0^t \blv \Big \langle \int b (f, s, \oX^k_s, \x_N (X_s)) \q^k_s (df), v \Big \rangle_{\V^*} \, ds
\\&+ \Big \langle \int_0^t \us (\q_s^k, s, \oX^k_s, \x_N (X_s)) d W^k_s, v \Big\rangle_\H.
\end{split} 
\end{equation}
Moreover, \(Q \circ \Z_N (X, K)^{-1} \in \J (x)\) (where we take \(K := (\delta_{\q^1_t} dt, \dots, \delta_{\q^N_t} dt)\)) and there exists a probability transition kernel \(\mathscr{Q}\) such that 
\begin{align}\label{eq: kernel dec}
Q (d z, d \omega) = \mathscr{Q} (z, d \omega) P (dz)
\end{align}
and \(\mathscr{Q} (\, \cdot\,, G)\) is \(P\)-a.s. \(\mathcal{A}_t\)-measurable for every \(G \in \mathcal{F}^N_t\). 
\end{proposition}

The proof of Proposition~\ref{prop: existence} is technical, though many steps resemble those in standard settings without random coefficients. The core idea is to construct an approximating sequence, establish its tightness in a suitable topology, and identify an arbitrary accumulation point as solution. The main difficulties stem from the infinite-dimensional setting and the requirement to work on a fixed probability space to accommodate random coefficients. 
	
	The proof proceeds in six steps. In Step 1, we project \(b\) and \(\bar{\sigma}\) to finite-dimensional spaces and render them bounded via cut-off functions. In Step 2, the classical existence theorem of Jacod--M{\'e}min yields solutions to the finite-dimensional stochastic equations corresponding to these approximating coefficients. In Step 3, we establish tightness of this approximating sequence in the so-called weak-strong topology (defined in the preliminary Step 0). In Step 4, we initiate a martingale problem argument to show that every accumulation point solves the desired equation. Step 5 is dedicated to proving the claim \(Q \circ \Z_N (X, K)^{-1} \in \J (x)\), which is first established for the approximating sequence. Finally, in Step 6, we conclude the martingale problem argument and pass the claim from Step 5 to the limit.

\begin{proof}[Proof of Proposition~\ref{prop: existence}]
	The proof of this proposition is based on ideas from \cite[Theorem~4.6]{GRZ} and \cite[Theorem~1.8]{jacod1981weak}.
	
	\smallskip\noindent
{\em Step 0:} Recall from Appendix~\ref{sec: app ws conv} that a sequence \((Q^n)_{n = 1}^\infty\) of probability measures on \((\Psi, \mathcal{G})\) is said to converge in the weak-strong sense to a probability measure \(Q^0\) on \((\Psi, \mathcal{G})\) if 
\[
E^{Q^n} \big[ g \big] \to E^{Q^0} \big[ g \big]
\]
for all bounded Carath\'eodory functions from \(\Psi\) into \(\bR\), i.e., all bounded functions \(\Psi \to \bR\) that are \(\mathcal{G}\)-measurable in the \(\Sigma\)-variable and continuous in the \(\Omega^N\)-variable. 

\smallskip 
\noindent
{\em Step 1:} Let \(\{\ell_n \colon n \in \mathbb{N}\} \subset \V^*\) be an orthonormal basis of \(\H\) such that 
\[
\| \Pi_n x \|_\V \leq C \|x\|_\V, \quad \forall \, n \in \mathbb{N}, x \in \V, 
\]
where 
\[
\Pi_n x := \sum_{k = 1}^n {_\V} \langle x, \ell_k \rangle_{\V^*} \ell_k, \quad x \in \V.
\]
Such an orthonormal basis exists by \cite[Lemma~4.4]{GRZ}. Furthermore, let \(\{ g_n \colon n \in \mathbb{N}\} \subset \U\) be an orthonormal basis of~\(\U\). 
Now, set 
\begin{align*}
\H_n &:= \on{span} \, \{ \ell_1, \dots, \ell_n \} \subset \V^* \subset \Y \subset \H \subset \V, \\ 
\U_n &:= \on{span} \, \{ g_1, \dots, g_n\} \subset \U.
\end{align*}
We define the coefficients \(b_n^k \colon \Sigma \times [0, T] \times \H_n^N \to \H_n\) and \(\sigma_n^k \colon \Sigma \times [0, T] \times \H_n^N \to L (\U_n; \H_n)\) by 
\begin{align*}
b_n^k (w, t, x_1, \dots, x_N) &:= \Pi_n\, \int b (f, t, x_k, \x_N (x)) \, \q^k_t (df) (w), \\
\sigma_n^k (w, t, x_1, \dots, x_N) &:= \Pi_n\, \us (\q^k_t (w), t, x_k, \x_N (x)),
\end{align*}
where \(x = (x_1, \dots, x_N) \in \H_n^N\). 
Let us summarize some properties of \(b_n^k\) and \(\sigma^k_n\). First, for every \((w, t) \in \Sigma \times [0, T]\), the maps 
\[
x \mapsto b^k_n (w, t, x), \quad x \mapsto \sigma^k_n (w, t, x)
\]
are continuous by Condition~\ref{cond: main1}~(ii). Moreover, thanks to Condition~\ref{cond: main1}~(iii), we also have 
\begin{equation} \label{eq: coercivity finite}
	\begin{split}
\langle b^k_n (w, t, x), x_k \rangle_{\H_n}
&\leq \lambda \Big( 1 + \|x_k\|_{\H_n}^2 + \frac{1}{N} \sum_{i = 1}^N \|x_i\|^2_{\H_n} \Big) - \N (x_k), \\
	\| \sigma^k_n (w, t, x) \|_{L_2 (\U_n; \H_n)}^2 &\leq \lambda \Big( 1 + \|x_k\|_{\H_n}^2 + \frac{1}{N} \sum_{i = 1}^N \|x_i\|^2_{\H_n} \Big).
\end{split}
\end{equation}
For \(m > 0\), let \(\phi_m \in C^\infty_c (\bR; [0, 1])\) be a cutoff function such that 
\[
\phi_m (y) = \begin{cases} 1, & |y| \leq m, \\ 0, & |y| \geq 2m.\end{cases}
\]
We also set 
\begin{align*}
	b^k_{n, m} (w, t, x) := \phi_m (\|x\|_{\H_n^N})\, b^k_n (w, t, x), \quad \sigma^k_{n, m} (w, t, x) := \phi_m (\|x\|_{\H_n^N})\, \sigma^k_n (w, t, x), 
\end{align*}
where \(\|x\|_{\H_n^N} = \sum_{i = 1}^N \|x_i\|_{\H_n}\). 
Notice that \(b^k_{n, m}\) and \(\sigma^k_{n, m}\) have the same continuity properties as \(b^k_n\) and \(\sigma^k_n\). Furthermore, \(b^k_{n, m}\) and \(\sigma^k_{n, m}\) are bounded, which follows from Condition~\ref{cond: main1}~(ii). 

\smallskip
\noindent
{\em Step 2:} 
By Lemma~\ref{appendix: existence SDE} in Appendix~\ref{sec: appendix}, there exists a probability measure \(Q^{n, m}\) on the filtered space \((\Psi, \mathcal{G}, (\mathcal{G}_{t})_{t \in [0, T]})\) such that \(Q^{n, m}\) admits a decomposition of the type \eqref{eq: kernel dec}, including the measurability property explained below \eqref{eq: kernel dec}, \(W^1, \dots, W^N\) are independent cylindrical Brownian motions under \(Q^{n, m}\) (which follows from the kernel decomposition, see Remark~\ref{appendix: rem factorization}~(a) in Appendix~\ref{sec: appendix}) and \(Q^{n, m}\)-a.s., for all \(k = 1, \dots, N\) and~\(t \in [0, T]\), 
\begin{align*}
X^k_t = \Pi_n x + \int_0^t b^k_{n, m} (s, X_s) \, ds + \int_0^t \sigma^k_{n, m} (s, X_s) \, d W^{k, n}_s,  
\end{align*}
where
\[
W^{k, n} := \sum_{i = 1}^n\, \langle W^k, g_i \rangle_\U \, g_i.
\]
In particular, each \(X^k\) is \(Q^{n, m}\)-a.s. \(\H_n\)-valued (in particular, finite dimensional). 

\smallskip
\noindent
{\em Step 3:} 
Thanks to \eqref{eq: coercivity finite}, standard arguments (see the proof of \cite[Theorem~C.3]{GRZ} for details) yield that the set \(\{Q^{n, m} \circ X^{-1} \colon m \in \mathbb{N}\} \subset \fP (C ([0,T]; \H^N_n))\) is relatively compact in \(\mathcal{P}(C ([0, T]; \H^N_n))\).
Consequently, as the \(\Sigma\)-marginal of each measure from \((Q^{n, m})_{m = 1}^\infty\) coincides with \(P\), Lemma~\ref{appendix: weak strong relative comp} in Appendix~\ref{sec: appendix} yields the existence of a subsequence \((Q^{n, N_m})_{m = 1}^\infty\) that converges in the weak-strong sense to a probability measure~\(Q^{n}\). In view of Remark~\ref{appendix: rem factorization}~(b) in Appendix~\ref{sec: appendix}, the weak-strong limit \(Q^n\) also admits a decomposition of the form~\eqref{eq: kernel dec}, including the measurability properties described below, and, by Remark~\ref{appendix: rem factorization}~(a) in Appendix~\ref{sec: appendix}, the processes \(W^1, \dots, W^N\) remain independent cylindrical Brownian motions under \(Q^n\).

\smallskip
\noindent
{\em Step 4:} 
For \(g \in C^2_b (\bR; \bR)\), \(u \in \U_n\) and \(h \in \H_n\), we set 
\begin{align*}
\mathsf{K}^k_{n, m} &(w, \omega, t) 
\\&:= g ( \langle W^{k, n}_t (w), u \rangle_{\U_n} + \langle X^k_t (\omega), h \rangle_{\H_n} ) - g( \langle W^{k, n}_0 (w), u \rangle_{\U_n} + \langle X^k_0 (\omega), h \rangle_{\H_n} ) \\&\qquad- \int_0^t \Big( g' ( \langle W^{k, n}_s (w) , u \rangle_{\U_n} + \langle X^k_s (\omega), h \rangle_{\H_n} ) \langle b^k_{n, m} (w, s, X_s (\omega)), h \rangle_{\H_n}
\\&\qquad\hspace{0.75cm}+ \frac{1}{2}\, g'' ( \langle W^{k, n}_s (w), u \rangle_{\U_n} + \langle X^k_s (\omega), h \rangle_{\H_n} ) \| u + (\sigma^k_{n, m})^* (w, s, X_s (\omega)) h \|^2_{\U_n} \Big) \, ds.
\end{align*}
Notice that, for each \(t \in [0, T]\) and \(w \in \Sigma\), the map \(\omega \mapsto \mathsf{K}^k_{n, \infty} (w,\omega, t)\) is continuous. Furthermore, whenever \((\omega^m)_{m = 1}^\infty \subset C ([0, T]; \H^N_n)\) is a sequence that converges uniformly, it also follows that 
\[
\big| \mathsf{K}^k_{n, m} (w, \omega^m, t) - \mathsf{K}^k_{n, \infty} (w, \omega^m, t)  \big| \to 0, \quad \text{as } m \to \infty,
\]
where we can apply the dominated convergence theorem thanks to the continuity of \(b^k_n\) and \(\sigma^k_n\).
Hence, by Lemma~\ref{appendix: remmert} in Appendix~\ref{sec: appendix}, for every compact set \(K \subset C ([0, T]; \H^N_n)\), 
\[
\sup_{\omega \in K} \big| \mathsf{K}^k_{n, m} (w, \omega, t) - \mathsf{K}^k_{n, \infty} (w, \omega, t)  \big| \to 0, \quad \text{as } m \to \infty. 
\]
By Prohorov's theorem, for every \(\varepsilon > 0\), there exists a compact set \(K = K (\varepsilon) \subset C ([0, T]; \H^N_n)\) such that 
\[
\sup_{m \in \mathbb{N}}Q^{n, N_m} (X \not \in K) \leq \varepsilon. 
\]
Hence, for every \(\varepsilon' > 0\), 
\begin{align*} 
Q^{n, N_m} \big( \big| \mathsf{K}^k_{n, m} (\, \cdot \,, t) - \mathsf{K}^k_{n, \infty} (\, \cdot \,, t) \big| \geq \varepsilon'\big) &\leq P \Big( \sup_{\omega \in K} \big| \mathsf{K}^k_{n, m} (\, \cdot \,, \omega, t) - \mathsf{K}^k_{n, \infty} (\, \cdot \,, \omega, t)  \big| \geq \varepsilon' \Big) + \varepsilon
\\&\to \varepsilon, \quad m \to \infty. 
\end{align*} 
Consequently, as \(\varepsilon > 0\) was arbitrary, 
\[
Q^{n, N_m} \big( \big| \mathsf{K}^k_{n, m} (\, \cdot \,, t) - \mathsf{K}^k_{n, \infty} (\, \cdot \,, t) \big| \geq \varepsilon'\big) \to 0, \quad m \to \infty, \ \forall \, \varepsilon' > 0.
\]
Using \eqref{eq: coercivity finite}, it follows from a standard Gronwall argument (see, e.g., \cite[pp. 1762 -- 1763]{GRZ}) that there exists a constant \(C > 0\), independent of \(n\) and \(m\), such that 
\begin{align}\label{eq: second moment}
E^{Q^{n, m}} \Big[ \sup_{s \in [0, T]} \|X_s\|^2_{\H_n^N} \Big] \leq C. 
\end{align}
Notice also that, by It\^o's formula, 
\begin{align} \label{eq: ito iden}
	\mathsf{K}^k_{n, m} = \int_0^\cdot g' ( \langle W^{k, n}_s, u \rangle_{\U_n} + \langle X^k_s, h \rangle_{\H_n} ) \langle u + (\sigma^k_{n, m})^* (s, X_s) h, d W^{k, n}_s \rangle_{\U_n}.
\end{align}
Hence, taking \eqref{eq: second moment}, \eqref{eq: ito iden} and the second line from \eqref{eq: coercivity finite} into consideration, we may conclude that (for each \(n \in \mathbb{N}\)) the family \(\{ Q^{n, m} \circ \mathsf{K}^k_{n, m} (\, \cdot\,,\, \cdot\, , s)^{-1} \colon s \in [0, T], m \in \mathbb{N}\}\) is uniformly integrable. 

In summary, we deduce from Lemma~\ref{appendix: convergence martingale} in Appendix~\ref{sec: appendix} that each \(\mathsf{K}^k_{n, \infty}\) is a \(Q^n\)-martingale, and Lemma~\ref{appendix: JS} in Appendix~\ref{sec: appendix} yields that, under \(Q^n\), 
\begin{align*}
	X^k = \Pi_n x + \int_0^\cdot b^k_{n} (s, X_s) \, ds + \int_0^\cdot \sigma^k_{n} (s, X_s) \, d W^{k, n}_s.
\end{align*}
Here, we remark that the structure of \(\mathsf{K}^k_{n, \infty}\) covered the bivariate process \((W^{k, n}, X^k)\), cf. \cite[Theorem~6.3]{jacod80} for a similar strategy.

\smallskip
\noindent
{\em Step 5:} Next, we show that \(Q^n \circ \Z_N (X, K)^{-1} \in \J (x)\). 
Fix some \(p \geq 1\). 
By It\^o's formula, \eqref{eq: coercivity finite} and Young's inequality for products, we obtain that
\begin{equation} \label{eq: big ineq ito} \begin{split} 
	\|X^k_t\|^{2p}_{\H} &= \|\Pi_n x \|^{2p}_{\H} + 2p \int_0^t \|X^k_s\|^{2 (p - 1)} \langle b^k_n (s, X_s), X^k_s \rangle_\H \, ds 
	\\&\qquad + p \int_0^t \|X_s^k\|^{2 (p - 1)}_\H \| \sigma^k_n (s, X_s) \|^2_{L_2 (\U_n; \H_n)} \, ds 
	\\&\qquad + 2 p ( p - 1 ) \int_0^t \|X^k_s\|^{2 ( p - 2 )} \| (\sigma^k_n)^* (s, X_s) X^k_s \|^2_{\U} \, ds 
	\\&\qquad + 2 p \int_0^t \| X^k_s \|^{2 (p - 1)}_\H \langle X^k_s, \sigma^k_n (s, X_s) \, d W^{k, n}_s \rangle_\H
	\\&\leq \|x \|^{2p}_{\H} +  \int_0^t \|X^k_s\|^{2 (p - 1)} \Big[3p \lambda\, \Big( 1 + \|X^k_s\|^2_\H + \frac{1}{N} \sum_{i = 1}^N \|X^i_s\|^2_{\H} \Big) - 2p \N (X^k_s) \Big] \, ds 
	\\&\qquad + 2 p ( p - 1 ) \int_0^t \|X^k_s\|^{2 ( p - 1 )} \, \lambda\, \Big( 1 + \|X^k_s\|^2_\H + \frac{1}{N} \sum_{i = 1}^N \|X^i_s\|^2_{\H} \Big) \, ds 
	\\&\qquad + 2 p \int_0^t \| X^k_s \|^{2 (p - 1)}_\H \langle X^k_s, \sigma^k_n (s, X_s) \, d W^{k, n}_s \rangle_\H
		\\&\leq \|x \|^{2p}_{\H} + \int_0^t \Big[\lambda p (2p + 1) \Big( 1 + \Big(3 - \frac{1}{p}\Big) \|X^k_s\|^{2p}_\H + \frac{1}{Np} \sum_{i = 1}^N \|X^i_s\|^{2p}_{\H} \Big) - 2p \N_p (X^k_s) \Big] \, ds 
	\\&\qquad + 2 p \int_0^t \| X^k_s \|^{2 (p - 1)}_\H \langle X^k_s, \sigma^k_n (s, X_s) \, d W^{k, n}_s \rangle_\H.
\end{split}
\end{equation} 
Consequently, using Burkhoder's inequality (\cite[Theorem~2.5]{krylov_rozovskii}), we obtain that 
\begin{equation} \label{eq: first main to kappa} \begin{split}
	\frac{1}{N} \sum_{k = 1}^N E^{Q^n} \Big[ \sup_{s \in [0, t]} \|X^k_s\|^{2p}_\H \Big] 
	&\leq \|x\|^{2p}_\H + E^{Q^n} \Big[ \int_0^t \lambda p (2p + 1) \Big( 1 + \frac{3}{N} \sum_{i = 1}^N \|X^i_s\|^{2p}_{\H} \Big)  \, ds \Big]
	\\&\qquad + \frac{6 p}{N} \sum_{k = 1}^N  E^{Q^n} \Big[ \Big(\int_0^t \|X^k_s\|^{4 p - 2}_\H \| \sigma^k_n (s, X_s) \|^2_{L_2 (\U_n; \H_n)} \, ds \Big)^{1/2} \Big].
\end{split}
\end{equation}
Using again \eqref{eq: coercivity finite} and Young's inequality for products, we obtain that 
\begin{align*}
	E^{Q^n} & \Big[ \Big(\int_0^t \|X^k_s\|^{4 p - 2}_\H \| \sigma^k_n (s, X_s) \|^2_{L_2 (\U_n; \H_n)} \, ds \Big)^{1/2} \Big] 
	\\&\leq E^{Q^n} \Big[ \sup_{r \in [0, t]} \|X^k_r\|^p_\H \Big(\int_0^t \|X^k_s\|^{2 (p - 1)}_\H \| \sigma^k_n (s, X_s) \|^2_{L_2 (\U_n; \H_n)} \, ds \Big)^{1/2} \Big]
	\\&\leq \frac{1}{12 p} E^{Q^n} \Big[ \sup_{r \in [0, t]} \|X^k_r\|^{2p}_\H \Big] + 3 p E^{Q^n} \Big[ \int_0^t \|X^k_s\|^{2 (p - 1)}_\H \| \sigma^k_n (s, X_s) \|^2_{L_2 (\U_n; \H_n)} \, ds \Big]
		\\&\leq \frac{1}{12 p} E^{Q^n} \Big[ \sup_{r \in [0, t]} \|X^k_r\|^{2p}_\H \Big] + 3 p E^{Q^n} \Big[ \int_0^t \lambda \Big(  1 + \Big(3 - \frac{1}{p}\Big) \|X^k_s\|^{2p}_\H + \frac{1}{Np} \sum_{i = 1}^N \|X^i_s\|^{2p}_{\H} \Big) \, ds \Big].
\end{align*}
Together with \eqref{eq: first main to kappa}, we conclude that 
\begin{align*}
		\frac{1}{N} \sum_{k = 1}^N E^{Q^n} \Big[ \sup_{s \in [0, t]} \|X^k_s\|^{2p}_\H \Big] &\leq 2 \|x\|^{2p}_\H + \int_0^t 2 \lambda (p (2p + 1) + 18p^2)\Big( 1 + \frac{3}{N} \sum_{i = 1}^N E^{Q^n} \big[ \|X^i_s\|^{2p}_{\H} \big] \Big) \, ds 
		\\
		&\leq 2 \|x\|^{2p}_\H + \int_0^t6  \lambda p (1 + 20 p)\Big( 1 + \frac{1}{N} \sum_{i = 1}^N E^{Q^n} \Big[ \sup_{r \in [0, s]} \|X^i_r\|^{2p}_{\H} \Big] \Big) \, ds.
\end{align*}
Now, applying Gronwall's lemma\footnote{Of course, Gronwall's lemma requires that the involved function is integrable. In our case, it is well-known that the expectations are finite (see, e.g., \cite{GRZ} or \cite[Lemma~5.1.5]{liu_rockner}). Alternatively, it would also be possible to work with stopping times. At this point, our main interest is an explicit upper bound to relate \(Q^n\) to the set~\(\J(x)\).} to
\[
t \mapsto 1 + \frac{1}{N} \sum_{k = 1}^N E^{Q^n} \Big[ \sup_{s \in [0, t]} \|X^k_s\|^{2p}_\H \Big]
\]
yields that
\begin{align*}
	\frac{1}{N} \sum_{k = 1}^N E^{Q^n} \Big[ \sup_{s \in [0, T]} \|X^k_s\|^{2p}_\H \Big] &\leq (1 + 2 \|x\|^{2p}_\H  ) e^{6  \lambda p T (1 + 20 p)} - 1.
\end{align*}
Finally, taking expectation in \eqref{eq: big ineq ito}, and using the martingale property of the It\^o integral, we observe that 
\begin{align*}
0 \leq 	\frac{1}{N} \sum_{k = 1}^N E^{Q^n} \big[ \|X^k_T\|^{2p}_\H \big] \leq \|x\|^{2p}_\H  &+ \int_0^T \lambda p (2p + 1) \Big(1 + \frac{3}{N} \sum_{i = 1}^N  E^{Q^n} \big[ \|X^i_s\|^{2p}_\H \big] \Big) \, ds \\&- \frac{2p}{N} \sum_{k = 1}^N  E^{Q^n} \Big[ \int_0^T \N_p (X^k_s) \, ds \Big], 
\end{align*}
and consequently,
\begin{align*}
	\frac{1}{N} \sum_{k = 1}^N  &E^{Q^n} \Big[ \int_0^T \N_p (X^k_s) \, ds \Big] 
	\leq \frac{1}{2p} \Big[ \|x\|^{2p}_\H + 3 \lambda p T (2p + 1) (1 + 2 \|x\|^{2p}_\H  ) e^{6  \lambda p T (1 + 20 p)}  \Big].
\end{align*}
In summary, putting these estimates together, it follows that \(Q^n \circ \Z_N(X, K)^{-1} \in \J(x)\).

\smallskip
\noindent
{\em Step 6:} We are in the position to sketch the final step of the proof. First of all, it follows similar to Step~1 of the proof for Lemma~\ref{lem: Rk rel comp} that the family \(\{ Q^n \circ X^{-1} \colon n \in \mathbb{N}\}\) is relatively compact in \(\fP (\Omega^N)\). Thus, as in Step~3 of this proof, there exists a subsequence \((Q^{M_n})_{n = 1}^\infty\) that converges to a probability measure \(Q\) in the weak-strong sense. As before, the measure \(Q\) admits the decomposition \eqref{eq: kernel dec} and  \(W^1, \dots, W^N\) remain independent cylindrical Brownian motions under~\(Q\) (see Remark~\ref{appendix: rem factorization} in Appendix~\ref{sec: appendix}).  
Now, by virtue of Lemma~\ref{lem: mg chara N}, we can use a martingale problem argument as in Step~4 above (see also Step~3 from the proof of \cite[Theorem~4.5]{GRZ}) to conclude that \eqref{eq: existence identity} holds under \(Q\). We omit the details for brevity. 
Finally, \(Q \circ \Z_N (X, K)^{-1} \in \J (x)\) follows from Lemma~\ref{lem: M closed} and Step~5. 
\end{proof}

The following is an immediate consequence of Proposition~\ref{prop: existence}.
\begin{corollary} \label{coro: non empty}
If Condition~\ref{cond: main1} holds, then \(\cC^n (x) \not = \emptyset\) for all \(n \in \mathbb{N}\) and \(x \in \H\). 
\end{corollary} 

\subsection{Proof of Theorem \ref{theo: main1} (i)}
The set \(\cR^n (x^n)\) is nonempty by Corollary~\ref{coro: non empty} and compact by Lemma~\ref{lem: cCN compact}. The compactness of \(\cR^0 (x^0)\) follows from Lemma~\ref{lem: C comp}. Anticipating the following section, the claim \(\cR^0 (x) \not = \emptyset\) follows from Theorem~\ref{theo: main1}~(ii).
\qed 

\subsection{Proof of Theorem \ref{theo: main1} (ii)}
	Take a sequence \((x^n)_{n = 0}^\infty \subset \H\) such that \(x^n \to x^0\) in \(\|\cdot\|_\H\) and let \((Q^n)_{n = 1}^\infty\) be such that \(Q^n \in \cR^n (x^n)\). By Lemma \ref{lem: Rk rel comp}, the set \(\bigcup_{n = 1}^\infty \cR^n (x^n) \subset \bigcup_{n, m = 1}^\infty \cR^n (x^m) = \cR ( \{x^n \colon n \in \mathbb{N}\} )\) is relatively compact in \(\fP^\al (\fP^\al(\Theta))\). Consequently, \((Q^n)_{n = 1}^\infty\) is relatively compact in \(\fP^\al (\fP^\al (\Theta))\). It remains to prove that each of its accumulation points is in \(\cR^0 (x^0)\).
	To lighten our presentation, we prove that \(Q^0\in \cR^0 (x^0)\) whenever \(Q^n \to Q^0\) in \(\fP^\al (\fP^\al (\Theta))\). The argument is based on Lemma~\ref{lem: mg chara}.
	
	First of all, 
	\begin{align*}
		Q^0 ( \{ \mu \in \W(\Theta) \colon \mu \circ X^{-1}_0 = \delta_{x^0} \} ) = 1
	\end{align*}
follows as in the proof for Lemma~\ref{lem: C comp}, i.e., (iii.a) from Lemma~\ref{lem: mg chara} holds for \(Q^0\)-a.a. \(\mu \in \fP (\Theta)\).

Lemma~\ref{lem: M closed} yields that \(Q^0 \in \J (x^0)\). In particular, (iii.b) from Lemma~\ref{lem: mg chara} holds for \(Q^0\)-a.a. \(\mu \in \fP (\Theta)\). 
	
	It remains to investigate part (iii.c) from Lemma \ref{lem: mg chara}. Take \(s, t \in \mathbb{Q}_+ \cap [0, T], s < t, g \in \tC\) and \(\test \in \tT_s\). 
	We define \(\oZ^k\) and \(\oZ\) as in the proof of Lemma~\ref{lem: C comp}. In particular, recall \eqref{eq: oZ after bound}.
We now prove that \(Q^0\)-a.s. \(\oZ = 0\). As \(\B, \tC\) and \(\tT_s\) are countable, this 
implies that almost all realizations of \(Q^0\) satisfy (iii.c) from Lemma \ref{lem: mg chara}. In summary, we then can conclude that \(Q^0 \in \cR^0 (x^0)\). 

The proof of \(Q^0\)-a.s. \(\oZ= 0\) uses a strategy we learned from \cite{bhatt1998interacting}. It is divided into two steps. First, we prove that 
\begin{align} \label{eq: first main oZ = 0}
	\lim_{n \to \infty} E^{Q^n} \big[ | \oZ | \big] = E^{Q^0} \big[ | \oZ | \big], 
\end{align} 
and afterwards, we show that 
\begin{align} \label{eq: second main oZ = 0}
	\lim_{n \to \infty} E^{Q^n} \big[ | \oZ |^2 \big] = 0.
\end{align} 
Putting \eqref{eq: first main oZ = 0} and \eqref{eq: second main oZ = 0} together yields that \(E^{Q^0} [ | \oZ | ] = 0\), proving that \(Q^0\)-a.s. \(\oZ = 0\). 

\smallskip
\noindent
{\em Step 1: Proof of \eqref{eq: first main oZ = 0}.}
By the triangle inequality, we observe that 
\begin{equation} \label{eq: main triangle}
	\begin{split}
	| E^{Q^n} [ | \oZ | ] - E^{Q^0} [ | \oZ | ] | &\leq | E^{Q^n} [ | \oZ | ] - E^{Q^n} [ | \oZ^k | ] |  \\&\hspace{1cm}+ 	| E^{Q^n} [ | \oZ^k | ] - E^{Q^0} [ | \oZ^k | ] | 
	\\&\hspace{1cm}+ 	| E^{Q^0} [ | \oZ^k | ] - E^{Q^0} [ | \oZ | ] |
	\\&=: I_{n, k} + II_{n, k} + III_{k}.
	\end{split}
\end{equation}
By Lemma~\ref{lem: oM mit k cont} and Lemma~\ref{appendix: bogachev} in Appendix~\ref{sec: appendix}, the map \(\mu \mapsto \oZ^k (\mu)\) is bounded and continuous and we get that \(II_{n, k} \to 0\) as \(n \to \infty\) for every \(k > 0\). 
Next, we discuss the term \(I_{n, k}\). 
With \(Q^n = P^n \circ \Z_n^{-1}\), \(Q^n \in \J (x^n)\) and \eqref{eq: bound}, we obtain that 
\begin{align*}
	I_{n, k} &\leq \frac{C}{n}\sum_{i = 1}^n E^{P^n} \Big[ \int_0^T \int \big| \mathcal{L}_{g, v} (f, r, X^i_r, \x_n (X_r)) \big| \1_{\{| \mathcal{L}_{g, v} (f, r, X^i_r, \x_n (X_r))| \, >\, k\}} \, M^i (dr, df) \Big]
	\\&\leq \frac{1}{k^{\gamma - 1}} \frac{C}{n}\sum_{i = 1}^n E^{P^n} \Big[ \int_0^T \int \big| \mathcal{L}_{g, v} (f, r, X^i_r, \x_n (X_r)) \big|^\gamma \, M^i (dr, df) \Big] 
	\\&\leq \frac{C}{k^{\gamma- 1}}, 
\end{align*}
where we emphasize that the constant \(C > 0\) is independent of \(n\). 
Similarly, using that \(Q^0 \in \J (x^0)\) and \eqref{eq: bound}, we obtain that 
\begin{align*}
	III_{k} \leq \frac{C}{k^{\gamma - 1}}.
\end{align*}
Together, we obtain that  \(I_{n, k} + III_{k} \to 0\) as \(k \to \infty\) uniformly in \(n\). Thus, choosing first \(k\) large and taking then \(n \to \infty\) shows that \(I_{n, k} + II_{n, k} + III_{k}\) can be made arbitrarily small, which entails that \eqref{eq: first main oZ = 0} holds. 

\smallskip
\noindent
{\em Step 2: Proof of \eqref{eq: second main oZ = 0}.}
For \(r \in [0, T]\) and \(i = 1, \dots, n\), we set
\[
\oK^{i, n}_r := g ( {_\V}\langle X^i_r, v \rangle_{\V^*}) - g ({_\V}\langle X^i_0, v\rangle_{\V^*}) - \int_0^r\int \mathcal{L}_{g, v} (f, u, X^i_u, \x_n (X_u)) \, M^i (du, df).
\]
Notice that 
\begin{align*}
	E^{Q^n} \big[ \oZ^2 \big] = \frac{1}{n^2} \sum_{i, j = 1}^n E^{P^n} \Big[ (\oK^{i, n}_t - \oK^{i, n}_s) (\oK^{j, n}_t - \oK^{j, n}_s) \, \test (X^i, M^i) \test (X^j, M^j) \Big].
\end{align*}
Recall from Definition~\ref{def: C^n} that there are independent cylindrical Brownian motions \(W^1, \dots, W^n\) such that \(P^n\)-a.s., for \(k = 1, \dots, n\), 
	\begin{align*}
		{_\V}\langle X^k_t, v \rangle_{\V^*} = {_\V} \langle x^n, v\rangle_{\V^*} &+ \int_0^t \blv \Big\langle \int b (f, s, X^k_s, \x_n (X_s)) \, \v (s, M^k, df) , v \Big \rangle_{\V^*} \, ds
		\\&+ \Big\langle \int_0^t \us (\v (s, M^k), s, X^k_s, \x_n (X_s)) \, d W^k_s, v \Big \rangle_\H.
	\end{align*}
	Thus, by It\^o's formula, we obtain that \(P^n\)-a.s.
\[
\oK^{i, n} = \int_0^\cdot g' ({_\V}\langle X^i_u, v \rangle_{\V^*}) \langle \us^* (\v(u, M^i), u, X^i_u, \x_n (X_u)) v, d W^i_u \rangle_\U.
\]
Take \(i \not = j\).
By the independence of \(W^i\) and \(W^j\), the quadratic variation of \(\oK^{i, n}\) and \(\oK^{j, n}\) vanishes. As \(\oK^{i, n}\) and \(\oK^{j, n}\) are square integrable \(P^n\)-martingales (see Lemma \ref{lem: mg chara N}), this means that the product \(\oK^{i, n} \oK^{j, n}\) is a \(P^n\)-martingale. Consequently, using that \(\oK^{i, n}, \oK^{j, n}\) and \(\oK^{i, n}\oK^{j, n}\) are \(P^n\)-martingales (recall that \(\test\) is \(\mathcal{O}_s\)-measurable), we obtain that
\begin{align*}
E^{P^n} \Big[ (\oK^{i, n}_t &- \oK^{i, n}_s) (\oK^{j, n}_t - \oK^{j, n}_s) \, \test (X^i, M^i) \test (X^j, M^j) \Big]
\\&= E^{P^n} \Big[ (\oK^{i, n}_t \oK^{j, n}_t  - \oK^{i, n}_t \oK^{j, n}_s - \oK^{i, n}_s \oK^{j, n}_t + \oK^{i, n}_s \oK^{j, n}_s) \, \test (X^i, M^i) \test (X^j, M^j) \Big]
\\&= E^{P^n} \Big[ (\oK^{i, n}_s \oK^{j, n}_s  - \oK^{i, n}_s \oK^{j, n}_s - \oK^{i, n}_s \oK^{j, n}_s + \oK^{i, n}_s \oK^{j, n}_s) \, \test (X^i, M^i) \test (X^j, M^j) \Big] = 0.
\end{align*} 
Finally, using Burkholder's inequality, \(Q^n \in \J (x^n)\) and \eqref{eq: cond growth diffusion}, we get that 
\begin{align*}
	E^{Q^n} \big[ \oZ^2 \big] &= \frac{1}{n^2} \sum_{i = 1}^n E^{P^n} \Big[ (\oK^{i, n}_t - \oK^{i, n}_s)^2 \, \test (X^i, M^i)^2 \Big]
	\\&\leq \frac{C}{n^2} \sum_{i = 1}^n E^{P^n} \Big[ (\oK^{i, n}_t - \oK^{i, n}_s)^2 \Big]
	\\&\leq \frac{C}{n^2} \sum_{i = 1}^n E^{P^n} \Big[ \int_s^t \| \us^* (\v (u, M^i), u, X^i_u, \x_n (X_u)) \|_{L_2 (\U; \H)}^2 \, du \Big]
	\\&\leq \frac{C}{n^2} \sum_{i = 1}^n E^{P^n} \Big[ \int_0^T \lambda \Big( 1+ \|X^i_u\|^2_\H + \frac{1}{n} \sum_{k = 1}^n \|X^k_u\|^2_\H \Big) \, du \Big]
	\\&\leq \frac{C}{n} \Big( 1 + \frac{1}{n} \sum_{k = 1}^n E^{P^n} \Big[ \sup_{u \in [0, T]} \|X^k_u\|^2_\H \Big] \Big)
	\\&\leq \frac{C}{n}.
\end{align*}
This bound proves \eqref{eq: second main oZ = 0}. The proof of Theorem~\ref{theo: main1}~(ii) is complete.
 \qed

	\subsection{Proof of Theorem \ref{theo: main1} (iii)}
	We use the notation from Theorem \ref{theo: main1} (iii). 
	By the compactness of \(\cR^n (x^n)\), which is due to Theorem \ref{theo: main1} (i), and standard properties of the limes superior, there exists a subsequence \((N_n)_{n = 1}^\infty\) of \(1, 2, \dots\) and measures \(Q^{N_n} \in \cR^{N_n} (x^{N^n})\) such that
	\[
	\limsup_{n \to \infty} \sup_{Q \in \cR^n (x^n)} E^Q \big[ \psi \big] = \lim_{n \to \infty} E^{Q^{N_n}} \big[ \psi \big].
	\]
	By Theorem \ref{theo: main1} (ii), there is a subsequence of \((Q^{N_n})_{n = 1}^\infty\) that converges in \(\fP^\al (\fP^\al (\Theta))\) to a measure \(Q^0 \in \cR^0 (x^0)\). Recalling that \(\psi\) is upper semicontinuous from \(\fP^\al(\Theta)\) into~\(\bR\) and that it satisfies the growth condition \eqref{eq: bound property}, we get from Lemma~\ref{lem: gen Fatou W} in Appendix~\ref{sec: appendix} that
	\[
	\lim_{n \to \infty} E^{Q^{N_n}} \big[ \psi \big] \leq E^{Q^0} \big[ \psi \big] \leq \sup_{Q \in \cR^0 (x^0)} E^Q \big[ \psi \big].
	\]
	This completes the proof.
	\qed

	\subsection{Proof of Theorem \ref{theo: main1} (iv)}
	The following strategy is inspired by the proof for \cite[Theorem~2.12]{LakSIAM17}. In particular, we learned the idea to use the Krein--Milman theorem from this reference.
	Let us start with an auxiliary result whose proof is postponed to the end of this section. 
	\begin{lemma} \label{lem: approx point mass}
		Assume that the Conditions~\ref{cond: main1} and \ref{cond: main2} hold. Let \((x^n)_{n = 0}^\infty \subset \H\) be a sequence such that \(x^n \to x^0\) in \(\|\cdot\|_\H\) and take \(P \in \cC^0 (x^0)\). Then, there exists a sequence \((Q^{n})_{n = 1}^\infty\) with \(Q^{n} \in \cR^{n} (x^{n})\) such that \(Q^{n} \to \delta_{P}\) in \(\fP^\al (\fP^\al (\Theta))\).
	\end{lemma}
With this lemma at hand, we can prove Theorem \ref{theo: main1} (iv). 
Define 
\[
L := \Big\{ Q \in \fP^\al (\fP^\al (\Theta)) \colon \exists \, Q^n \in \cR^n (x^n) \text{ such that } Q = \lim_{n \to \infty} Q^n \text{ in } \fP^\al (\fP^\al (\Theta)) \Big\}. 
\] 
We show that \(\cR^0 (x^0) \subset L\), which implies the claim of Theorem \ref{theo: main1} (iv). By Corollary \ref{coro: conv}, each \(\cR^n (x^n)\) is convex. Thus, also \(L\) is convex. Furthermore, by \cite[Proposition~7.1.20]{hu_papa}, \(L\) is closed in \(\fP^\al (\fP^\al (\Theta))\). 
Recall from Theorem~\ref{theo: main1} (i) that the set \(\cR^0 (x^0)\) is nonempty and compact in \(\fP^\al (\fP^\al (\Theta))\). 
Furthermore, \(\cR^0 (x^0)\) is convex and we denote its extreme points by \(\on{ex}\, (\cR^0 (x^0))\).
We notice that the topologies induced by either \(\fP^\al (\fP^\al (\Theta))\) and \(\fP(\fP^\al (\Theta))\) coincide on \(\cR^0 (x^0)\) by the definition of \(\J (x^0)\), and furthermore that \(\cR^0 (x^0)\) is compact in both of these topologies. 
Thanks to the Krein--Milman theorem (\cite[Theorem~7.68]{charalambos2013infinite}), we have
\[
\cR^0 (x^0) = \overline{\on{co}}\, \big[\on{ex}\, (\cR^0 (x^0))\big],
\]
where \(\overline{\on{co}} \,[\, \cdot\,]\) denotes the closure (in \(\fP^\al (\fP^\al (\Theta))\)) of the convex hull. 
Hence, as \(L\) is convex and closed, to conclude \(\cR^0 (x^0) \subset L\), it suffices to show that \(\on{ex}\, (\cR^0 (x^0)) \subset L\).
The set \(\mathcal{C}^0 (x^0)\) with the weak topology is a Lusin space and consequently, so is \(\mathcal{P} (\mathcal{C}^0 (x^0))\), see Lemma~\ref{appendix: Suslin prob meas} in Appendix~\ref{sec: appendix}. 
Thus, by Lemma~\ref{appendix: extreme points} in Appendix~\ref{sec: appendix}, it follows that
\[
\on{ex}\, ( \cR^0 (x^0)) \subset \Big\{  \mu \in \fP (\cC^0 (x^0)) \colon \mu = \lambda \, \delta_{\mu^1} + (1 - \lambda) \, \delta_{\mu^2}, \ \lambda \in [0, 1], \ \mu^1, \mu^2 \in \cC^0 (x^0) \Big\}.
\]
Using this observation, as each \(\cR^n (x^n)\) is convex, Lemma~\ref{lem: approx point mass} implies that \(\on{ex}\, ( \cR^0 (x^0)) \subset L\). This completes the proof of Theorem \ref{theo: main1} (iv).
 \qed
\vspace{0.25cm}

It is left to prove Lemma \ref{lem: approx point mass}. The main strategy is to adapt the classical coupling argument for propagation of chaos in finite dimensions (\cite{LakLN,SM}) to our infinite-dimensional controlled setting. Before we present the technical proof, let us sketch the main idea for a class of finite dimensional equations.
\begin{SoI}
Suppose that \(X\) is a \(d\)-dimensional solution process to the stochastic equation 
\[
d X_t = \ell (X_t, \alpha_t, P^X_t) \, d W_t, \quad X_0 = x_0,
\] 
where \(\alpha\) is a control process and \(W\) is a Brownian motion, and \(\ell\) is a bounded Lipschitz continuous function such that 
\[
\| \ell (x, \alpha, \mu) - \ell (y, \alpha, \nu) \|^2 \leq C \Big( \|x - y\|^2 + \mathsf{w}_2 (\mu, \nu)^2 \Big), 
\] 
where \(\mathsf{w}_2\) is the 2-Wasserstein distance on \(\fP^2 (\bR^d)\).
To approximate \(X\) by a particle system, let \((X^k, \alpha^k, W^k)_{k = 1}^N\) be independent, identically distributed copies of \((X, \alpha, W)\) such that  
\[
d X^k_t = \ell (X^k_t, \alpha^k_t, P^X_t) \, d W^k_t, \quad X^k_0 = x_0.
\] 
Now, let \(Y^1, \dots, Y^N\) be solutions to
\begin{align} \label{eq: approx sketch} 
d Y^k_t = \ell (Y^k_t, \alpha^k_t, \x_N (Y^1_t, \dots, Y^N_t)) \, d W^k_t, \quad Y^k_0 = x_0, 
\end{align} 
whose existence is guaranteed by the Lipschitz hypothesis on \(\ell\). 
Using that \(X^k\) and \(Y^k\) are driven by the same Brownian motion and associated to the same control process, we obtain from It\^o's formula that 
\begin{align} \label{eq: estimate from ito} 
E \Big[ \sup_{s \in [0, t]} \| X^k_s - Y^k_s \|^2 \Big] \leq C \int^t_0 E \Big[ \|X^k_s - Y^k_s\|^2 + \mathsf{w}_2 (P^X_s, \x_N (Y^1_s, \dots, Y^N_s))^2 \Big] \, ds. 
\end{align} 
Thus, by Gronwall's lemma, 
\[
E \Big[ \sup_{s \in [0, t]} \| X^k_s - Y^k_s \|^2 \Big] \leq C \int^t_0 E \Big[ \mathsf{w}_2 (P^X_s, \x_N (Y^1_s, \dots, Y^N_s))^2 \Big] \, ds.
\] 
Finally, 
\begin{align*}
E \Big[ \sup_{s \in [0,t]} \mathsf{w}_2 (P^X_s, \x_N (Y^1_s, \dots, Y^N_s))^2\Big] &\leq C \, \Big( E \Big[ \sup_{s \in [0,t]} \mathsf{w}_2 (P^X_s, \x_N (X^1_s, \dots, X^N_s))^2\Big] \\&\quad + E \Big[ \sup_{s \in [0,t]} \mathsf{w}_2 (\x_N (X^1_s, \dots, X^N_s), \x_N (Y^1_s, \dots, Y^N_s))^2\Big] \Big) 
\\&\leq C \, \Big( E \Big[ \sup_{s \in [0,t]} \mathsf{w}_2 (P^X_s, \x_N (X^1_s, \dots, X^N_s))^2\Big] \\&\quad +  \int^t_0 E \Big[ \mathsf{w}_2 (P^X_s, \x_N (Y^1_s, \dots, Y^N_s))^2 \Big] \, ds \Big), 
\end{align*} 
which, again by Gronwall's lemma, yields 
\[
E \Big[ \sup_{s \in [0,t]} \mathsf{w}_2 (P^X_s, \x_N (Y^1_s, \dots, Y^N_s))^2\Big]  \leq C E \Big[ \sup_{s \in [0,t]} \mathsf{w}_2 (P^X_s, \x_N (X^1_s, \dots, X^N_s))^2\Big].
\] 
As the right hand side converges to zero by a version of the law of large numbers, it follows that \(\x_N (Y^1_s, \dots, Y^N_s)\) approximates the law of \(X_s\). 

Our proof of Lemma~\ref{lem: approx point mass} adapts the same general strategy, but several steps are technically more delicate in the infinite dimensional setting. First, unlike the finite dimensional case, we cannot rely on a strong existence theorem for the interacting particle system \eqref{eq: approx sketch}; instead, we use the weak existence result of Proposition~\ref{prop: existence}. Second, since \(\Omega\) carries several relevant norms, we must carefully track the convergence mode to ensure that the coupling argument yields the desired limit. To do so, we combine the coupling method with a relative compactness argument and use it only to identify the limit. Finally, to obtain an estimate analogous to \eqref{eq: estimate from ito} in the appropriate norm, we cannot use the standard Itô formula; instead, we rely on the Krylov–Rozovskii It\^o formula (Lemma~\ref{appendix: Krylov}), which is adapted to the variational framework.
\end{SoI} 
	
\begin{proof}[Proof of Lemma \ref{lem: approx point mass}]
	Let \((x^n)_{n = 0}^\infty \subset \H\) be a sequence such that \(x^n \to x^0\) in \(\|\cdot\|_\H\) and take \(R \in \cC^0 (x^0)\). 
	By definition, there exists a standard extension \((\Sigma, \mathcal{A}, (\mathcal{A}_t)_{t \in [0, T]}, P)\) of \((\Theta, \mathcal{O}, \mathbf{O}, R)\), supporting a standard cylindrical Brownian motion \(W\), such that 
	\(P\)-a.s., for all \(v \in \V^*\),
	\begin{align*}
			{_\V}\langle X, v \rangle_{\V^*} = {_\V}\langle x, v\rangle_{\V^*}&+ \int_0^\cdot \blv \Big\langle  \int b (f, s, \oX_s, R_s) \, \v (s, M, df), v \Big \rangle_{\V^*} \, ds
			\\&+ \Big\langle \int_0^\cdot \us (\v (s, M), s, \oX_s, R_s) \, d W_s, v \Big \rangle_\H, 
		\end{align*}
	where \(R_s = R \circ X^{-1}_s = P \circ X^{-1}_s\).
	For \(n \in \mathbb{N}\), consider the \(n\)-fold product \((\Sigma^n, \mathcal{A}^n, (\mathcal{A}^n_t)_{t \in [0, T]}, P^n)\), where
	\[
	\mathcal{A}^n_t := \bigcap_{s > t} \bigotimes_{k = 1}^n \mathcal{A}_s, \quad P^n := \bigotimes_{k = 1}^n P.
	\]
	With obvious notation, we have \(P^n\)-a.s., for \(k = 1, \dots, n\) and \(v \in \V^*\),
	\begin{align*}
	{_\V}\langle X^k, v \rangle_{\V^*} = {_\V}\langle x, v\rangle_{\V^*}&+ \int_0^\cdot \blv \Big\langle  \int b (f, s, \oX_s^k, R_s) \, \v (s, M^k, df), v \Big \rangle_{\V^*} \, ds
	\\&+ \Big\langle \int_0^\cdot \us (\v (s, M^k), s, \oX^k_s, R_s) \, d W^k_s, v \Big \rangle_\H, 
\end{align*}
where, by construction, \(W^1, \dots, W^n\) are independent cylindrical Brownian motions. 
	Further, we set 
	\[
	\Psi := \Sigma^n \times \Omega^n, \quad \mathcal{G}_t := \mathcal{A}^n \otimes \cF^n, \quad \mathcal{G}_t := \bigcap_{s > t} \, (\mathcal{A}^n_s \otimes \cF^n_s),
	\]
	and we denote the coordinate process on the second coordinate by \((Y^1, \dots, Y^n)\). By Proposition~\ref{prop: existence}, there exists a probability measure \(Q^n\) on \((\Psi, \mathcal{A})\) such that \(Q^n \circ \Z_n (Y, M)^{-1} \in \J (x^n)\), where \(M := (M^1, \dots, M^n)\), and \(Q^n\)-a.s., for all \(k = 1, \dots, n\) and \(v \in \V^*\),
	\begin{align*}
		    {_\V}\langle Y^k, v \rangle_{\V^*} =  {_\V}\langle x^n, v \rangle_{\V^*} &+ \int_0^\cdot \blv \Big \langle \int b (f, s, Y^k_s, \x_n (Y_s)) \, \v (s, M^k, df), v \Big \rangle_{\V^*} \, ds
		\\&+ \Big \langle \int_0^\cdot \us (\v (s, M^k), s, Y^k_s, \x_n (Y_s)) \, d W^k_s, v \Big\rangle_\H.
	\end{align*}
Furthermore, \(Q^n\) has a decomposition of the type \eqref{eq: kernel dec}, which entails that the dynamics of \(X^1, \dots, X^n\) are still valid under \(Q^n\), see Remark~\ref{appendix: rem factorization} in Appendix~\ref{sec: appendix}. 

We notice that \(R^n := Q^n \circ \Z_n (Y, M)^{-1} \in \cR^n (x^n)\). In the following, we prove that \(R^n \to \delta_R\) in \(\fP^\al (\fP^\al (\Theta))\), which finishes the proof. We proceed with several steps. 

\smallskip
\noindent
{\em Step 0.} Let us start with notations and preparations.
We introduce the set
\[
\Upsilon := \Big\{ \omega \in C ([0, T]; \X) \colon \int_0^T \| \omega (s) \|^\alpha_\Y \, ds < \infty \Big\} \times \m \subset \Theta,
\]
and recall
\[
\d (\omega, \alpha) = \sup_{s \in [0, T]} \|\omega(s) - \alpha (s)\|_{\V} + \Big( \int_0^T \| \omega (s) - \alpha (s)\|^\alpha_\Y \, ds \Big)^{1/ \alpha}.
\]
Let us also consider the metrics
\begin{align*}
		\d^\circ (\omega, \alpha)&:= \sup_{s \in [0, T]} \|\omega(s) - \alpha (s)\|_{\X} + \Big( \int_0^T \| \omega (s) - \alpha (s)\|^\alpha_\Y \, ds \Big)^{1/ \alpha}, \\
	\d' (\omega, \alpha) &:= \sup_{s \in [0, T]} \| \omega (s) - \alpha (s)\|_\V,   \\
	\d^* (\omega, \alpha) &:= \sup_{s \in [0, T]} \| \omega (s) - \alpha (s)\|_{\X}.
\end{align*}
Using that \(\X \subset \V\), it follows that \(\d, \d', \d^* \leq c\, \d^\circ, \d' \leq \d\) and \(\d' \leq c\,\d^*\) for some constant \(c > 0\). Thus, on the set \(\Upsilon\), the topologies induced by \(\d + \r\) or \(\d^* + \r\) are both stronger than those induced by \(\d' + \r\), and the topology induced by \(\d^\circ + \r\) is stronger than those induced by either \(\d + \r, \d' + \r\) or \(\d^* + \r\). 
As \(\Upsilon\) is Polish with the topology from \(\d^\circ  + \r\), it is Lusin with each of the topology induced by \(\d + \r, \d' + \r\) or \(\d^* + \r\). The Borel \(\sigma\)-fields coincide for all these topologies by Lemma~\ref{appendix: Borel sigma fields} in Appendix~\ref{sec: appendix}, the set \(\mathcal{P} (\Upsilon)\) is the same for all of these topologies and \(\mathcal{P} (\Upsilon) \in \mathcal{B} (\mathcal{P} (\Theta))\) by Lemmata~\ref{appendix: lusin Borel} and~\ref{appendix: Suslin prob meas} in Appendix~\ref{sec: appendix}.

\smallskip 
\noindent
{\em Step 1.} Notice that \((R^n)_{n = 1}^\infty\) is relatively compact in \(\fP^\al (\fP^\al (\Theta))\) by Lemma~\ref{lem: Rk rel comp}. Therefore, \(R^n \to \delta_R\) in \(\fP^\al (\fP^\al (\Theta))\) follows once we prove that all accumulation points of \((R^n)_{n = 1}^\infty\) coincide with \(\delta_R\).
Let us explain our argument for this in detail.
First, we show that \(\delta_R, R^n\in \fP (\fP (\Upsilon))\).
Consequently, by virtue of Lemma~\ref{appendix: subspace weak conv} in Appendix~\ref{sec: appendix}, it suffices to prove that all accumulation points of \((R^n)_{n = 1}^\infty\) in \(\fP (\fP (\Upsilon))\), where \(\Upsilon\) is endowed with the topology induced by \(\d + \r\), are given by~\(\delta_R\).
Take a sequence \((\mu^n)_{n = 1}^\infty \subset \fP (\fP (\Upsilon))\) such that \(\mu^n \to \mu\) in \(\fP (\fP (\Upsilon))\) when \(\Upsilon\) is endowed with the topology induced by \(\d + \r\) and \(\mu^n \to \mu^*\) when \(\Upsilon\) is endowed with the topology induced by \(\d^* + \r\). Then, we also have \(\mu^n \to \mu\) and \(\mu^n \to \mu^*\) in \(\fP (\fP (\Upsilon))\) when \(\Upsilon\) is endowed with the topology generated by \(\d' + \r\). By the uniqueness of the limit, we get that~\(\mu = \mu^*\). 

In summary, to complete our proof, it suffices to prove that \(R^n \to \delta_R\) in \(\fP (\fP (\Upsilon))\) when \(\Upsilon\) is endowed with the topology induced by \(\d^* + \r\), and that these measures are really supported on \(\fP (\fP (\Upsilon))\). This is program of our second (and final) step.

\smallskip
\noindent
{\em Step 2.}
In the following, we tailor a coupling idea from \cite{LakLN,SM} to our setting.
Notice that \(Q^n\)-a.s. \(X^k, Y^k \in L^\alpha ([0, T]; \Y)\) and that \(Q^n\)-a.s.
\begin{align*}
	&\int_0^T \int \| b (f, s, Y^k_s, \x_n (Y_s)) \|^{\alpha / (\alpha - 1)}_{\Y^*} \, \v (s, M^k, df) \, ds < \infty, \\ 
	&\int_0^T \int \| b (f, s, X^k_s, R_s) \|^{\alpha / (\alpha - 1)}_{\Y^*} \, \v (s, M^k, df) \, ds < \infty,
\end{align*}
by Condition~\ref{cond: main2}~(ii) and the integrability properties of \(X^k\) and \(Y^k\).
Since \(\H \subset \X\) continuously, we also have~\(Q^n\)-a.s.
\begin{align*}
	&\int_0^T\| \us (\v (s, M^k), s, Y^k_s, \x_n (Y_s)) \|^{2}_{L_2 (\U; \X)}  \, ds < \infty, \\ 
	&\int_0^T  \| \us (\v (s, M^k), s, X^k_s, R_s) \|^{2}_{L_2(\U; \X)} \, ds < \infty,
\end{align*}
which shows that the It\^o integrals in the dynamics of \(X^k\) and \(Y^k\) are \(\X\)-valued continuous (local) martingales. From these observations and recalling the Gelfand triple \(\Y\subset \X \,\cong \,\X^* \subset \Y^*\), we conclude that \(Q^n\)-a.s. \(X^k\) and \(Y^k\) have \(\Y^*\)-valued continuous paths.
Now, we deduce from the Krylov--Rozovskii It\^o formula (see Lemma~\ref{appendix: Krylov} in Appendix~\ref{sec: appendix}) that \((X^k, M^k)\) and \((Y^k, M^k)\) have \(Q^n\)-a.s. paths in \(\Upsilon\) and, using also Condition~\ref{cond: main2}~(ii), \(Q^n\)-a.s. 
\begin{equation*} \begin{split}
	\|Y^{k}_t &- X^k_t\|^2_{\X} - \|x^n - x^0\|_{\X}^2
	\\ &\ = \int_0^t 2 \langle Y^{k}_s - X^k_s, (\us (\v (s, M^k), s, Y^{k}_s, \x_n (Y_s)) - \us (\v (s, M^k), s, X^k_s, R_s) ) \, d W^k_s \rangle_{\X}
	\\&\ \qquad+ \int_0^t 2 \mathop{\vphantom{\int}}\nolimits_{\Y}\hspace{-0.1cm} \Big \langle Y^{k}_s - X^k_s, \int \big( b (f, s, Y^{k}_s, \x_n (Y_s)) - b (f, s, X^k_s, R_s) \big) \, \v (s, M^k, df) \Big \rangle_{\Y^*} \, ds
	\\&\ \qquad  + \int_0^t \| \us (\v (s, M^k),s, Y^{k}_s, \x_n (Y_s)) - \us (\v (s, M^k),s, X^k_s, R_s)\|^2_{L_2 (\U; \X)} \, ds 
	\\&\ \leq  \int_0^t 2 \langle Y^{k}_s - X^k_s, (\us (\v (s, M^k), s, Y^{k}_s, \x_n (Y_s)) - \us (\v (s, M^k), s, X^k_s, R_s) ) \, d W^k_s \rangle_{\X}
	\\&\ \qquad + \int_0^t C \Big( \|Y^{k}_s - X^k_s\|^2_{\X} + \w^{\X}_2 (\x_n (Y_s), R_s)^2 \Big) \,  ds.
\end{split} \end{equation*}
Using Burkholder's inequality, again Condition~\ref{cond: main2} and Young's inequality for products, we obtain that 
\begin{align*}
	E^{Q^n} \Big[ \sup_{s \in [0, t]} & \Big| \int_0^s 2 \langle Y^{k}_u - X^k_u, (\us (\v (u, M^k), u, Y^{k}_u, \x_n (Y_u)) - \us (\v (u, M^k), u, X^k_u, R_u) ) \, d W^k_u \rangle_{\X} \Big| \, \Big]
	\\&\leq CE^{Q^n} \Big[ \Big( \int_0^t \|Y^{k}_u - X^k_u\|_{\X}^2 \,\big[ \|Y^{k}_u - X^k_u\|^2_{\X} + \w^{\X}_2 (\x_n (Y_u), R_u)^2\, \big] \, du \Big)^{1/2}\, \Big] 
	\\&\leq CE^{Q^n} \Big[ \sup_{s \in [0, t]} \|Y^{k}_s - X^k_s\|_{\X}  \Big( \int_0^t \big[ \|Y^{k}_u - X^k_u\|^2_{\X} + \w^{\X}_2 (\x_n (Y_u), R_u)^2\, \big] \, du \Big)^{1/2}\, \Big] 
	\\&\leq E^{Q^n} \Big[\, \frac{1}{2} \sup_{s \in [0, t]} \|Y^{k}_s - X^k_s\|_{\X}^2 + \frac{C}{2}\, \int_0^t \big[ \|Y^{k}_u - X^k_u\|^2_{\X} + \w^{\X}_2 (\x_n (Y_u), R_u)^2\, \big] \, du \, \Big].
\end{align*} 
In summary, we conclude that 
\begin{align*}
	E^{Q^n} \Big[ \sup_{s \in [0, t]} \|Y^{k}_s - X^k_s\|^2_{\X} \Big] \leq 2 \|x^n - x^0\|^2_{\X} + C E^{Q^n} \Big[ \int_0^t \big[ \|Y^{k}_u - X^k_u\|^2_{\X} + \w^{\X}_2 (\x_n (Y_u), R_u)^2 \big] \, du \, \Big].
\end{align*} 
Gronwall's lemma yields that 
\begin{align} \label{eq: Gron 1}
	E^{Q^n} \Big[ \sup_{s \in [0, t]} \|Y^{k}_s - X^k_s\|^2_{\X} \Big] \leq C \Big( \|x^n - x^0\|_{\X}^2 + \int_0^t E^{Q^n} \Big[\w_2^{\X} (\x_n (Y_s), R_s)^2 \Big] ds \Big).
\end{align}
For \(\mu, \nu \in \fP (\Omega)\), define 
\[
\w_t (\mu, \nu)^2:= \inf_{\pi \in \Pi (\mu, \nu)} \iint \sup_{s \in [0, t]} \| \omega (s) - \alpha (s)\|^2_{\X}\, \pi (d\omega, d \alpha), 
\]
where \(\Pi (\mu, \nu)\) denotes the set of all couplings of \(\mu\) and \(\nu\). 
Using the coupling \(\frac{1}{n} \sum_{k = 1}^n \delta_{(Y^{k}, X^k)}\), we observe that 
\begin{align} \label{eq: coupl 1}
	\w_t(\x_n (Y), \x_n (X))^2 \leq \frac{1}{n} \sum_{k= 1}^n \sup_{s \in [0, t]} \|Y^{k}_s - X^k_s\|^2_{\X}.
\end{align}
Hence, using the triangle inequality, \eqref{eq: Gron 1} and \eqref{eq: coupl 1}, it follows that
\begin{align*}
	E^{Q^n} \Big[ & \w_t  (\x_n (Y), R \circ X^{-1})^2 \Big] 
	\\&\leq 2 E^{Q^n} \Big[ \w_t (\x_n (Y), \x_n (X))^2 \Big] + 2 E^{Q^n} \Big[ \w_t (\x_n (X), R \circ X^{-1})^2 \Big]
	\\&\leq  2 E^{Q^n} \Big[ \frac{1}{n} \sum_{k = 1}^n \sup_{s \in [0, t]} \|Y^{k}_s - X^k_s\|^2_\X \Big] + 2 E^{Q^n} \Big[ \w_t (\x_n (X), R \circ X^{-1})^2 \Big]
	\\&\leq C \Big( \|x^n - x^0\|_{\X}^2 + \int_0^t E^{Q^n} \Big[\w_2^{\X} (\x_n (Y_s), R_s)^2 \Big] \, ds + E^{Q^n} \Big[ \w_t (\x_n (X), R \circ X^{-1})^2 \Big] \Big)
	\\&\leq C \Big( \|x^n - x^0\|_{\X}^2 + \int_0^t  E^{Q^n} \Big[\w_s (\x_n (Y), R \circ X^{-1})^2 \Big] \, ds + E^{Q^n} \Big[ \w_t (\x_n (X), R \circ X^{-1})^2 \Big] \Big).
\end{align*}
Using Gronwall's lemma once again, we get that 
\begin{align} \label{eq: main1}
	E^{Q^n} \Big[ \w_T (\x_n (Y), R \circ X^{-1})^2 \Big] \leq C \Big( \|x^n - x^0\|_{\X}^2+ E^{Q^n} \Big[ \w_T (\x_n (X), R \circ X^{-1})^2 \Big] \Big).
\end{align}
Under \(Q^n\), the processes \(X^1, X^2, \dots, X^n\) are independent with law \(R \circ X^{-1}\). Hence, as also
\[
E^R \Big[ \sup_{s \in [0, T]} \|X_s\|^2_{\X} \Big] \leq C E^R \Big[ \sup_{s \in [0, T]} \|X_s\|^2_{\H} \Big] < \infty, 
\]
by the definition of \(\NN\), it follows from Lemma~\ref{appendix: law large numbers} in Appendix~\ref{sec: appendix} that 
\begin{align} \label{eq: iid conv}
E^{Q^n} \Big[ \w_T (\x_n (X), R \circ X^{-1})^2 \Big] \to 0 \text{ as } n \to \infty.
\end{align}
In summary, we conclude from \eqref{eq: Gron 1}, \eqref{eq: main1} and \eqref{eq: iid conv} that 
\begin{align} \label{eq: final conv}
\frac{1}{n} \sum_{k = 1}^n E^{Q_n} \Big[ \sup_{s \in [0, T]} \|Y^k_s - X^k_s\|^2_\X \Big] \to 0 \text{ as } n \to \infty.
\end{align}
Endow \(\Upsilon\) with the metric \(\d^* + \r\). Using the triangle inequality, the coupling \(\frac{1}{n} \sum_{k = 1}^n\) \(\delta_{(Y^k, M^k), (X^k, M^k)}\), \eqref{eq: final conv} and again Lemma~\ref{appendix: law large numbers} in Appendix~\ref{sec: appendix}, 
we get that 
\begin{align*}
	E^{Q_n} \Big[ & \w_2^{\Upsilon} (\Z_n (Y, M), R)^2 \Big] 
	\\&\leq 2\, \Big( E^{Q_n} \Big[ \w_2^{\Upsilon} (\Z_n (Y, M), \Z_n (X, M))^2 \Big] + E^{Q_n} \Big[ \w_2^{\Upsilon} (\Z_n (X, M), R)^2 \Big] \Big)
	\\&\leq 2\, \Big( \frac{1}{n} \sum_{k = 1}^n E^{Q_n} \Big[ \sup_{s \in [0, T]} \|Y^k_s - X^k_s\|^2_\X \Big] + E^{Q_n} \Big[ \w_2^{\Upsilon} (\Z_n (X, M), R)^2 \Big] \Big)
	\to 0 \text{ as } n \to \infty.
\end{align*}
As Wasserstein convergence implies weak convergence for separable metrizable underlying spaces (see \cite[Theorem~11.8.2]{dudley}), it follows that \(R^n \to \delta_R\) in \(\fP (\fP (\Upsilon))\). This completes the proof.
\end{proof}
		\subsection{Proof of Theorem \ref{theo: main1} (v)}
		We use the notation from Theorem \ref{theo: main1} (v). 
Take an arbitrary measure \(Q^0 \in \cR^0 (x^0)\). Then, by Theorem \ref{theo: main1} (iv), there exists a sequence \(Q^{n} \in \cR^{n} (x^{n})\) such that \(Q^{n} \to Q^0\) in \(\fP^\al (\fP^\al (\Theta))\). By the assumptions on \(\psi\), and Lemma~\ref{lem: gen Fatou W} in Appendix~\ref{sec: appendix}, 
\begin{align*}
	E^{Q^0} \big[ \psi \big] \leq \liminf_{n \to \infty} E^{Q^{n}} \big[ \psi \big] 
	\leq \liminf_{n \to \infty} \sup_{Q\in \cR^{n} (x^{n})} E^{Q} \big[ \psi \big].
\end{align*}
As \(Q^0\) was arbitrary, we get that
\[
\sup_{Q \in \cR^0 (x^0)} E^Q \big[ \psi \big] \leq \liminf_{n \to \infty} \sup_{Q \in \cR^n (x^n)} E^Q \big[ \psi \big].
\]
The proof is complete.
\qed
	
		\subsection{Proof of Theorem \ref{theo: main1} (vi)}
		By Theorem \ref{theo: main1} (iii) and (v), for every sequence \((x^n)_{n = 0}^\infty \subset \H\) with \(x^n \to x^0\) in \(\|\cdot\|_\H\), we get that
		\[
		\sup_{Q \in \cR^n (x^n)} E^Q \big[ \psi \big] \to \sup_{Q \in \cR^0 (x^0)} E^Q \big[ \psi \big], \quad n \to \infty.
		\]
		Now, by Lemma~\ref{appendix: remmert} in Appendix~\ref{sec: appendix}, the map \(x \mapsto \sup_{Q \in \cR^0 (x)} E^Q [ \psi ]\) is continuous and \eqref{eq: compact convergence} holds.
		The proof is complete. \qed 
		\subsection{Proof of Theorem \ref{theo: main1} (vii)}
	To keep our notation simple, we write \(\w := \w_{\al}^{\fP^\al (\Theta)}\). 
		By Lemma~\ref{appendix: remmert} in Appendix~\ref{sec: appendix}, it suffices to prove that 
		\begin{align*}
		\mathsf{h} (\cR^n (x^n), \cR^0 (x^0)) =	\max \Big\{ \max_{Q \in \cR^n (x^n)} \w (Q, \cR^0 (x^0)) , \max_{Q \in \cR^0 (x^0)} \w (Q, \cR^n (x^n))\Big\} \to 0
		\end{align*}
	for all sequences \((x^n)_{n = 0}^\infty \subset \H\) with \(x^n \to x^0\) in \(\|\cdot \|_\H\).
	Notice that the maxima are attained by the compactness of the sets \(\cR^n (x^n)\) and \(\cR^0 (x^0)\) in the space \(\mathcal{P}^\varrho (\mathcal{P}^\varrho (\Theta))\).
	
	We start investigating the first term. 	
	By the compactness of \(\cR^n (x^n)\), for every \(n  \in \mathbb{N}\), there exists a measure \(Q^n \in \cR^n (x^n)\) such that 
	\[
	\max_{Q \in \cR^n (x^n)} \w (Q, \cR^0 (x^0)) = \w (Q^n, \cR^0 (x^0)).
	\]
	By Theorem~\ref{theo: main1}~(ii), every subsequence of \(1, 2, \dots\) has a further subsequence \((N_n)_{n = 1}^\infty\) such that \((Q^{N_n})_{n = 1}^\infty\) converges in \(\fP^\al(\fP^\al (\Theta))\) to a measure \(Q^0 \in \cR^0 (x^0)\). By the continuity of the distance function, we have 
	\[
	\w (Q^{N_n}, \cR^0 (x^0)) \to \w (Q^0, \cR^0 (x^0)) = 0.
	\]
	We conclude that
	\[
		\max_{Q \in \cR^n (x^n)} \w(Q, \cR^0 (x^0)) = \w (Q^n, \cR^0 (x^0)) \to 0 \text{ as } n \to \infty.
	\]
	
	We turn to the second term. By the compactness of \(\cR^0 (x^0)\), for every \(n \in \mathbb{N}\), there exists a measure \(R^n \in \cR^0 (x^0)\) such that 
	\[
	\max_{Q \in \cR^0 (x^0)} \w (Q, \cR^n (x^n)) = \w (R^n, \cR^n (x^n)).
	\]
	Let \((N^n_1)_{n = 1}^\infty\) be an arbitrary subsequence of \(1, 2, \dots\). Again by compactness of \(\cR^0 (x^0)\), there exists a subsequence \((N^n_2)_{n = 1}^\infty \subset (N^n_1)_{n = 1}^\infty\) such that \((R^{N^n_2})_{n = 1}^\infty\) converges in \(\fP^\al (\fP^\al (\Theta))\) to a measure \(R^0 \in \cR^0 (x^0)\). 
	By Theorem~\ref{theo: main1}~(iv), there exists a sequence \((Q^{N^n_2})_{n = 1}^\infty\) such that \(Q^{N^n_2} \in \cR^{N^n_2}(x^{N^n_2})\) and \(Q^{N^n_2} \to R^0\) in \(\fP^\al (\fP^\al (\Theta))\). Finally, 
	\[
	 \w (R^{N^n_2}, \cR^{N^n_2} (x^{N^n_2})) \leq \w (R^{N^n_2}, Q^{N^n_2}) \leq  \w (R^{N^n_2}, R^0) + \w(R^0, Q^{N^n_2}) \to 0.
	\]
	As \((N^n_1)_{n = 1}^\infty\) was arbitrary, this proves that 
	\[
	 \w (R^n, \cR^n (x^n)) \to 0.
	\]
	In summary, \(\cR^n (x^n) \to \cR^0 (x^0)\) in the Hausdorff metric topology. The proof is complete. \qed

\appendix

\section{A collection of auxiliary results} \label{sec: appendix}
To keep this paper as self-contained as possible and to improve its readability, this appendix collects a variety of auxiliary results from the literature that are used in our proofs. 

\subsection{Lusin spaces}
A Hausdorff space \(\Lambda\) is called Lusin if there exists a Polish space \(\Sigma\) and a continuous bijection from \(\Sigma\) to \(\Lambda\); equivalently, there exists a stronger topology on \(\Lambda\) that turns it into a Polish space. For the following facts, see \cite[p. 96, and Theorem~2 on p.~95]{schwarz}.
\begin{lemma} 
	Every Lusin space is separable and every Borel subset of a Lusin space is again a Lusin space.
\end{lemma}   
The next result follows from \cite[Corollary~2, p.~101]{schwarz}. 
\begin{lemma} \label{appendix: Borel sigma fields}
	The Borel \(\sigma\)-fields on a set coincide for all comparable Lusin topologies. 
\end{lemma}

The following useful fact is a restatement of \cite[Theorem~5, p. 101]{schwarz}. 

\begin{lemma} \label{appendix: lusin Borel}
	Let \(\Lambda\) be a Hausdorff space. Every Lusin subspace (with the induced topology) is Borel in \(\Lambda\). 
\end{lemma} 

The next lemma follows from \cite[Theorem~7, p.~385]{schwarz}. 
\begin{lemma} \label{appendix: Suslin prob meas} 
	For a Lusin (resp., Polish) space \(\Lambda\), the space \(\mathcal{P}(\Lambda)\) with the weak topology is also Lusin (resp., Polish).
\end{lemma}

\subsection{Measure theory}

The following is a restatement of \cite[Theorem~A.3.12]{DE_97}.
\begin{lemma} \label{appendix: lower semi} 
Let \(\Lambda\) be a Polish space and let \(g \colon \Lambda \to [0, \infty]\) be lower semicontinuous. Then, the map \(\mu \mapsto \int g\,  d \mu\) from \(\mathcal{P} (\Lambda)\) into \([0, \infty]\) is also lower semicontinuous. 
\end{lemma} 

The following is a partial restatement of \cite[Theorem~8.10.61]{bogachev}.

\begin{lemma} \label{appendix: bogachev}
	Let \(\Lambda\) and \(\Gamma\) be Polish spaces and let \(g \colon \Lambda \times \Gamma \to \bR\) be a bounded function. If \(g\) is Borel measurable (resp. upper semicontinuous, resp. continuous), then the same is true for
	\[
	\Lambda \times \mathcal{P} (\Gamma) \to \bR, \quad (\lambda, \mu) \mapsto \int g (\lambda, \gamma) \, \mu (d \gamma).
	\]
\end{lemma} 

The following lemma is a consequence of the equivalence (1) \(\Leftrightarrow\) (3) from \cite[Proposition~A.1]{LakSPA15}.
\begin{lemma} \label{appendix: wasserstein cont}
	Let \((\Lambda, m_\Lambda)\) and \((\Gamma, m_\Gamma)\) be complete separable metric spaces and \(p \in [1, \infty) \cup \{0\}\). If \(g \colon \Gamma \to \mathcal{P}^p (\Lambda)\) is a continuous map and there are \(\lambda_0 \in \Lambda, \gamma_0 \in \Gamma\) and \(C > 0\) such that
	\[
	\int m_\Lambda (\lambda, \lambda_0)^p \, g (\gamma) (d \lambda) \leq C (1 + m_\Gamma(\gamma, \gamma_0)^p), \quad \forall \, \gamma \in \Gamma, 
	\]
	then \(P \mapsto P \circ g^{-1}\) is continuous from \(\mathcal{P}^p (\Gamma)\) into \(\mathcal{P}^p (\mathcal{P}^p (\Lambda))\). 
\end{lemma} 

The following version of the law of large numbers is a restatement of \cite[Corollary~2.14]{LakLN}.

\begin{lemma} \label{appendix: law large numbers}
	Let \(\Lambda\) be a separable metrizable space and take \(\mu \in \mathcal{P}^p (\Lambda)\). Let \((X^n)_{n = 1}^\infty\) be a sequence of independent \(\Lambda\)-valued random variables all with law \(\mu\) and set \(\mu^n := \frac{1}{n} \sum_{k = 1}^n \delta_{X^k}\). Then, \[\w^\Lambda_p (\mu^n, \mu) \to 0, \quad n \to \infty, \] a.s. and in \(L^p\).
\end{lemma} 

The following elementary fact appears to be well-known (\cite[p. 1657]{LakSIAM17}) but we did not find a reference, so we provide a proof.
\begin{lemma} \label{lem: gen Fatou W}
	Let \((\Lambda, m_\Lambda)\) be a complete separable metric space and take \(p \geq 1\). Assume that \(g \colon \Lambda \to \bR\) is an upper semicontinuous function such that 
	\[
	\exists\, \lambda_0 \in \Lambda \colon \ c:= \sup \Big\{ \frac{| g (\lambda) |}{1 + m_\Lambda (\lambda, \lambda_0)^p} \colon \lambda \in \Lambda \Big\} < \infty.
	\]
	Let \((\mu^n)_{n = 0}^\infty \subset \fP^p (\Lambda)\) be a sequence such that \(\mu^n \to \mu^0\) in \(\fP^p (\Lambda)\). Then, 
	\[
	\limsup_{n \to \infty} E^{\mu^n} [ g ] \leq E^{\mu^0} [ g ].
	\]
\end{lemma}
\begin{proof}
	By Skorokhod's coupling theorem, on some probability space (whose expectation we denote by \(E\)), there are random variables \(X^0, X^1, \dots\) with laws \(\mu^0, \mu^1, \dots\) such that a.s. \(X^n \to X^0\). Notice that 
	\[
	Y^n := c \big[ 1 + m_\Lambda(X^n, \lambda_0)^p \big] - g (X^n) \geq 0. 
	\]
	Hence, Fatou's lemma yields that 
	\begin{align*}
		E \Big[ \liminf_{n \to \infty} Y^n \Big] \leq \liminf_{n \to \infty} E [ Y^n ].
	\end{align*}
	As \(g\) is upper semicontinuous, we have a.s.
	\[
	\liminf_{n \to \infty} Y^n = c \big[ 1 + m_\Lambda(X^0, \lambda_0)^p \big] - \limsup_{n \to \infty} g (X^n) \geq\ c\big[ 1 +m_\Lambda (X^0, \lambda_0)^p \big] - g (X^0),
	\]
	which implies 
	\[
	cE\big[ 1 + m_\Lambda(X^0, \lambda_0)^p \big] - E [ g (X^0) ] \leq \liminf_{n \to \infty} E [ Y^n ].
	\]
	As \(\mu^n \to \mu^0\) in \(\fP^p (\Lambda)\), we get that
	\[
	\liminf_{n \to \infty} E [ Y^n ] = cE \big[ 1 + m_\Lambda (X^0, \lambda_0)^p \big] - \limsup_{n \to \infty} E [ g (X^n) ],
	\]
	and finally, 
	\[
	- E [ g (X^0) ] \leq - \limsup_{n \to \infty} E [ g (X^n) ]. \qedhere
	\]
\end{proof}

The following consequence of the Portmanteau theorem is a restatement of \cite[Corollary~3.3.2]{EK}. 

\begin{lemma} \label{appendix: subspace weak conv}
	Let \(\Lambda\) be a metrizable space and \(\Lambda' \in \mathcal{B} (\Lambda)\). For any sequence \((\mu^n)_{n = 0}^\infty \subset \mathcal{P} (\Lambda)\) with \(\mu^n (\Lambda') = 1\), we have \(\mu^n \to \mu^0\) in \(\mathcal{P}(\Lambda)\) if and only if \(\mu^n \to \mu^0\) in \(\mathcal{P} (\Lambda')\).
\end{lemma} 

Finally, the next lemma is a consequence of \cite[Theorem~2.1, Examples~2.1~(a)]{winkler}.

\begin{lemma} \label{appendix: extreme points}
	Let \(\Lambda\) be a Lusin space, let \(g \colon \Lambda \to \bR\) be a Borel map and take \(c \in \bR\). Then,
	\[
	H := \Big\{
	\mu \in \mathcal{P} (\Lambda) \colon \int |g| \, d\mu < \infty, \ \int g \, d\mu \leq c\, \Big\}
	\]
	is a convex set and its extreme points \(\on{ex}\, (H)\) satisfy
	\[
	\on{ex}\, (H) \subset \Big\{ \mu \in \mathcal{P} (\Lambda) \colon \mu = \alpha \delta_{\lambda_1} + (1 - \alpha) \delta_{\lambda_2}, \, \alpha \in [0, 1], \, \lambda_1, \lambda_2 \in \Lambda \Big\}.
	\]
\end{lemma} 

\subsection{Weak-strong convergence on product spaces} \label{sec: app ws conv}
Let \(\Lambda\) and \(\Sigma\) be Polish spaces and define \(\mathcal{P}_{ws} (\Lambda \times \Sigma)\) to be the space of probability measures on \((\Lambda \times \Sigma, \mathcal{B} (\Lambda \times \Sigma))\) endowed with the weakest topology under which \(\mu \mapsto \int g \, d\mu\) is continuous for every bounded Borel function \(g \colon \Lambda \times \Sigma \to \bR\) that is continuous in the \(\Sigma\)-variable. If \(\mu^n \to \mu^0\) in \(\mathcal{P}_{ws}(\Lambda \times \Sigma)\), we say that \((\mu^n)_{n = 1}^\infty\) converges in the weak-strong sense to \(\mu^0\). 

We write \(\mathcal{P}_s (\Lambda)\) for the space \(\mathcal{P} (\Lambda)\) endowed with the topology of set-wise convergence, i.e., the weakest topology under which \(\mu \to \int g \, d \mu\) is continuous for every bounded Borel function \(g \colon \Lambda \to \bR\). 
The following is a restatement of \cite[Corollary~2.9]{SPS_1981__15__529_0}.

\begin{lemma} \label{appendix: convergence in ws} 
	Take \((\mu_n)_{n = 0}^\infty \subset \mathcal{P}_{ws} (\Lambda \times \Sigma)\). Then, \(\mu^n \to \mu^0\) in \(\mathcal{P}_{ws} (\Lambda \times \Sigma)\) if and only if \(\mu^n \to \mu^0\) in \(\mathcal{P} (\Lambda \times \Sigma)\) and \(\{ \mu^n (\, \cdot \times \Sigma)  \colon n \in \mathbb{N}\}\) is relatively compact in \(\mathcal{P}_s (\Lambda)\). 
\end{lemma} 

The following is a restatement of \cite[Theorem~2.5]{CPS23}.
\begin{lemma} \label{appendix: weak strong relative comp}
	Let \(\Pi \subset \mathcal{P}_{ws} (\Lambda \times \Sigma)\). The following are equivalent: 
	\begin{enumerate}
		\item[\textup{(a)}] 
		\(\Pi\) is relatively compact in \(\mathcal{P}_{ws} (\Lambda \times \Sigma)\).
		\item[\textup{(b)}] 
		\(\Pi_\Lambda:= \{ \mu (\, \cdot \times \Sigma) \colon \mu \in \Pi\}\) is relatively compact in \(\mathcal{P}_s (\Lambda)\) and \(\Pi_\Sigma := \{ \mu (\Lambda \times \cdot \,) \colon \mu \in \Pi\}\) is relatively compact in \(\mathcal{P}(\Sigma)\). 
		\item[\textup{(c)}] 
		\(\Pi\) is relatively sequentially compact in \(\mathcal{P}_{ws} (\Lambda \times \Sigma)\).
		\item[\textup{(d)}] 
		\(\Pi_\Lambda\) is relatively sequentially compact in \(\mathcal{P}_s (\Lambda)\) and \(\Pi_\Sigma\) is relatively sequentially compact in \(\mathcal{P}(\Sigma)\). 
	\end{enumerate}
\end{lemma} 

With \(\m\) as defined in Section~\ref{sec: notation}, Lemma~\ref{appendix: convergence in ws} shows that on the space \(\m\) the weak and the weak-strong topologies coincide:

\begin{lemma} \label{appendix: stable}
	Let \(g \colon [0, T] \times F \to \bR\) be a bounded Carath\'eodory function, i.e., measurable in the first and continuous in the second coordinate. Then, \(m \mapsto \int g \, dm\) is continuous from \(\m\) into~\(\bR\). 
\end{lemma}

\subsection{Relative compactness criteria}

We recall a relative compactness criterion for subsets of the Wasserstein space \(\mathcal{P}^p (\mathcal{P}^p (\Lambda))\), where \(\Lambda\) is a Polish space. In the spirit of \eqref{eq: I}, we set 
\[
I (P) (A) := \int \mu (A) \, P (d \mu), \quad A \in \mathcal{B} (\Lambda), \ P \in \mathcal{P}^p (\mathcal{P}^p (\Lambda)). 
\]
The following is a restatement of \cite[Corollary~B.2]{LakSPA15}.

\begin{lemma} \label{appendix: tight wasserstein} 
	Let \((\Lambda, m_\Lambda)\) be complete separable metric space and take \(p \in [1, \infty) \cup \{0\}\). Let \(\mathcal{R} \subset \mathcal{P}^p (\mathcal{P}^p(\Lambda))\) be such that \(I (\mathcal{R})\) is relatively compact in \(\mathcal{P}(\Lambda)\) and 
	\[
	\sup_{P \in \mathcal{R}} \int m_\Lambda(\lambda, \lambda_0)^\ell \, I (P) (d \lambda) < \infty, \quad \text{ for some } \lambda_0 \in \Lambda \text{ and }  \ell > p, 
	\]
	then \(\mathcal{R}\) is relatively compact in \(\mathcal{P}^p (\mathcal{P}^p (\Lambda))\). 
\end{lemma} 

For the next result, let \(\H, \X, \V, \N\) and \(\Omega\) be as in Section~\ref{sec: notation}.
The following is an adaptation of \cite[Lemma~4.3]{GRZ} to our setting. 

\begin{lemma} \label{appendix: tight GRZ}
	Let \(\mathcal{R}\subset \mathcal{P} (C ([0, T]; \V))\). Assume that \(\V^*\) is compactly embedded into \(\H\) and that, for some \(\ell > 0\),
	\[
	\sup_{P \in \mathcal{R}}E^P \Big[ \sup_{s \in [0, T]} \|X_s\|_\H + \sup_{s \not = t} \frac{\|X_t - X_s\|_\V}{|t - s|^\ell} + \int_0^T \N (X_s) \, ds \Big] < \infty. 
	\]
	Then, \(\mathcal{R}\) is relatively compact in \(\mathcal{P} (\Omega)\). 
\end{lemma} 

\subsection{The Besov--H{\"o}lder embedding}

A useful estimate in the context of Lemma~\ref{appendix: tight GRZ} is the Besov--H{\"o}lder embedding, see \cite[Corollary~A.2]{FV}. 

\begin{lemma} \label{appendix: besov}
	Take a complete metric space \((\Lambda, m_\Lambda)\) and let \(q > 1, \alpha \in (1 / q, 1)\) and \(\omega \in C ([0, T]; \Lambda)\). Then, there exists a constant \(C = C (q, \alpha)\) such that 
	\begin{align*}
		\sup_{s \neq t} \frac{m_\Lambda (\omega (s), \omega(t))^q}{|s - t|^{q\alpha - 1}} \leq C \int_0^T \int_0^T \frac{m_\Lambda (\omega (s), \omega(t))^q}{|s - t|^{1 + q \alpha}} \, ds\, dt.
	\end{align*}
\end{lemma} 

\subsection{The Arzel\`a--Ascoli theorem}
Let \((\Lambda, m_\Lambda)\) and \((\Sigma, m_\Sigma)\) be two complete separable metric spaces, \(\Lambda\) being locally compact. For \(K \subset \Lambda\), a function \(g \in C (\Lambda; \Sigma)\) and \(h > 0\), we set 
\[
w_K (g, h) := \sup \Big\{ m_\Sigma (g (s), g (t)) \colon s, t \in K, \, m_\Lambda(s, t) \leq h \Big\}. 
\] 
The following version of the Arzel\`a--Ascoli theorem is a restatement of \cite[Theorem~A.5.2]{Kal21}.
\begin{lemma} \label{appendix: arzela}
	Let \(A \subset C (\Lambda; \Sigma)\), where latter is endowed with the local uniform topology, and let \(\Lambda'\) be a dense subset of \(\Lambda\). Then, \(A\) is relatively compact in \(C (\Lambda; \Sigma)\) if and only if the following two properties hold:
	\begin{enumerate}
		\item[\textup{(i)}] \(\{ g (t) \colon g \in A \}\) is relatively compact in \(\Sigma\) for every \(t \in \Lambda'\).
		\item[\textup{(ii)}] \(\lim_{h \to 0} \sup_{g \in A} w_K (g, h) = 0\) for every compact set \(K \subset \Lambda\). 
	\end{enumerate}
\end{lemma}

\subsection{Compact and continuous convergence}

The following lemma can be proved similar to \cite[Theorem on pp. 98--99]{remmert}.

\begin{lemma} \label{appendix: remmert} 
Let \(\Lambda\) and \(\Sigma\) be metrizable spaces and let \((f^n)_{n = 1}^\infty\) be a sequence of functions \(\Lambda \to \Sigma\). The following are equivalent:
\begin{enumerate} 
	\item[\textup{(a)}] \(f^n\) converges uniformly on compact subsets of \(\Lambda\) to a function \(f \in C (\Lambda; \Sigma)\).
	\item[\textup{(b)}] 
	For every sequence \((\lambda^n)_{n = 1}^\infty \subset \Lambda\), the limit \(\lim_{n \to \infty} f^n (\lambda^n)\) exists in \(\Sigma\). 
\end{enumerate} 
\end{lemma} 

\subsection{Berge's maximum theorem}

We recall a useful elementary version of Berge's maximum theorem, see \cite[Propositions~1.3.1, 1.3.3]{hu_papa}.

\begin{lemma} \label{appendix: berge}
	Let \(\Lambda\) and \(\Gamma\) be Hausdorff spaces and suppose that \(\Gamma\) is compact. If \(g \colon \Lambda \times \Gamma \to [- \infty, \infty]\) is lower (resp., upper) semicontinuous, then 
	\(
	\lambda \mapsto \sup_{\gamma \in \Gamma} g (\lambda, \gamma)
	\) 
	is also lower (resp., upper) semicontinuous.
\end{lemma}

\subsection{Semimartingale theory}

The following lemma is a special case of \cite[Theorem~II.2.42]{JS} for continuous processes. 
\begin{lemma} \label{appendix: JS}
	Let \(Y = (Y_t)_{t \in [0, T]}\) be a real-valued continuous process, let \(B = (B_t)_{t \in [0, T]}\) be a real-valued continuous process of finite variation with \(B_0 = 0\) and let \(C = (C_t)_{t \in [0, T]}\) be an increasing real-valued continuous process with \(C_0 = 0\). The following are equivalent:
	\begin{enumerate}
		\item[\textup{(a)}] \(Y - Y_0 - B\) is a continuous local martingale with quadratic variation process \(C\). 
		\item[\textup{(b)}] For all \(g \in C^2_b (\bR; \bR)\), the process 
		\[
		g (Y_t) - g (Y_0) - \int_0^t g' (Y_s) \, d B_s - \frac{1}{2} \int_0^t g'' (Y_s) \, d C_s, \quad t \in [0, T], 		
		\]
		is a continuous local martingale. 
	\end{enumerate}
\end{lemma}

The following is a representation result for cylindrical continuous local martingales, which follows from \cite[Remark~3, Corollary~6]{OndRep}.

\begin{lemma}  \label{appendix: O Rep} 
	Let \(\H\) and \(\U\) be two separable Hilbert spaces, let \((\Omega, \cF, (\cF_t)_{t \in [0, T]}, P)\) be a filtered probability space, let \(g \colon [0, T] \times \Omega \to L (\U; \H)\) be progressively measurable, i.e., \((t, \omega) \mapsto g (t, \omega) u\) is progressively measurable for every \(u \in \U\), and let \(D \subset \H\) be a group for the binary operation \(+\) that separates points. 
	Suppose that a.s. 
	\[
	\int_0^T \|g_s\|^2_{L_2 (\U; \H)} \, ds < \infty,
	\]
	and that \(M = \{M (h) \colon h\in D \}\) is a family of real-valued continuous local martingales with starting value \(M_0 (h) = 0\) and quadratic variation
	\[
	\langle M (h), M (h) \rangle_t = \int_0^t \|g^*_s h \|_\U^2 \, ds, \quad t \in [0, T], \ h \in D.
	\]
	Then, possibly on a standard extension of \((\Omega, \cF, (\cF_t)_{t \in [0, T]}, P)\) that supports an infinite number of independent standard Brownian motions that are, all together, independent of \(M\), there exists a cylindrical Brownian motion \(W\) over \(\U\) such that 
	\[
	M_t (h) = \Big \langle \int_0^t g_s \, d W_s, h \Big \rangle_\H, \quad t \in [0, T], \ h \in D. 
	\]
\end{lemma} 

\subsection{Stochastic differential equations with random coefficients}
In this section we recall a weak existence result for stochastic differential equations with random coefficients that was established in \cite{jacod1981weak}. For our purpose, it suffices to stick with path-continuous solutions, while \cite{jacod1981weak} also covers processes with jumps. 

Let us start by introducing the probabilistic setup. Fix two numbers \(r, d \in \mathbb{N}\), a finite time horizon \(T > 0\) and take a filtered probability space \((\Omega, \cF, (\cF_t)_{t \in [0, T]}, P)\) with right-continuous filtration that supports an \(r\)-dimensional standard Brownian motion \(W = (W_t)_{t \in [0, T]}\). Let \(\mathcal{X} = C ([0, T]; \bR^d)\) be the space of continuous functions \([0, T]\to \bR^d\) endowed with the uniform topology. The coordinate process on \(\mathcal{X}\) is denoted by \(X = (X_t)_{t \in [0, T]}\), i.e., \(X_t (x) = x(t)\) for \(x \in \mathcal{X}\) and \(t \in [0, T]\). Then, define the product space
\[
\overline{\Omega} := \Omega \times \mathcal{X}, \quad \overline{\cF} := \cF \otimes \mathcal{B} (\mathcal{X}), \quad \overline{\cF}_t := \bigcap_{s > t} \, ( \cF_s \otimes \sigma (X_r, r \in [0, s])). 
\] 
As coefficients, take \((\overline{\cF}_t)_{t \in [0, T]}\)-predictable functions \(b \colon[0, T] \times \overline{\Omega} \to \bR^d\) and \(g \colon[0, T] \times \overline{\Omega} \to \bR^{d \times r}\). 
The following existence theorem follows from \cite[Theorem~1.8, Corollary~2.20]{jacod1981weak}.
\begin{lemma} \label{appendix: existence SDE} 
	Assume that \(x \mapsto b (t, \omega, x)\) and \(x \mapsto g (t, \omega, x)\) are continuous on \(\mathcal{X}\) for all \(t \in [0, T]\) and \(\omega \in \Omega\). 
	Further, assume that there exists a constant \(C > 0\) such that 
	\[
	\|b (t, \omega, x)\|_{\bR^d} + \| g (t, \omega, x)\|_{\bR^{d \times r}} \leq C \Big( 1 + \sup_{s \in [0, t]} \|x (s)\|_{\bR^d} \Big)
	\]
	for all \((t, \omega, x) \in [0, T] \times \overline{\Omega}\). 
Then, there exists a probability measure \(\overline{P}\) on \((\overline{\Omega}, \overline{\cF}, (\overline{\cF}_t)_{t \in [0, T]})\) such that \(W\) is an \(r\)-dimensional \(\overline{P}\)-Brownian motion and \(\overline{P}\)-a.s. 
\[
d X_t = b (t, X) \, dt + g (t, X) \, d W_t.
\]
Furthermore, \(\overline{P}\) admits a factorization 
\begin{align} \label{eq: decomposition jacod memin} 
\overline{P} (d \omega, d x) = \overline{Q} (\omega, dx) \, P (d \omega), 
\end{align} 
where \(\overline{Q}\) is a transition kernel from \((\Omega, \cF)\) into \((\mathcal{X}, \mathcal{B} (\mathcal{X}))\) such that \(\overline{Q} (\, \cdot \,, G)\) is \(P\)-a.s. \(\cF_t\)-measurable for all \(G \in \mathcal{X}_t\) and \(t \in [0, T]\). 
\end{lemma}

\begin{remark} \label{appendix: rem factorization} 
	\begin{enumerate}
\item[\textup{(a)}] 	The fact that \(W\) is a \(\overline{P}\)-Brownian motion can be deduced from the factorization \eqref{eq: decomposition jacod memin}, where the measurability properties below \eqref{eq: decomposition jacod memin} are also assumed. Indeed, by \cite[Lemma~2.17]{jacod1981weak}, in this case any martingale on \((\Omega, \cF, (\cF_t)_{t \in [0, T]}, P)\) is also a martingale on \((\overline{\Omega}, \overline{\cF}, (\overline{\cF}_t)_{t \in [0, T]}, \overline{P})\). 
\item[\textup{(b)}] It was shown in the proof of \cite[Theorem~2.5]{jacod1981weak} that the factorization \eqref{eq: decomposition jacod memin} is stable under weak-strong convergence. More specifically, for a sequence \((P^n)_{n = 0}^\infty \subset \mathcal{P} (\overline{\Omega})\) such that \(P^n \to P^0\) in the weak-strong sense, if each \(P^n\) admits a factorization of the type \eqref{eq: decomposition jacod memin}, where the kernel satisfy the properties described below \eqref{eq: decomposition jacod memin}, the same is true for the weak-strong limit \(P^0\). 
\end{enumerate} 
\end{remark}
We also provide a useful stability property on the space \(\overline{\Omega}\). The following result is a consequence of \cite[Theorem~3.20]{CPS23}. 

\begin{lemma} \label{appendix: convergence martingale}
	Let \((P^n)_{n = 0}^\infty \subset \mathcal{P} (\overline{\Omega})\) be such that \(P^n \to P^0\) in the weak-strong sense. Let \(M^0  = (M_t)_{t \in [0, T]}\) be a real-valued \cadlag \((\overline{\cF}_t)_{t \in [0, T]}\)-adapted process on \(\overline{\Omega}\) such that \(x \mapsto M^0_t (\omega, x)\) is continuous for every \((t, \omega) \in [0, T]\times \Omega\). Suppose that there are \cadlag martingales \((M^n_t)_{t \in [0, T]}\) on \((\overline{\Omega}, \overline{\cF}, (\overline{\cF}_t)_{t \in [0, T]}, P^n)\) such that \(\{ M^n_t \colon t \in [0, T], n \in \mathbb{N}\}\) is uniformly integrable and 
	\[
	P^n (| M^n_t - M^0_t | \geq \varepsilon) \to 0, \quad \forall \, \varepsilon > 0, \, t \in [0, T].
	\]
	Then, \(M^0\) is a martingale on  \((\overline{\Omega}, \overline{\cF}, (\overline{\cF}_t)_{t \in [0, T]}, P^0)\).
	
\end{lemma}

\subsection{The Krylov--Rozovskii It\^o formula}

Let \(\Y\) be a separable reflexive Banach space, let \(\X\) and \(\U\) be separable Hilbert spaces and consider the Gelfand triplet
\[
 \Y \subset \X \cong\, \X^* \subset \Y^*.
\]
The next lemma is due to Krylov--Rozovskii \cite[Theorem~I.3.1]{krylov_rozovskii}; see \cite[Theorem~4.2.5]{liu_rockner} for a formulation close to the one below.

\begin{lemma} \label{appendix: Krylov}
Take \(p > 1\) and a filtered probability space that supports a cylindrical standard Brownian motion \(W\) over \(\U\), a \(\Y^*\)-valued progressively measurable process \(Y = (Y_t)_{t \in [0, T]}\) such that a.s. 
\[
\int_0^T \|Y_s\|^{p / (p - 1)}_{\Y^*} \, ds < \infty, 
\]
and a progressively measurable \(L_2 (\U; \X)\)-valued process \(Z = (Z_t)_{t \in [0, T]}\) such that a.s. 
\[
\int_0^T \|Z_s\|_{L_2 (\U; \X)}^2 \, ds < \infty. 
\]
For an initial value \(v_0 \in \X\), define the \(\Y^*\)-valued continuous process \(V = (V_t)_{t \in [0, T]}\) by 
\[
V_t := v_0 + \int_0^t Y_s \, ds + \int_0^t Z_s \, d W_s, \quad t \in [0, T], 
\]
and assume that a.s. 
\[
\int_0^T \|V_s\|^p_\Y \, ds < \infty. 
\]
Then, a.a. paths of \(V\) are in \(C ([0, T]; \X)\) and a.s. 
\[
\| V_t\|_\X^2 = \|v_0\|_\X^2 + \int_0^t \Big(2\, _{\Y}\langle V_s, Y_s \rangle_{\Y^*} + \|Z_s\|^2_{L_2 (\U; \X)} \Big) \, ds + 2 \int_0^t \langle V_s, Z_s \, d W_s \rangle_\X, \quad t \in [0, T]. 
\]
\end{lemma}

\bibliographystyle{abbrv}
\bibliography{references}
\end{document}